\documentclass[11pt,twoside]{amsart}

\usepackage{amsthm}
\usepackage{amsmath}
\usepackage{hyperref}
\usepackage{ifthen}
\usepackage{mathtools}

\usepackage{etoolbox}
\let\bbordermatrix\bordermatrix
\patchcmd{\bbordermatrix}{8.75}{4.75}{}{}
\patchcmd{\bbordermatrix}{\left(}{\left[}{}{}
\patchcmd{\bbordermatrix}{\right)}{\right]}{}{}

\usepackage{graphicx}
\usepackage{tikz}
\usetikzlibrary{calc}
\usetikzlibrary{arrows}
\usetikzlibrary{shapes}

\theoremstyle{plain}
\newtheorem{theorem}{Theorem}
\newtheorem{lemma}[theorem]{Lemma}
\newtheorem{corollary}[theorem]{Corollary}
\newtheorem{proposition}[theorem]{Proposition}

\theoremstyle{definition}

\newtheorem{example}[theorem]{Example}

\newtheorem{problem}[theorem]{Problem}

\theoremstyle{remark}

\newcommand{\set}[1]{\left\{#1\right\}}

\def\tn{\textnormal}
\def\ld{\lambda}

\def\mN{\mathbb{N}}
\def\mR{\mathbb{R}}
\def\mS{\mathbb{S}}
\def\mZ{\mathbb{Z}}

\def\A{\mathcal{A}}
\def\B{\mathcal{B}}

\def\K{\mathcal{K}}

\def\S{\mathcal{S}}
\def\T{\mathcal{T}}

\def\es{\emptyset}

\def\b{\beta}

\def\ce{\coloneqq}

\newcommand{\ignore}[1]{}

\DeclareMathOperator{\LS}{LS}

\DeclareMathOperator{\CG}{CG}

\DeclareMathOperator{\FRAC}{FRAC}

\DeclareMathOperator{\STAB}{STAB}

\DeclareMathOperator{\diag}{diag}

\DeclareMathOperator{\conv}{conv}
\DeclareMathOperator{\cone}{cone}

\DeclareMathOperator{\inte}{int}
\DeclareMathOperator{\relint}{relint}

\title[Stable Set Polytopes with Rank $|V(G)|/3$ for the $\LS_+$ Operator]{\bf Stable Set Polytopes with Rank $|V(G)|/3$ \\ for the Lov{\'a}sz--Schrijver SDP Operator}

\author{Yu Hin (Gary) Au}
\thanks{Yu Hin (Gary) Au: Corresponding author. Department of Mathematics and Statistics, University of Saskatchewan, Saskatoon, Saskatchewan, S7N 5E6 Canada. E-mail: gary.au@usask.ca}

\author{Levent Tun{\c c}el}
\thanks{Levent Tun{\c c}el: Research of this author was supported in part by an NSERC Discovery Grant. Department of Combinatorics and Optimization, Faculty of Mathematics, University of Waterloo, Waterloo, Ontario, N2L 3G1 Canada. E-mail: levent.tuncel@uwaterloo.ca}

\date{January 12, 2025, revised: \today}


\begin{document}

\begin{abstract}
We study the lift-and-project rank of the stable set polytope of graphs with respect to the Lov{\'a}sz--Schrijver SDP operator $\LS_+$ applied to the fractional stable set polytope. In particular, we show that for every positive integer $\ell$, the smallest possible graph with $\LS_+$-rank $\ell$ contains $3\ell$ vertices. This result is sharp and settles a conjecture posed by Lipt{\'a}k and the second author in 2003, as well as answers a generalization of a problem posed by Knuth in 1994. We also show that for every positive integer $\ell$ there exists a vertex-transitive graph on at most $4\ell+12$ vertices with $\LS_+$-rank at least $\ell$.
\end{abstract}

\maketitle 

\section{Introduction}\label{sec01}

In discrete optimization, a common and very successful approach for tackling a given problem is to model it as an integer program and analyze it using convex optimization techniques. More precisely, suppose we are interested in solving the integer program
\begin{equation}\label{eq101}
\max \set{c^{\top}x : x \in P \cap \set{0,1}^n},
\end{equation}
where $c\in \mR^n$ and $P \subseteq [0,1]^n$ are given. Notice that we can replace $P \cap \set{0,1}^n$ by 
\[
P_I \ce \conv\set{ P \cap \set{0,1}^n},
\]
the \emph{integer hull} of $P$, as the feasible region in ~\eqref{eq101} to obtain a convex optimization problem. However, for a general given $P$ (given as the solution set of a system of linear inequalities), it is $\mathcal{N}\mathcal{P}$-hard to efficiently obtain a description (such as a list of its facet-inducing inequalities) of $P_I$. Now if the given set $P$ is tractable (i.e.,  it admits a polynomial-time separation oracle), we could also choose to simply optimize $c^{\top}x$ over $P$ and at least obtain an approximate solution and an upper bound on the optimal value of~\eqref{eq101} in polynomial time. However, as many sets $P$ can share the same integer hull, the quality of the approximate solution obtained under this approach very much depends on whether $P$ is a ``tight'' or ``loose'' relaxation of $P_I$.

One way to systematically tighten a given relaxation is via \emph{lift-and-project methods}. While a number of operators fall under this approach, most notably those devised in~\cite{SheraliA90, LovaszS91, BalasCC93, Lasserre01, BienstockZ04, GouveiaPT10, AuT16}, in this work we  focus on the operator $\LS_+$ (also known as $N_+$ in the literature) devised by Lov{\'a}sz and Schrijver~\cite{LovaszS91}. 

Before we define $\LS_+$, we need some notation. Given a set $P \subseteq [0,1]^n$, we define the \emph{homogenized cone} of $P$ to be
\[
\cone(P) \ce \set{ \begin{bmatrix} \lambda \\ \lambda x \end{bmatrix} : \ld \geq 0, x \in P}.
\]
Notice that $\cone(P) \subseteq \mR^{n+1}$, and we will index the new coordinate by $0$. Also, given a vector $x$ (which, by default, is a column vector) and index $i$, we let $x_i$ or $[x]_i$ denote the $i$-entry in $x$. Next, we let $e_i$ be the unit vector whose $i$-entry is $1$, with all other entries being $0$. We let $\mS_+^{n}$ denote the set of $n \times n$ real symmetric positive semidefinite matrices. To express that a symmetric matrix $M \in \mR^{n \times n}$ is positive semidefinite, we may write $M \in \mS_+^n$, or alternatively use the notation $M \succeq 0$. We also let $\diag(M)$ denote the vector made up of the diagonal entries of $M$. Also, given a positive integer $n$, we let $[n] \ce \set{1,2,\ldots, n}$. 

The $\LS_+$ operator can be defined as follows. Given $P \subseteq [0,1]^n$, let
\[
\widehat{\LS}_+(P) \ce \set{ Y \in \mS_+^{n+1} : Ye_0 = \diag(Y), Ye_i, Y(e_0-e_i) \in \cone(P)~\forall i \in [n] }.
\]
Then we define
\[
\LS_+(P) \ce \set{ x \in \mR^n : \exists Y \in \widehat{\LS}_+(P), Ye_0 = \begin{bmatrix} 1 \\ x \end{bmatrix}}.
\]
Intuitively, $\LS_+$ \emph{lifts} $P$ to a set of $(n+1) \times (n+1)$ matrices and imposes some constraints in the lifted space to obtain $\widehat{\LS}_+(P)$, and then \emph{projects} it back down to $\mR^n$ to obtain the tightened relaxation $\LS_+(P)$. Then one can show that $P_I \subseteq \LS_+(P) \subseteq P$ (see, for instance,~\cite[Lemma 3]{AuT24b} for a proof).

Moreover, we can apply $\LS_+$ successively to a set $P$ to obtain yet tighter relaxations. Define $\LS_+^0(P) \ce P$,  and for every positive integer $k \geq 1$ define $\LS_+^k(P) \ce \LS_+\left( \LS_+^{k-1}(P) \right)$. Then, for every set $P$, $\LS_+$ generates a hierarchy of nested convex relaxations which satisfy
\[
P \supseteq \LS_+(P) \supseteq \LS_+^2(P) \supseteq \cdots \supseteq \LS_+^n(P) = P_I.
\]
Thus, instead of optimizing over $P$, one can optimize over the tightened relaxation $\LS_+^k(P)$ for a chosen $k$ and obtain a potentially better approximate solution. Furthermore, if $P$ is tractable and $k = O(1)$, then $\LS_+^k(P)$ is also tractable. Thus, the $\LS_+$-relaxations offer a ``generic'' polynomial-time approximation algorithm for a broad range of $0,1$ integer programs --- and as an immediate extension, many hard discrete optimization problems. We define the \emph{$\LS_+$-rank} of a set $P$ to be the smallest integer $k$ where $\LS_+^k(P) = P_I$. Since the $n$-th relaxation generated by $\LS_+$ is guaranteed to be equal to $P_I$~\cite{LovaszS91}, every $P \subseteq [0,1]^n$ has $\LS_+$-rank at most $n$.

In this manuscript, we are particularly interested in studying the $\LS_+$-relaxations for the stable set problem of graphs. Given a simple, undirected graph $G = (V(G), E(G))$, we say that a set of vertices $S \subseteq V(G)$ is a \emph{stable set} in $G$ if no two vertices in $S$ are joined by an edge in $G$. The \emph{(maximum) stable set problem}, which aims to find the stable set of the largest cardinality in a given graph, is one of the most well-studied problems in combinatorial optimization and is well-known to be $\mathcal{N}\mathcal{P}$-hard.

Given a graph $G$, we define its \emph{fractional stable set polytope} to be
\[
\FRAC(G) \ce \set{ x \in [0,1]^{V(G)} : x_i + x_j \leq 1~\forall \set{i,j} \in E(G)},
\]
and its \emph{stable set polytope} to be $\STAB(G) \ce \FRAC(G)_I$. Notice that $x \in \set{0,1}^{V(G)}$ belongs to $\FRAC(G)$ if and only if it is the incidence vector of a stable set in $G$, and that $\STAB(G)$ is precisely the convex hull of the incidence vectors of all stable sets in $G$. It is well-known that $\FRAC(G) = \STAB(G)$ if and only if the given graph is bipartite. In other cases, we can then apply $\LS_+$ to $\FRAC(G)$ to obtain a hierarchy of convex relaxations that approximate $\STAB(G)$. Furthermore, the $\LS_+$-rank of $\FRAC(G)$ (which we will simply call the $\LS_+$-rank of $G$ and denote by $r_+(G)$) gives a measure of the level of complexity of the stable set problem on $G$ in the perspective of $\LS_+$ and $\FRAC(G)$. For instance, Lov{\'a}sz and Schrijver~\cite{LovaszS91} showed that many well-known families of graphs, including perfect graphs, odd cycles, odd antiholes, and odd wheels, have $\LS_+$-rank $1$. In the last decade, there has been significant progress (see, for instance,~\cite{BianchiENT13, BianchiENT17, Wagler22, BianchiENW23}) 
 on obtaining a combinatorial characterization of graphs with $\LS_+$-rank $1$, which are commonly known as \emph{$\LS_+$-perfect graphs} in the literature.

While $\LS_+$ can compute $\STAB(G)$ in polynomial time for many graphs $G$, this also raises the natural question of which graphs give the worst-case instances for $\LS_+$. While there are easy to construct polytopes in $[0,1]^n$ which have the highest possible $\LS_+$-rank of $n$ (see, for instance,~\cite{GoemansT01, AuT18}), Lipt{\'a}k and the second author proved the following~\cite[Theorem 39]{LiptakT03}.

\begin{theorem}\label{thm101}
For every graph $G$, $r_+(G) \leq \lfloor \frac{|V(G)|}{3} \rfloor$.
\end{theorem}

Given a graph $G$, let $\alpha(G)$ denote the \emph{stability number} of a graph $G$, which is defined to be the cardinality of the largest stable set in $G$. Besides Theorem~\ref{thm101}, there are also classical upper bounds for $r_+(G)$ that depend on $\alpha(G)$. In particular, results of Lov\'{a}sz and Schrijver~\cite{LovaszS91} imply that, for every graph $G$,
\[
r_+(G)\le \alpha(G)
\qquad\text{and}\qquad
r_+(G)\le |V(G)|-\alpha(G)-1.
\]
(Note that the second bound already follows for a weaker lift-and-project operator that is equivalent to removing the positive semidefiniteness constraint from the definition of $\LS_+$.) Taken together, these two bounds imply that $r_+(G) \leq \lfloor \frac{|V(G)|-1}{2} \rfloor$. Theorem~\ref{thm101} sharpens this upper bound to $\lfloor \frac{|V(G)|}{3} \rfloor$. One of the main messages of the present work is that this bound is sharp and is attained by infinitely many graphs. 

Given a positive integer $\ell$, let $n_+(\ell)$ denote the smallest number of vertices on which there exists a graph with $\LS_+$-rank $\ell$. Then, Theorem~\ref{thm101} readily implies that $n_+(\ell) \geq 3\ell$ for every positive integer $\ell$. Thus, we say that a graph $G$ is \emph{$\ell$-minimal} if $r_+(G) = \ell$ and $|V(G)| = 3\ell$.

Well, for which $\ell$ do $\ell$-minimal graphs exist? Again, since $\FRAC(G) = \STAB(G)$ if and only if $G$ is bipartite, it is easy to see that $n_+(1) = 3$, attained by the $3$-cycle. Lipt{\'a}k and the second author~\cite{LiptakT03} showed that $G_{2,1}$ from  Figure~\ref{figKnownEG} is $2$-minimal, and went on to conjecture that $\ell$-minimal graphs exist for every positive integer $\ell$. Subsequently, Escalante, Montelar, and Nasini~\cite{EscalanteMN06} showed that there is only one other $2$-minimal graph ($G_{2,2}$ from Figure~\ref{figKnownEG}, also see~\cite[discussion following Proposition 21]{AuT24b}), as well as discovered the first known $3$-minimal graph ($G_{3,1}$ from Figure~\ref{figKnownEG}). Then, after nearly two decades of relatively little progress on this front, the authors~\cite{AuT24b} recently discovered the first known $4$-minimal graph ($G_{4,1}$ from Figure~\ref{figKnownEG}), which implies the existence of several other new $3$- and $4$-minimal graphs. 

\def\y{0.70}
\def\sc{1.7}
\def\x{180}
\def\z{360/4}

\begin{figure}[htbp]
\begin{center}
\begin{tabular}{cccc}

\begin{tikzpicture}[scale=\sc, thick,main node/.style={circle,  minimum size=1.5mm, fill=black, inner sep=0.1mm,draw,font=\tiny\sffamily}]

\node[main node] at ({cos(\x+(0)*\z)},{sin(\x+(0)*\z)}) (1) {};
\node[main node] at ({cos(\x+(1)*\z)},{sin(\x+(1)*\z)}) (2) {};
\node[main node] at ({cos(\x+(2)*\z)},{sin(\x+(2)*\z)}) (3) {};

\node[main node] at ({ \y* cos(\x+(3)*\z) + (1-\y)*cos(\x+(2)*\z)},{ \y* sin(\x+(3)*\z) + (1-\y)*sin(\x+(2)*\z)}) (4) {};
\node[main node] at ({cos(\x+(3)*\z)},{sin(\x+(3)*\z)}) (5) {};
\node[main node] at ({ \y* cos(\x+(3)*\z) + (1-\y)*cos(\x+(4)*\z)},{ \y* sin(\x+(3)*\z) + (1-\y)*sin(\x+(4)*\z)}) (6) {};

 \path[every node/.style={font=\sffamily}]
(1) edge (2)
(2) edge (3)
(3) edge (1)
(4) edge (5)
(5) edge (6)
(4) edge (3)
(6) edge (1)
(6) edge (2);
\end{tikzpicture}
&
\begin{tikzpicture}[scale=\sc, thick,main node/.style={circle,  minimum size=1.5mm, fill=black, inner sep=0.1mm,draw,font=\tiny\sffamily}]

\node[main node] at ({cos(\x+(0)*\z)},{sin(\x+(0)*\z)}) (1) {};
\node[main node] at ({cos(\x+(1)*\z)},{sin(\x+(1)*\z)}) (2) {};
\node[main node] at ({cos(\x+(2)*\z)},{sin(\x+(2)*\z)}) (3) {};

\node[main node] at ({ \y* cos(\x+(3)*\z) + (1-\y)*cos(\x+(2)*\z)},{ \y* sin(\x+(3)*\z) + (1-\y)*sin(\x+(2)*\z)}) (4) {};
\node[main node] at ({cos(\x+(3)*\z)},{sin(\x+(3)*\z)}) (5) {};
\node[main node] at ({ \y* cos(\x+(3)*\z) + (1-\y)*cos(\x+(4)*\z)},{ \y* sin(\x+(3)*\z) + (1-\y)*sin(\x+(4)*\z)}) (6) {};

 \path[every node/.style={font=\sffamily}]
(1) edge (2)
(2) edge (3)
(3) edge (1)
(4) edge (5)
(5) edge (6)
(4) edge (2)
(4) edge (3)
(6) edge (1)
(6) edge (2);
\end{tikzpicture}

&

\def\x{270 - 360/5}
\def\z{360/5}
\begin{tikzpicture}[scale=\sc, thick,main node/.style={circle,  minimum size=1.5mm, fill=black, inner sep=0.1mm,draw,font=\tiny\sffamily}]
\node[main node] at ({cos(\x+(0)*\z)},{sin(\x+(0)*\z)}) (1) {};
\node[main node] at ({cos(\x+(1)*\z)},{sin(\x+(1)*\z)}) (2) {};
\node[main node] at ({cos(\x+(2)*\z)},{sin(\x+(2)*\z)}) (3) {};

\node[main node] at ({ \y* cos(\x+(3)*\z) + (1-\y)*cos(\x+(2)*\z)},{ \y* sin(\x+(3)*\z) + (1-\y)*sin(\x+(2)*\z)}) (4) {};
\node[main node] at ({cos(\x+(3)*\z)},{sin(\x+(3)*\z)}) (5) {};
\node[main node] at ({ \y* cos(\x+(3)*\z) + (1-\y)*cos(\x+(4)*\z)},{ \y* sin(\x+(3)*\z) + (1-\y)*sin(\x+(4)*\z)}) (6) {};

\node[main node] at ({ \y* cos(\x+(4)*\z) + (1-\y)*cos(\x+(3)*\z)},{ \y* sin(\x+(4)*\z) + (1-\y)*sin(\x+(3)*\z)}) (7) {};
\node[main node] at ({cos(\x+(4)*\z)},{sin(\x+(4)*\z)}) (8) {};
\node[main node] at ({ \y* cos(\x+(4)*\z) + (1-\y)*cos(\x+(5)*\z)},{ \y* sin(\x+(4)*\z) + (1-\y)*sin(\x+(5)*\z)}) (9) {};

 \path[every node/.style={font=\sffamily}]
(1) edge (3)
(1) edge (2)
(2) edge (3)
(4) edge (5)
(5) edge (6)
(7) edge (8)
(8) edge (9)
(4) edge (2)
(6) edge (1)
(6) edge (3)
(7) edge (1)
(7) edge (2)
(9) edge (3)
(6) edge (7);
\end{tikzpicture}

&

\def\x{270 - 360/6}
\def\z{360/6}

\begin{tikzpicture}[scale=\sc, thick,main node/.style={circle,  minimum size=1.5mm, fill=black, inner sep=0.1mm,draw,font=\tiny\sffamily}]

\node[main node] at ({cos(\x+(0)*\z)},{sin(\x+(0)*\z)}) (1) {};
\node[main node] at ({cos(\x+(1)*\z)},{sin(\x+(1)*\z)}) (2) {};
\node[main node] at ({cos(\x+(2)*\z)},{sin(\x+(2)*\z)}) (3) {};

\node[main node] at ({ \y* cos(\x+(3)*\z) + (1-\y)*cos(\x+(2)*\z)},{ \y* sin(\x+(3)*\z) + (1-\y)*sin(\x+(2)*\z)}) (4) {};
\node[main node] at ({cos(\x+(3)*\z)},{sin(\x+(3)*\z)}) (5) {};
\node[main node] at ({ \y* cos(\x+(3)*\z) + (1-\y)*cos(\x+(4)*\z)},{ \y* sin(\x+(3)*\z) + (1-\y)*sin(\x+(4)*\z)}) (6) {};

\node[main node] at ({ \y* cos(\x+(4)*\z) + (1-\y)*cos(\x+(3)*\z)},{ \y* sin(\x+(4)*\z) + (1-\y)*sin(\x+(3)*\z)}) (7) {};
\node[main node] at ({cos(\x+(4)*\z)},{sin(\x+(4)*\z)}) (8) {};
\node[main node] at ({ \y* cos(\x+(4)*\z) + (1-\y)*cos(\x+(5)*\z)},{ \y* sin(\x+(4)*\z) + (1-\y)*sin(\x+(5)*\z)}) (9) {};

\node[main node] at ({ \y* cos(\x+(5)*\z) + (1-\y)*cos(\x+(4)*\z)},{ \y* sin(\x+(5)*\z) + (1-\y)*sin(\x+(4)*\z)}) (10) {};
\node[main node] at ({cos(\x+(5)*\z)},{sin(\x+(5)*\z)}) (11) {};
\node[main node] at ({ \y* cos(\x+(5)*\z) + (1-\y)*cos(\x+(6)*\z)},{ \y* sin(\x+(5)*\z) + (1-\y)*sin(\x+(6)*\z)}) (12) {};

 \path[every node/.style={font=\sffamily}]
(1) edge (2)
(1) edge (3)
(2) edge (3)
(4) edge (5)
(5) edge (6)
(7) edge (8)
(8) edge (9)
(10) edge (11)
(11) edge (12)
(4) edge (9)
(7) edge (12)
(10) edge (6)
(1) edge (4)
(2) edge (4)
(2) edge (6)
(3) edge (6)
(2) edge (7)
(3) edge (7)
(3) edge (9)
(1) edge (9)
(3) edge (10)
(1) edge (10)
(1) edge (12)
(2) edge (12);
\end{tikzpicture}

\\
$G_{2,1}$ & $G_{2,2}$ & $G_{3,1}$  & $G_{4,1}$ 
\end{tabular}
\end{center}
\caption{Several known $\ell$-minimal graphs due to~\cite{LiptakT03, EscalanteMN06, AuT24b}}\label{figKnownEG}
\end{figure}
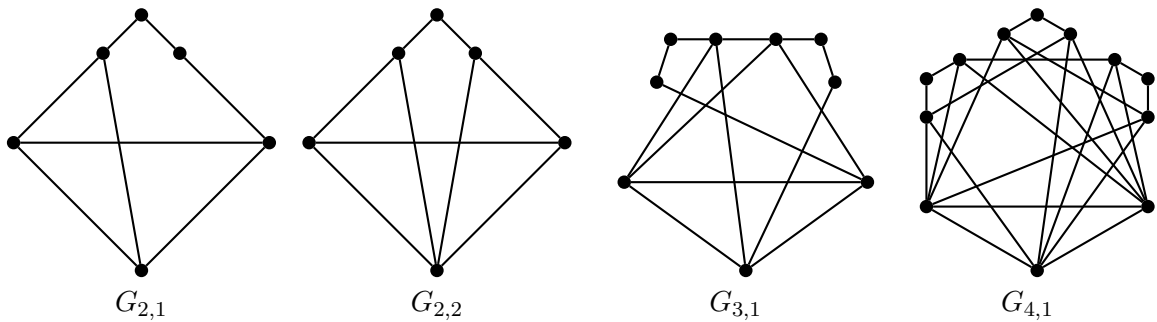

As for the asymptotic behaviour of $n_+(\ell)$, Stephen and the second author~\cite{StephenT99} showed that the line graph of the complete graph on $2\ell+1$ vertices has $\LS_+$-rank $\ell$, which implies that $n_+(\ell) \leq 2\ell^2 + \ell$ in general. Recently, the authors~\cite{AuT24} discovered a family of graphs which showed that $n_+(\ell) \leq 16\ell$ for every positive integer $\ell$, thus implying that $n_+(\ell) = \Theta(\ell)$ asymptotically.

In this work, we show that $n_+(\ell) = 3\ell$ for every positive integer $\ell$, which settles the aforementioned conjecture in~\cite{LiptakT03} affirmatively. Here is one such family of $\ell$-minimal graphs. Given an integer $k \geq 3$, define the graph $\A_{k}$ where
\begin{align*}
V(\A_{k}) \ce {}& \set{1,2,3} \cup \set{i_0, i_1, i_2  :  4\leq i \leq k},\\
E(\A_{k}) \ce {}& \set{ \set{1,2}, \set{2,3}, \set{1,3}} \cup \\
& \set{ \set{i_0,i_1}, \set{i_0,i_2}, \set{i_1,2}, \set{i_1,3}, \set{i_2,1}  :  4 \leq i \leq k} \cup\\
& \set{ \set{i_2, j_1} : 4 \leq i < j \leq k}.
\end{align*}

\def\y{0.7}
\def\sc{2.2}

\begin{figure}[htbp]
\begin{center}
\begin{tabular}{ccc}

\def\x{270 - 360/3}
\def\z{360/3}

\begin{tikzpicture}
[scale=\sc, thick,main node/.style={circle, minimum size=3.8mm, inner sep=0.1mm,draw,font=\tiny\sffamily}]
\node[main node] at ({cos(\x+(0)*\z)},{sin(\x+(0)*\z)}) (1) {$1$};
\node[main node] at ({cos(\x+(1)*\z)},{sin(\x+(1)*\z)}) (2) {$2$};
\node[main node] at ({cos(\x+(2)*\z)},{sin(\x+(2)*\z)}) (3) {$3$};

 \path[every node/.style={font=\sffamily}]
(1) edge (2)
(2) edge (3)
(1) edge (3);
\end{tikzpicture}

&
\def\x{270 - 360/4}
\def\z{360/4}

\begin{tikzpicture}
[scale=\sc, thick,main node/.style={circle, minimum size=3.8mm, inner sep=0.1mm,draw,font=\tiny\sffamily}]
\node[main node] at ({cos(\x+(0)*\z)},{sin(\x+(0)*\z)}) (1) {$1$};
\node[main node] at ({cos(\x+(1)*\z)},{sin(\x+(1)*\z)}) (2) {$2$};
\node[main node] at ({cos(\x+(2)*\z)},{sin(\x+(2)*\z)}) (3) {$3$};

\node[main node] at ({ \y* cos(\x+(3)*\z) + (1-\y)*cos(\x+(2)*\z)},{ \y* sin(\x+(3)*\z) + (1-\y)*sin(\x+(2)*\z)}) (4) {$4_1$};
\node[main node] at ({cos(\x+(3)*\z)},{sin(\x+(3)*\z)}) (5) {$4_0$};
\node[main node] at ({ \y* cos(\x+(3)*\z) + (1-\y)*cos(\x+(4)*\z)},{ \y* sin(\x+(3)*\z) + (1-\y)*sin(\x+(4)*\z)}) (6) {$4_2$};

 \path[every node/.style={font=\sffamily}]
(1) edge (2)
(2) edge (3)
(1) edge (3)
(4) edge (5)
(5) edge (6)
(4) edge (2)
(4) edge (3)
(6) edge (1);
\end{tikzpicture}

&
\def\x{270 - 360/5}
\def\z{360/5}

\begin{tikzpicture}
[scale=\sc, thick,main node/.style={circle, minimum size=3.8mm, inner sep=0.1mm,draw,font=\tiny\sffamily}]
\node[main node] at ({cos(\x+(0)*\z)},{sin(\x+(0)*\z)}) (1) {$1$};
\node[main node] at ({cos(\x+(1)*\z)},{sin(\x+(1)*\z)}) (2) {$2$};
\node[main node] at ({cos(\x+(2)*\z)},{sin(\x+(2)*\z)}) (3) {$3$};

\node[main node] at ({ \y* cos(\x+(3)*\z) + (1-\y)*cos(\x+(2)*\z)},{ \y* sin(\x+(3)*\z) + (1-\y)*sin(\x+(2)*\z)}) (4) {$4_1$};
\node[main node] at ({cos(\x+(3)*\z)},{sin(\x+(3)*\z)}) (5) {$4_0$};
\node[main node] at ({ \y* cos(\x+(3)*\z) + (1-\y)*cos(\x+(4)*\z)},{ \y* sin(\x+(3)*\z) + (1-\y)*sin(\x+(4)*\z)}) (6) {$4_2$};

\node[main node] at ({ \y* cos(\x+(4)*\z) + (1-\y)*cos(\x+(3)*\z)},{ \y* sin(\x+(4)*\z) + (1-\y)*sin(\x+(3)*\z)}) (7) {$5_1$};
\node[main node] at ({cos(\x+(4)*\z)},{sin(\x+(4)*\z)}) (8) {$5_0$};
\node[main node] at ({ \y* cos(\x+(4)*\z) + (1-\y)*cos(\x+(5)*\z)},{ \y* sin(\x+(4)*\z) + (1-\y)*sin(\x+(5)*\z)}) (9) {$5_2$};

 \path[every node/.style={font=\sffamily}]
(1) edge (2)
(2) edge (3)
(1) edge (3)
(4) edge (5)
(5) edge (6)
(7) edge (8)
(8) edge (9)
(4) edge (2)
(4) edge (3)
(6) edge (1)
(7) edge (2)
(7) edge (3)
(9) edge (1)
(6) edge (7);
\end{tikzpicture}
\\
$\A_3$ & $\A_4$ & $\A_5$\\
\\
\def\x{270 - 360/6}
\def\z{360/6}

\begin{tikzpicture}
[scale=\sc, thick,main node/.style={circle, minimum size=3.8mm, inner sep=0.1mm,draw,font=\tiny\sffamily}]
\node[main node] at ({cos(\x+(0)*\z)},{sin(\x+(0)*\z)}) (1) {$1$};
\node[main node] at ({cos(\x+(1)*\z)},{sin(\x+(1)*\z)}) (2) {$2$};
\node[main node] at ({cos(\x+(2)*\z)},{sin(\x+(2)*\z)}) (3) {$3$};

\node[main node] at ({ \y* cos(\x+(3)*\z) + (1-\y)*cos(\x+(2)*\z)},{ \y* sin(\x+(3)*\z) + (1-\y)*sin(\x+(2)*\z)}) (4) {$4_1$};
\node[main node] at ({cos(\x+(3)*\z)},{sin(\x+(3)*\z)}) (5) {$4_0$};
\node[main node] at ({ \y* cos(\x+(3)*\z) + (1-\y)*cos(\x+(4)*\z)},{ \y* sin(\x+(3)*\z) + (1-\y)*sin(\x+(4)*\z)}) (6) {$4_2$};

\node[main node] at ({ \y* cos(\x+(4)*\z) + (1-\y)*cos(\x+(3)*\z)},{ \y* sin(\x+(4)*\z) + (1-\y)*sin(\x+(3)*\z)}) (7) {$5_1$};
\node[main node] at ({cos(\x+(4)*\z)},{sin(\x+(4)*\z)}) (8) {$5_0$};
\node[main node] at ({ \y* cos(\x+(4)*\z) + (1-\y)*cos(\x+(5)*\z)},{ \y* sin(\x+(4)*\z) + (1-\y)*sin(\x+(5)*\z)}) (9) {$5_2$};

\node[main node] at ({ \y* cos(\x+(5)*\z) + (1-\y)*cos(\x+(4)*\z)},{ \y* sin(\x+(5)*\z) + (1-\y)*sin(\x+(4)*\z)}) (10) {$6_1$};
\node[main node] at ({cos(\x+(5)*\z)},{sin(\x+(5)*\z)}) (11) {$6_0$};
\node[main node] at ({ \y* cos(\x+(5)*\z) + (1-\y)*cos(\x+(6)*\z)},{ \y* sin(\x+(5)*\z) + (1-\y)*sin(\x+(6)*\z)}) (12) {$6_2$};

 \path[every node/.style={font=\sffamily}]
(1) edge (2)
(2) edge (3)
(1) edge (3)
(4) edge (5)
(5) edge (6)
(7) edge (8)
(8) edge (9)
(10) edge (11)
(11) edge (12)
(4) edge (2)
(4) edge (3)
(6) edge (1)
(7) edge (2)
(7) edge (3)
(9) edge (1)
(10) edge (2)
(10) edge (3)
(12) edge (1)
(6) edge (7)
(6) edge (10)
(9) edge (10);
\end{tikzpicture}

&

\def\x{270 - 360/7}
\def\z{360/7}

\begin{tikzpicture}
[scale=\sc, thick,main node/.style={circle, minimum size=3.8mm, inner sep=0.1mm,draw,font=\tiny\sffamily}]
\node[main node] at ({cos(\x+(0)*\z)},{sin(\x+(0)*\z)}) (1) {$1$};
\node[main node] at ({cos(\x+(1)*\z)},{sin(\x+(1)*\z)}) (2) {$2$};
\node[main node] at ({cos(\x+(2)*\z)},{sin(\x+(2)*\z)}) (3) {$3$};

\node[main node] at ({ \y* cos(\x+(3)*\z) + (1-\y)*cos(\x+(2)*\z)},{ \y* sin(\x+(3)*\z) + (1-\y)*sin(\x+(2)*\z)}) (4) {$4_1$};
\node[main node] at ({cos(\x+(3)*\z)},{sin(\x+(3)*\z)}) (5) {$4_0$};
\node[main node] at ({ \y* cos(\x+(3)*\z) + (1-\y)*cos(\x+(4)*\z)},{ \y* sin(\x+(3)*\z) + (1-\y)*sin(\x+(4)*\z)}) (6) {$4_2$};

\node[main node] at ({ \y* cos(\x+(4)*\z) + (1-\y)*cos(\x+(3)*\z)},{ \y* sin(\x+(4)*\z) + (1-\y)*sin(\x+(3)*\z)}) (7) {$5_1$};
\node[main node] at ({cos(\x+(4)*\z)},{sin(\x+(4)*\z)}) (8) {$5_0$};
\node[main node] at ({ \y* cos(\x+(4)*\z) + (1-\y)*cos(\x+(5)*\z)},{ \y* sin(\x+(4)*\z) + (1-\y)*sin(\x+(5)*\z)}) (9) {$5_2$};

\node[main node] at ({ \y* cos(\x+(5)*\z) + (1-\y)*cos(\x+(4)*\z)},{ \y* sin(\x+(5)*\z) + (1-\y)*sin(\x+(4)*\z)}) (10) {$6_1$};
\node[main node] at ({cos(\x+(5)*\z)},{sin(\x+(5)*\z)}) (11) {$6_0$};
\node[main node] at ({ \y* cos(\x+(5)*\z) + (1-\y)*cos(\x+(6)*\z)},{ \y* sin(\x+(5)*\z) + (1-\y)*sin(\x+(6)*\z)}) (12) {$6_2$};

\node[main node] at ({ \y* cos(\x+(6)*\z) + (1-\y)*cos(\x+(5)*\z)},{ \y* sin(\x+(6)*\z) + (1-\y)*sin(\x+(5)*\z)}) (13) {$7_1$};
\node[main node] at ({cos(\x+(6)*\z)},{sin(\x+(6)*\z)}) (14) {$7_0$};
\node[main node] at ({ \y* cos(\x+(6)*\z) + (1-\y)*cos(\x+(7)*\z)},{ \y* sin(\x+(6)*\z) + (1-\y)*sin(\x+(7)*\z)}) (15) {$7_2$};

 \path[every node/.style={font=\sffamily}]
(1) edge (2)
(2) edge (3)
(1) edge (3)
(4) edge (5)
(5) edge (6)
(7) edge (8)
(8) edge (9)
(10) edge (11)
(11) edge (12)
(13) edge (14)
(14) edge (15)
(4) edge (2)
(4) edge (3)
(6) edge (1)
(7) edge (2)
(7) edge (3)
(9) edge (1)
(10) edge (2)
(10) edge (3)
(12) edge (1)
(13) edge (2)
(13) edge (3)
(15) edge (1)
(6) edge (7)
(6) edge (10)
(9) edge (10)
(6) edge (13)
(9) edge (13)
(12) edge (13);
\end{tikzpicture}

&

\def\x{270 - 360/8}
\def\z{360/8}

\begin{tikzpicture}
[scale=\sc, thick, main node/.style={circle, minimum size=3.8mm, inner sep=0.1mm,draw,font=\tiny\sffamily}]
\node[main node] at ({cos(\x+(0)*\z)},{sin(\x+(0)*\z)}) (1) {$1$};
\node[main node] at ({cos(\x+(1)*\z)},{sin(\x+(1)*\z)}) (2) {$2$};
\node[main node] at ({cos(\x+(2)*\z)},{sin(\x+(2)*\z)}) (3) {$3$};

\node[main node] at ({ \y* cos(\x+(3)*\z) + (1-\y)*cos(\x+(2)*\z)},{ \y* sin(\x+(3)*\z) + (1-\y)*sin(\x+(2)*\z)}) (4) {$4_1$};
\node[main node] at ({cos(\x+(3)*\z)},{sin(\x+(3)*\z)}) (5) {$4_0$};
\node[main node] at ({ \y* cos(\x+(3)*\z) + (1-\y)*cos(\x+(4)*\z)},{ \y* sin(\x+(3)*\z) + (1-\y)*sin(\x+(4)*\z)}) (6) {$4_2$};

\node[main node] at ({ \y* cos(\x+(4)*\z) + (1-\y)*cos(\x+(3)*\z)},{ \y* sin(\x+(4)*\z) + (1-\y)*sin(\x+(3)*\z)}) (7) {$5_1$};
\node[main node] at ({cos(\x+(4)*\z)},{sin(\x+(4)*\z)}) (8) {$5_0$};
\node[main node] at ({ \y* cos(\x+(4)*\z) + (1-\y)*cos(\x+(5)*\z)},{ \y* sin(\x+(4)*\z) + (1-\y)*sin(\x+(5)*\z)}) (9) {$5_2$};

\node[main node] at ({ \y* cos(\x+(5)*\z) + (1-\y)*cos(\x+(4)*\z)},{ \y* sin(\x+(5)*\z) + (1-\y)*sin(\x+(4)*\z)}) (10) {$6_1$};
\node[main node] at ({cos(\x+(5)*\z)},{sin(\x+(5)*\z)}) (11) {$6_0$};
\node[main node] at ({ \y* cos(\x+(5)*\z) + (1-\y)*cos(\x+(6)*\z)},{ \y* sin(\x+(5)*\z) + (1-\y)*sin(\x+(6)*\z)}) (12) {$6_2$};

\node[main node] at ({ \y* cos(\x+(6)*\z) + (1-\y)*cos(\x+(5)*\z)},{ \y* sin(\x+(6)*\z) + (1-\y)*sin(\x+(5)*\z)}) (13) {$7_1$};
\node[main node] at ({cos(\x+(6)*\z)},{sin(\x+(6)*\z)}) (14) {$7_0$};
\node[main node] at ({ \y* cos(\x+(6)*\z) + (1-\y)*cos(\x+(7)*\z)},{ \y* sin(\x+(6)*\z) + (1-\y)*sin(\x+(7)*\z)}) (15) {$7_2$};

\node[main node] at ({ \y* cos(\x+(7)*\z) + (1-\y)*cos(\x+(6)*\z)},{ \y* sin(\x+(7)*\z) + (1-\y)*sin(\x+(6)*\z)}) (16) {$8_1$};
\node[main node] at ({cos(\x+(7)*\z)},{sin(\x+(7)*\z)}) (17) {$8_0$};
\node[main node] at ({ \y* cos(\x+(7)*\z) + (1-\y)*cos(\x+(8)*\z)},{ \y* sin(\x+(7)*\z) + (1-\y)*sin(\x+(8)*\z)}) (18) {$8_2$};

 \path[every node/.style={font=\sffamily}]
(1) edge (2)
(2) edge (3)
(1) edge (3)
(4) edge (5)
(5) edge (6)
(7) edge (8)
(8) edge (9)
(10) edge (11)
(11) edge (12)
(13) edge (14)
(14) edge (15)
(16) edge (17)
(17) edge (18)
(4) edge (2)
(4) edge (3)
(6) edge (1)
(7) edge (2)
(7) edge (3)
(9) edge (1)
(10) edge (2)
(10) edge (3)
(12) edge (1)
(13) edge (2)
(13) edge (3)
(15) edge (1)
(16) edge (2)
(16) edge (3)
(18) edge (1)
(6) edge (7)
(6) edge (10)
(9) edge (10)
(6) edge (13)
(9) edge (13)
(12) edge (13)
(6) edge (16)
(9) edge (16)
(12) edge (16)
(15) edge (16);
\end{tikzpicture}

\\
$\A_6$ & $\A_7$ & $\A_8$
\end{tabular}
\caption{Several members of the family of graphs $\A_{k}$}\label{figAk}
\end{center}
\end{figure}

Figure~\ref{figAk} gives the drawings of $\A_{k}$ for $3 \leq k \leq 8$. Then we have the following:

\begin{proposition}\label{prop502}
For every $k \geq 3$, $\A_k$ is $(k-2)$-minimal.
\end{proposition}

We detail the proof of Proposition~\ref{prop502}, as well as establish the $\ell$-minimality of many other graphs, in Section~\ref{sec04}. Notably, the $\ell$-minimal graphs we present herein are all \emph{stretched cliques} --- graphs which can be obtained by starting with a complete graph, and applying a number of vertex-stretching operations, which we define in Section~\ref{sec02}. We remark that our vertex-stretching operation is a slight variant of that defined in~\cite{AuT24b}, and is closely related to similar operations studied earlier in~\cite{LiptakT03, AguileraEF14, BianchiENT17}. More generally, a number of stretched cliques have also been studied as instances of interest for other lift-and-project hierarchies~\cite{PenaVZ07, DobreV15, LaurentV23, Vargas23}. 

Understanding the behaviour of the $\LS_+$ operator, and in particular, the behaviour of the $\LS_+$-rank under graph operations was a natural research
direction following the seminal paper of Lov\'{a}sz and Schrijver~\cite{LovaszS91}. Related questions about the behaviour of the $\LS_+$-rank under basic graph operations were also raised by Goemans and the second author~\cite{GoemansT01}. There has been some very nice work establishing connections between
lift-and-project operator ranks (for some operators related to $\LS_+$) and graph minor operations especially on the maximum cut problem (see~\cite{Laurent2001,Laurent03,Laurent2003b}). However, as it was shown in ~\cite{LiptakT03}, $\LS_+$-rank behaves rather erratically with respect to
many established graph operations in the context of the stable set problem.  Thus, deeper investigation of the $\LS_+$-rank of graphs was justified. A key piece of such investigations is understanding
the combinatorial structure of minimal obstructions to effectiveness of such convex relaxations of the stable set problem obtained via the $\LS_+$ operator. Importance of
these questions were raised by many others. Notably, in a very well-known survey ``The Sandwich Theorem'' ~\cite{Knuth1994} about the Lov\'{a}sz theta function, Knuth poses six ``perplexing questions.'' Two of these questions involve the effectiveness of the $\LS_+$ operator on the stable set problem. One of the questions asks about what we call here 2-minimal graphs (answered in~\cite{LiptakT03}). One of the main results of this paper answers a more general version of Knuth's question by giving constructions of $2^{\ell -1}$ non-isomorphic $\ell$-minimal graphs for every positive integer $\ell$. 

\subsection{Proof strategy overview and a roadmap}

At a high level, our argument begins on a structural level as we develop a framework for analyzing broad classes of stretched cliques. This larger framework not only allows us to explicitly construct families of $\ell$-minimal graphs such as $\A_k$, but also explains why extremal examples are abundant, which yields additional consequences such as exponentially many non-isomorphic $\ell$-minimal graphs and vertex-transitive graphs with large $\LS_+$-rank.

The proof proceeds as follows. We begin with graphs obtained by stretching vertices in complete graphs, and organize them into the classes $\K_{n,d}$. Inside this family, the subclass $\tilde{\K}_{n,d}$ is especially convenient because the rank inequality $\sum_{i \in V(G)} x_i \leq \alpha(G)$ is a facet-inducing inequality of $\STAB(G)$ in this case.  Finally, because the natural deletion and destruction operations do not preserve $\tilde{\K}_{n,d}$, we further restrict the setting to another subclass of stretched cliques $\hat{\K}_{n,d}$ where the inductive argument carries out nicely. The explicit family $\A_k$ is then obtained as a particularly transparent instance of this general framework.

A core part of our inductive argument is the construction of a witness vector $v(G,\epsilon)$. Very informally, the vector $v(G,\epsilon)$ is a perturbed average of the maximum stable sets of a stretched clique. The perturbation is chosen so that $v(G,\epsilon)$ violates the rank inequality of $G$, while at the same time still lies in the relevant relaxation after many rounds of $\LS_+$. Thus, $v(G,\epsilon)$ serves as the main certificate for our lower bounds on $\LS_+$-rank.

The manuscript is organized as follows: In Section~\ref{sec02}, we define the aforementioned vertex-stretching operation, and mention some properties of the stable set polytopes of stretched cliques. Then we return to analyzing $\LS_+$-relaxations in Section~\ref{sec03}, as we establish the necessary facts --- some regarding general $\LS_+$-relaxations and some specifically applicable to the graphs of our interest --- and build up to our main results (Theorem~\ref{thm409} and Corollary~\ref{cor501}). We then explore the implications of these results in Section~\ref{sec04}. Therein, we prove Proposition~\ref{prop502} and establish the $\ell$-minimality of $\A_k$, and go on to show that there are in fact at least $2^{\ell-1}$ distinct $\ell$-minimal graphs for every positive integer $\ell$ (Theorem~\ref{thm504}). As a consequence of our construction, we also obtain a family of vertex-transitive graphs on $4\ell+8$ vertices with $\LS_+$-rank at least $\ell$ for every odd positive integer $\ell$ (Proposition~\ref{prop508} and Figure~\ref{figBk}). We conclude our manuscript by discussing some relevant remaining open questions in Section~\ref{sec05}.

\section{Stretched cliques and their properties}\label{sec02}

In this section, we revisit the vertex-stretching operation introduced in~\cite{AuT24b} and investigate some graphs that can be obtained by iteratively applying this operation to a complete graph. 

First, we need some graph theoretical notation. Given a positive integer $n$, let $K_n$ denote the \emph{complete graph} on $n$ vertices, with $V(K_n) \ce [n]$ and $E(K_n) \ce \set{ \set{i,j} : 1 \leq i < j \leq n}$. Also, given a graph $G$ and vertex $v \in V(G)$, we define the \emph{(open) neighborhood} of $v$ to be
\[
\Gamma_G(v) \ce \set{ u \in V(G) : \set{u,v} \in E(G)}.
\]
Next, given a set of vertices $S \subseteq V(G)$, we let $G - S$ denote the subgraph of $G$ induced by the vertices $V(G) \setminus S$, and call this the graph obtained from $G$ by the \emph{deletion} of $S$.  (When $S = \set{v}$, we will simply write $G - v$ instead of $G - \set{v}$.) Then we define 
\[
G \ominus v \ce G - (\Gamma_G(v) \cup \set{v})
\]
to be the graph obtained from $G$ by the \emph{destruction} of $v$. Next, given a graph $G$, vertex $v \in V(G)$, and (possibly empty) sets $A_1, \ldots, A_k \subseteq \Gamma_G(v)$ where $\bigcup_{i=1}^k A_i = \Gamma_G(v)$, we define the \emph{stretching} of $v$ to be the following transformation to $G$:
\begin{itemize}
\item
Replace $v$ by $k+1$ vertices: $v_0, v_1, \ldots, v_k$;
\item
For every $j \in [k]$, add an edge between $v_j$ and every vertex in $\set{v_0} \cup A_j$.
\end{itemize}

We remark that the above definition of vertex stretching is a slight variant of that defined in~\cite{AuT24b}, which has an additional requirement that $A_1, \ldots, A_k$ must each be a non-empty and proper subset of $\Gamma_G(v)$. Herein we say that a vertex-stretching operation is \emph{proper} if it satisfies this more restrictive definition. We also call the operation \emph{$k$-stretching} when we need to specify $k$. For example, in Figure~\ref{figVertexStretch}, $G_1 = K_6$, $G_2$ is obtained by $2$-stretching vertex $5$ in $G_1$, and $G_3$ is obtained by $2$-stretching vertex $6$ in $G_2$. 

The following is a key property of the vertex-stretching operation.

\begin{lemma}\label{lem300}
Let $G'$ be a graph obtained from $G$ by stretching a vertex in $G$. Then $r_+(G') \geq r_+(G)$.
\end{lemma}

\begin{proof}
The case for when the vertex stretching is proper was shown in~\cite[Proposition 14]{AuT24b}, so it remains to prove our claim for when the operation is not proper. Suppose $G'$ is obtained by $k$-stretching the vertex $v \in V(G)$ with $A_1, \ldots, A_k$ satisfying $\bigcup_{i=1}^k A_i = \Gamma_G(v)$. The stretching not being proper implies that $A_i = \Gamma_G(v)$ for some $i \in [k]$, and/or $A_i = \emptyset$ for some $i \in [k]$.

Define $S \ce \set{ i \in [k] : A_i \neq \emptyset}$, and let $G''$ be the subgraph of $G'$ induced by $(V(G) \setminus \set{v}) \cup \set{v_0} \cup \set{v_i : i \in S}$. Since $G''$ is an induced subgraph of $G'$, we have $r_+(G') \geq r_+(G'')$. We next show that  $r_+(G'') \geq r_+(G)$, which implies our claim. If $A_i \neq \Gamma_G(v)$ for every $i \in S$, then $G''$ can be obtained from $G$ by a proper vertex-stretching operation, and so $r_+(G'') \geq r_+(G)$. Otherwise, there exists $j \in S$ where $A_j = \Gamma_G(v)$. In that case, the subgraph of $G''$ induced by the vertices $(V(G) \setminus \set{v}) \cup \set{v_j}$ is isomorphic to $G$, which implies $r_+(G'') \geq r_+(G)$. Thus, our claim follows.
\end{proof}

The relationship between vertex-stretching operations and the $\LS_+$-rank of a graph was first studied in~\cite{LiptakT03}. Among other results, it was shown therein that applying a type-1 stretching operation (which is a proper $2$-stretching of a vertex where $A_1$ and $A_2$ are disjoint) cannot decrease the $\LS_+$-rank of a graph. Similar vertex-stretching operations and their impact on the $\LS_+$-rank of a graph were also analyzed in~\cite{AguileraEF14, BianchiENT17}. More recently, the authors showed that applying a proper $2$-stretching operation to a complete graph on at least $4$ vertices always increases its $\LS_+$-rank from $1$ to $2$. The key breakthrough of this work is that we are now able to prove that, for every $n \geq 4$, it is possible to $2$-stretch $n-3$ vertices in $K_n$ to increase its $\LS_+$-rank from $1$ to $n-2$, which produces an $(n-2)$-minimal graph. These results are detailed in Section~\ref{sec03}.

\def\y{0.7}
\def\sc{2.2}
\def\x{270 - 360/6}
\def\z{360/6}

\begin{figure}[htbp]
\begin{center}
\begin{tabular}{ccc}

\begin{tikzpicture}
[scale=\sc, thick,main node/.style={circle, minimum size=3.8mm, inner sep=0.1mm,draw,font=\tiny\sffamily}]
\node[main node] at ({cos(\x+(0)*\z)},{sin(\x+(0)*\z)}) (1) {$1$};
\node[main node] at ({cos(\x+(1)*\z)},{sin(\x+(1)*\z)}) (2) {$2$};
\node[main node] at ({cos(\x+(2)*\z)},{sin(\x+(2)*\z)}) (3) {$3$};
\node[main node] at ({cos(\x+(3)*\z)},{sin(\x+(3)*\z)}) (4) {$4$};
\node[main node] at ({cos(\x+(4)*\z)},{sin(\x+(4)*\z)}) (5) {$5$};
\node[main node] at ({cos(\x+(5)*\z)},{sin(\x+(5)*\z)}) (6) {$6$};

 \path[every node/.style={font=\sffamily}]
(1) edge (2)
(1) edge (3)
(2) edge (3)
(1) edge (4)
(2) edge (4)
(3) edge (4)
(1) edge (5)
(2) edge (5)
(3) edge (5)
(4) edge (5)
(1) edge (6)
(2) edge (6)
(3) edge (6)
(4) edge (6)
(5) edge (6);
\end{tikzpicture}

&
\scalebox{1}{
\begin{tikzpicture}
[scale=\sc, thick,main node/.style={circle, minimum size=3.8mm, inner sep=0.1mm,draw,font=\tiny\sffamily}]
\node[main node] at ({cos(\x+(0)*\z)},{sin(\x+(0)*\z)}) (1) {$1$};
\node[main node] at ({cos(\x+(1)*\z)},{sin(\x+(1)*\z)}) (2) {$2$};
\node[main node] at ({cos(\x+(2)*\z)},{sin(\x+(2)*\z)}) (3) {$3$};
\node[main node] at ({cos(\x+(3)*\z)},{sin(\x+(3)*\z)}) (4) {$4$};

\node[main node] at ({cos(\x+(4)*\z)},{sin(\x+(4)*\z)}) (5) {$5_0$};
\node[main node] at ({  cos(\x+(4)*\z) + (1-\y)*cos(330)},{ sin(\x+(4)*\z) + (1-\y)*sin(330)}) (6) {$5_1$};
\node[main node] at ({ cos(\x+(4)*\z) + (1-\y)*cos(210},{ sin(\x+(4)*\z) + (1-\y)*sin(210)}) (8) {$5_2$};

\node[main node] at ({cos(\x+(5)*\z)},{sin(\x+(5)*\z)}) (9) {$6$};

 \path[every node/.style={font=\sffamily}]
(1) edge (2)
(1) edge (3)
(2) edge (3)
(1) edge (4)
(2) edge (4)
(3) edge (4)
(6) edge (5)
(8) edge (5)
(1) edge (9)
(2) edge (9)
(3) edge (9)
(4) edge (9)
(3) edge (6)
(4) edge (6)
(9) edge (6)
(1) edge (8)
(2) edge (8)
(9) edge (8);
\end{tikzpicture}
}
&

\scalebox{1}{
\begin{tikzpicture}
[scale=\sc, thick,main node/.style={circle, minimum size=3.8mm, inner sep=0.1mm,draw,font=\tiny\sffamily}]
\node[main node] at ({cos(\x+(0)*\z)},{sin(\x+(0)*\z)}) (1) {$1$};
\node[main node] at ({cos(\x+(1)*\z)},{sin(\x+(1)*\z)}) (2) {$2$};
\node[main node] at ({cos(\x+(2)*\z)},{sin(\x+(2)*\z)}) (3) {$3$};
\node[main node] at ({cos(\x+(3)*\z)},{sin(\x+(3)*\z)}) (4) {$4$};

\node[main node] at ({cos(\x+(4)*\z)},{sin(\x+(4)*\z)}) (5) {$5_0$};
\node[main node] at ({  cos(\x+(4)*\z) + (1-\y)*cos(330)},{ sin(\x+(4)*\z) + (1-\y)*sin(330)}) (6) {$5_1$};
\node[main node] at ({ cos(\x+(4)*\z) + (1-\y)*cos(210},{ sin(\x+(4)*\z) + (1-\y)*sin(210)}) (8) {$5_2$};

\node[main node] at ({cos(\x+(5)*\z)},{sin(\x+(5)*\z)}) (9) {$6_0$};
\node[main node] at ({  cos(\x+(5)*\z) + (1-\y)*cos(30)},{ sin(\x+(5)*\z) + (1-\y)*sin(30)}) (10) {$6_1$};
\node[main node] at ({ cos(\x+(5)*\z) + (1-\y)*cos(270},{  sin(\x+(5)*\z) + (1-\y)*sin(270)}) (11) {$6_2$};

 \path[every node/.style={font=\sffamily}]
(1) edge (2)
(1) edge (3)
(2) edge (3)
(1) edge (4)
(2) edge (4)
(3) edge (4)
(6) edge (5)
(8) edge (5)
(10) edge (9)
(11) edge (9)
(3) edge (6)
(4) edge (6)
(1) edge (8)
(2) edge (8)
(11) edge (1)
(11) edge (2)
(11) edge (3)
(11) edge (4)
(11) edge (6)
(10) edge (3)
(10) edge (8);
\end{tikzpicture}
}
\\
$G_1$ & $G_2$ & $G_3$
\end{tabular}
\caption{Illustrating the vertex-stretching operation}\label{figVertexStretch}
\end{center}
\end{figure}


Thus, we restrict our discussion to $2$-stretching below. In particular, given integers $n \geq 3$ and $d \geq 0$, let $\K_{n,d}$ denote the set of graphs that can be obtained by $2$-stretching $d$ of the $n$ vertices from $K_n$. For example, the graph $\A_k$ introduced in Section~\ref{sec01} and illustrated in Figure~\ref{figAk} belongs to $\K_{k,k-3}$ for every $k \geq 3$. We also introduce some terminology that will ease our subsequent discussion of graphs in $\K_{n,d}$. Given $G \in \K_{n,d}$, let $D(G) \subseteq [n]$ be the set of vertices of $K_n$ which were stretched to obtain $G$. For each $i \in D(G)$, we call $i_0$ a \emph{hub} vertex, and $i_1, i_2$ \emph{wing} vertices. We also call each $i \in [n] \setminus D(G)$ an \emph{unstretched} vertex in $G$. Finally, given an index $i \in [n]$, we define the vertices \emph{associated} with $i$ to be $i_0, i_1, i_2$ if $i \in D(G)$, and just the unstretched vertex $i$ otherwise. Notice that every vertex in $G$ is associated with a unique $i \in [n]$. Also, observe that, given $G \in \K_{n,d}$ and distinct $i,j \in [n]$, it follows from the definition of vertex stretching that there must be at least one edge in $G$ joining a vertex associated with $i$ and a vertex associated with $j$.

\begin{example}\label{eg301}
Consider the graphs in Figure~\ref{figVertexStretch}. First, we have $G_1 \in \K_{6,0}$, $G_2 \in \K_{6,1}$, and $G_3 \in \K_{6,2}$. Also, notice that $D(G_3) = \set{5,6}$, and that the wing vertex $6_2$ is adjacent to a vertex associated with $i$ for every $i \in \set{1,2,3,4,5}$ --- we will revisit this point later on in Example~\ref{eg303}.
\end{example}

The following lemma gives some basic properties of graphs in $\K_{n,d}$.

\begin{lemma}\label{lem302}
Let $G \in \K_{n,d}$ where $n \geq 3$ and $d \geq 0$. Then
\begin{itemize}
\item[(i)]
for every $i \in D(G)$, $G \ominus i_0 \in \K_{n-1,d-1}$;
\item[(ii)]
for every $i \in [n] \setminus D(G)$,  $G - i \in \K_{n-1,d}$;
\item[(iii)]
$\alpha(G) = d+1$.
\end{itemize}
\end{lemma}

\begin{proof}
Both (i) and (ii) follow directly from the definition of $\K_{n,d}$ and our definition of vertex stretching. (iii) follows from the proof of~\cite[Lemma 20]{AuT24b} (which does not require that the vertex stretching be proper in its argument).
\end{proof}

Let $\bar{e}$ denote the vector of all-ones (whose dimension will be clear from the context). Given a graph $G$, the inequality $\bar{e}^{\top}x \leq \alpha(G)$ is often known as the \emph{rank inequality} of $G$, and is always valid for $\STAB(G)$. We next describe the graphs $G \in \K_{n,d}$ for which the rank inequality is in fact a facet-inducing inequality for $\STAB(G)$. Given $G \in \K_{n,d}$, $i \in D(G)$, and $\ell \in \set{1,2}$, define $\tilde{\Gamma}_{G}(i_{\ell})$ to be the set of indices $j \in [n] \setminus \set{i}$ where there is an edge between $i_{\ell}$ and a vertex associated with $j$. We define $\tilde{\K}_{n,d} \subseteq \K_{n,d}$ to be the set of graphs $G$ where 
\[
\tilde{\Gamma}_{G}(i_{\ell}) \subset [n] \setminus \set{i}
\]
for every $i \in D(G)$ and for every $\ell \in \set{1,2}$.

\begin{example}\label{eg303}
Let us revisit the graphs from Figure~\ref{figVertexStretch}. First, $G_1 \in \tilde{\K}_{6,0}$ as $D(G_1) = \emptyset$. More generally, we have $\tilde{\K}_{n,0} = \K_{n,0}$ for every $n \geq 3$.

For $G_2$, we have $D(G_2) = \set{5}$, and
\[
\tilde{\Gamma}_{G_2}(5_1) = \set{3,4,6}, \quad \tilde{\Gamma}_{G_2}(5_2) = \set{1,2,6},
\]
both of which are proper subsets of $[6] \setminus \set{5}$. Thus, $G_2 \in \tilde{\K}_{6,1}$. On the other hand (and as mentioned in Example~\ref{eg301}), notice that for $G_3$ we have $\tilde{\Gamma}_{G_3}(6_2) = \set{1,2,3,4,5} = [6] \setminus \set{6}$. Thus, $G_3 \in \K_{6,2} \setminus \tilde{\K}_{6,2}$.
\end{example}

Notice that to obtain a graph $G \in \tilde{\K}_{n,d}$ from $K_n$, every vertex-stretching operation must be proper. Thus,~\cite[Proposition 24]{AuT24b} applies to graphs in $\tilde{\K}_{n,d}$.

\begin{lemma}\label{lem304}
For every graph $G \in \tilde{\K}_{n,d}$, $\bar{e}^{\top}x \leq d+1$ is a facet-inducing inequality for $\STAB(G)$.
\end{lemma}

Next, we take a closer look at the graphs in $\K_{n,d} \setminus \tilde{\K}_{n,d}$. The next result helps show that the rank inequality is indeed not a facet-inducing inequality for the stable set polytope for these graphs.

\begin{lemma}\label{lem305}
Let $G \in \K_{n,d}$ where $n \geq 3$ and $d \geq 1$. If $G \not\in \tilde{\K}_{n,d}$, then there exists an edge $\set{u,v} \in E(G)$ consisting of a hub vertex and a wing vertex where $G - \set{u,v} \in \K_{n,d-1}$.
\end{lemma}

\begin{proof}
Given $G \not\in \tilde{\K}_{n,d}$, there exists $i \in D(G)$ and $\ell \in \set{1,2}$ where the wing vertex $i_{\ell}$ satisfies  $\tilde{\Gamma}_{G}(i_{\ell}) = [n] \setminus \set{i}$. Then notice that $G' \ce G - \set{ i_0, i_{3-\ell}} \in \K_{n,d-1}$, with the vertex $i_{\ell} \in V(G')$ now taking on the role of an unstretched vertex. 
\end{proof}

\begin{example}\label{eg306}
Recall that, as shown in Example~\ref{eg303}, $G_3$ from Figure~\ref{figVertexStretch} is in $\K_{6,2} \setminus \tilde{\K}_{6,2}$. Since $\tilde{\Gamma}_{G}(6_2) = [6] \setminus \set{6}$, we can remove the vertices $6_0, 6_1$ to obtain an induced subgraph in $\K_{6,1}$. In fact, observe that $G_3 -  \set{6_0, 6_1} \in \tilde{\K}_{6,1}$.
\end{example}

Thus, applying Lemma~\ref{lem305} iteratively, we see that for every graph $G \in \K_{n,d}$, there exists a partition $\set{C_0, C_1, \ldots, C_k}$ of $V(G)$, where
\begin{itemize}
\item
$C_0$ induces a graph in $\tilde{\K}_{n,d-k}$;
\item
$C_1, \ldots, C_k \in E(G)$, with each $C_j$ consisting of a hub vertex and a wing vertex.
\end{itemize}

We call such a partition $\set{C_0, \ldots, C_k}$ a \emph{stretched-clique decomposition} of $G$, and the subgraph of $G$ induced by $C_0$ a \emph{core stretched clique} of $G$. Notice that the rank inequality of $G$ is exactly the sum of the rank inequality of the core stretched clique induced by $C_0$, as well as the $k$ edge inequalities corresponding to $C_1, \ldots, C_k$. Moreover,  since every hub vertex in $G$ has degree $2$ (and thus does not belong to any clique of size at least $3$), each of the edge inequalities for the edges $C_1 , \ldots, C_k$ is facet-inducing for $\STAB(G)$. We next show that the rank inequality corresponding to $C_0$ is also a facet-inducing inequality for $\STAB(G)$.

\begin{lemma}\label{lem307}
Let $G \in \K_{n,d}$, and let $\set{C_0, \ldots, C_k}$ be a stretched-clique decomposition of $G$. Then
\begin{equation}\label{lem307eq1}
\sum_{v \in C_0} x_v \leq d-k+1
\end{equation}
is a facet-inducing inequality for $\STAB(G)$.
\end{lemma}

\begin{proof}
Let $G'$ be the subgraph of $G$ induced by $C_0$. Since $G'$ is an induced subgraph of $G$,~\eqref{lem307eq1} is valid for $\STAB(G)$. Next, we show that there are indeed $|V(G)|$ affinely independent incidence vectors of stable sets in $G$ that satisfy~\eqref{lem307eq1} with equality. First, if $k=0$, then $G' = G$ and the claim follows from Lemma~\ref{lem304}. Next, suppose $k \geq 1$, and so $G'$ is a proper subgraph of $G$. Then there is a collection of stable sets $\S$ of $G'$  where $|\S| = |V(G')|$ and the vectors $\set{ \chi_S : S \in \S}$ are affinely independent and all satisfy~\eqref{lem307eq1} with equality. Moreover, since the facet-inducing inequality~\eqref{lem307eq1} has full support (in the perspective of $G'$), we know that for every $v \in V(G')$, $v \in S$ for at least one $S \in \S$, and $v \not\in S$ for at least one $S \in \S$.

Now consider a vertex $v \in V(G) \setminus V(G')$. If $v = i_0$ for some $i \in D(G)$ (i.e., $v$ is a hub vertex), then by the construction in Lemma~\ref{lem305} we know that exactly one of $i_1, i_2$ is in $V(G')$. Suppose without loss of generality that it is $i_2$ (and so $i_1 \in V(G) \setminus V(G')$). Choose $S \in \S$ where $i_2 \not\in S$, and define $T_v \ce  S \cup \set{v}$. Then $T_v$ is a stable set in $G$ and $\chi_{T_v}$ satisfies~\eqref{lem307eq1} with equality.

Next, suppose $v \in V(G) \setminus V(G')$ is a wing vertex, and so without loss of generality let $v = i_1$ for some $i \in D(G)$. By the construction in Lemma~\ref{lem305} again, we know that $i_2 \in V(G')$, and $\tilde{\Gamma}_{G}(i_2) = [n] \setminus \set{i}$. Then choose $S \in \S$ such that $i_2 \in S$. 

Next, notice that for every $j \in D(G) \setminus \set{i}$, $S$ cannot contain both wing vertices $j_1$ and $j_2$ since $i_2 \in S$ and $i_2$ is adjacent to at least one of them (due to $\tilde{\Gamma}_{G}(i_2) = [n] \setminus \set{i}$). Now if $j_1 \in S$, then $\set{i_2, j_1}$ is not an edge, and so $i \not\in \tilde{\Gamma}_{G'}(j_1)$, and so $j_2 \in V(G')$. Now both $j_1, j_2 \in V(G')$ implies that $j_0 \in V(G')$. Thus, if we define $S'$ to be the set obtained from $S$ by replacing every wing vertex $j_{\ell} \in S$ that is not adjacent to $i_2$ by the hub vertex $j_0$, then $S'$ must be a stable set in $V(G')$. Also, $|S'| = |S|$, and so $\chi_{S'}$ must also satisfy~\eqref{lem307eq1} with equality. Finally, by the construction of $S'$, it must not contain any vertex that is adjacent to $v= i_1$. Hence, if we define $T_v \ce S' \cup \set{v}$ in this case, then we obtain a stable set in $G$ whose incidence vector satisfies~\eqref{lem307eq1} with equality.

Applying this process for all $v \in V(G) \setminus V(G')$, we see that
\begin{equation}\label{lemCoreStretchedCliqueFaceteq1}
\S \cup \set{T_v : v \in V(G) \setminus V(G')}
\end{equation}
gives a set of $|V(G)|$ stable sets whose incidence vectors satisfy~\eqref{lem307eq1} with equality. Also, observe that each $v \in V(G) \setminus V(G')$ appears in $T_v$ but not any other stable set~\eqref{lemCoreStretchedCliqueFaceteq1}, and thus we see that their incidence vectors must also be affinely independent. This finishes the proof.
\end{proof}


Next, given $G \in \K_{n,d}$, we define the \emph{deficiency} of $G$ to be the minimum $k$ for which there exists a stretched-clique decomposition of $G$ with $k$ edges. For example, our discussion in Example~\ref{eg306} shows that  $G_3$ from Figure~\ref{figVertexStretch} has deficiency $1$. In general, graphs in $\tilde{\K}_{n,d}$ have deficiency $0$, and Lemma~\ref{lem305} shows that every graph in $\K_{n,d}$ has deficiency at most $d$.

Given a graph $G$, let $\omega(G)$ denote the size of the largest clique in $G$ (the \emph{clique number} of $G$). The next result relates the deficiency and clique number of stretched cliques.

\begin{lemma}\label{lem308}
Let $G \in \K_{n,d}$. Then the deficiency of $G$ is at most $\max\set{0, \omega(G) - n +d}$.
\end{lemma}

\begin{proof}
First, suppose $\omega(G) \leq n-d$. We show that $G \in \tilde{\K}_{n,d}$ (and thus has deficiency $0$) in this case.   Consider any wing vertex $i_{\ell} \in V(G)$ where $i\in D(G)$ and $\ell \in \set{1,2}$. If $i_{\ell}$ were adjacent to every unstretched vertex in $[n] \setminus D(G)$, then $\set{i_{\ell}} \cup ([n] \setminus D(G))$ would induce a clique of size $n-d+1$ in $G$, a contradiction. Thus, there exists $j \in [n] \setminus D(G)$ where $\set{i,j}$ is not an edge, and thus, $j \not\in \tilde{\Gamma}_{G}(i_{\ell})$. This shows that $\tilde{\Gamma}_{G}(i_{\ell}) \subset [n] \setminus \set{i}$ for every wing vertex, and thus $G \in \tilde{\K}_{n,d}$.

Next, suppose $\omega(G) \geq n-d$. We prove our claim by induction on $\omega(G) - n + d$. The case $\omega(G) - n + d = 0$ has already been verified above. Next, given $G \in \K_{n,d}$, either $G \in \tilde{\K}_{n,d}$ (in which case $G$ has deficiency $0$ and the claim follows), or by Lemma~\ref{lem305} there exists $G' \in \K_{n,d-1}$ where $G'$ is an induced subgraph of $G$ and that the deficiency of $G'$ is that of $G$ minus one. By the inductive hypothesis we know that $G'$ has deficiency at most $\omega(G') - n + d -1$. Since $\omega(G') \leq \omega(G)$, it follows that $G$ has deficiency at most $\omega(G) - n + d$.
\end{proof}

The structural properties established in this section show that stretched cliques have a useful combinatorial decomposition, but they do not yet yield rank lower bounds. In the next section we combine this structure with a general $\LS_+$ argument. 

\section{Structural results for $\LS_+$-relaxations}\label{sec03}

In this section we prove the lower-bound mechanism that drives the paper. Our eventual goal is to show that, for many stretched cliques, the rank inequality requires a prescribed number of rounds of the $\LS_+$ to derive. This results in a lower bound on the $\LS_+$-rank of the graph, and in the extremal cases it will show that the upper bound $r_+(G)\le \lfloor \frac{|V(G)|}{3} \rfloor$ is attained.

We begin with two lemmas for general $\LS_+$-relaxations. These results are not specific to stretched cliques; rather, they isolate the convex-geometric mechanism behind our argument. Their purpose is to justify the witness vectors that will later be specialized to graphs in $\K_{n,d}$.

The main technical difficulty is that the natural inductive operations on stretched cliques --- deleting an unstretched vertex or destroying a hub vertex --- do not preserve the subclass $\tilde{\K}_{n,d}$. For that reason, after establishing the witness construction on $\tilde{\K}_{n,d}$, we introduce another subclass of stretched cliques $\hat{\K}_{n,d}$. This additional condition is what allows the inductive argument to close, and it is also what leads to many extremal examples rather than only a single explicit family.

We first isolate the general convex-geometric step behind the argument. The next two lemmas explain how a suitably perturbed vector can be used to show that a facet-defining inequality is not yet valid after a given number of rounds of $\LS_+$.
	
\begin{lemma}\label{lem401}
Let $P \subseteq [0,1]^n$ be a closed convex set, and let $k$ be a non-negative integer. Suppose that $P_I$ is full-dimensional and let $F \ce \set{ x \in P_I :  c^{\top}x = c_0}$
be a facet of $P_I$ where the facet-defining inequality $c^{\top}x \leq c_0$ is not valid for $\LS_+^{k}(P)$. Then, the relative interior of $F$ is a strict subset of the interior of $\LS_+^k(P)$.
\end{lemma}

\begin{proof}
Suppose the assumptions of the lemma hold. Let $\bar{x}$ be an optimal solution of
\[
\max\left\{c^{\top}x \, : \, x \in \LS_+^k(P)\right\}.
\] 
This maximum is attained because, by assumptions and the properties of $\LS_+$, $\LS_+^k(P)$ is a non-empty compact set and the objective function is continuous. Since $c^{\top}x \leq c_0$ is not a valid inequality for $\LS_+^k(P)$, and it is a facet-inducing inequality for $P_I$, we have $\bar{x} \notin P_I$. Let $Q \ce \conv\left(P_I \cup\{\bar{x}\}\right)$. Note that since $P_I$ is full-dimensional, $\relint(F) \subset \inte(Q) \subseteq \inte(\LS_+^{k}(P))$, where the last inclusion follows from the facts that $\bar{x} \in \LS_+^k(P)$, $P_I \subset \LS_+^k(P)$, and $\LS_+^k(P)$ is a convex set.
\end{proof}

Next, given a set $P \subseteq [0,1]^n$ and a linear equation $c^{\top}x = c_0$, we define the vector
\[
u_{c^{\top}x=c_0}(P) \ce \sum_{\substack{z \in \set{0,1}^n \cap P,\\ c^{\top}z = c_0}} \begin{bmatrix} 1 \\ z \end{bmatrix},
\]
when the intersection is not empty.
Observe that, for every choice of equation $c^{\top}x = c_0$, $u_{c^{\top}x=c_0}(P) \in \cone(P_I)$. We also extend the notation to allow multiple equalities --- e.g., $u_{c^{\top}x = c_0, c'^{\top}x = c_0'}(P) = \displaystyle\sum_{\substack{z \in \set{0,1}^n \cap P, \\ c^{\top}z = c_0, c'^{\top}z = c_0'}} \begin{bmatrix} 1 \\ z \end{bmatrix}$
(again, when the intersection is non-empty). Also, given a set $P \subseteq \mR^n$, we say that $P$ is \emph{lower-comprehensive} if for every $x \in P$ and for every $y$ where $x \geq y \geq 0$, it must be the case that $y \in P$ as well. Observe that $\FRAC(G)$ is lower-comprehensive for every graph $G$. Since $\LS_+$ preserves lower-comprehensiveness, it follows that $\LS_+^k(G)$ is lower-comprehensive for every non-negative integer $k$ (see, for instance, \cite{GoemansT01}).

The following is a key lemma to our main result, as it will be used in the inductive step of our argument that the vertex-stretching operation does increase the $\LS_+$-rank of a graph under some circumstances.

\begin{lemma}\label{lem402}
Let $P \subseteq [0,1]^n$ be a lower-comprehensive closed convex set and let $k$ be a non-negative integer. Suppose $P_I$ is full-dimensional and let $c^{\top}x \leq c_0$ be a facet-inducing inequality for $P_I$. If there exists a set of indices $D \subseteq [n]$ where
\begin{itemize}
\item
$c_0 > c^{\top} \chi_D  > 0$;
\item
for every $i \in D$, there exists $\epsilon > 0$ where $u_{x_i=1,c^{\top}x=c_0}(P) - \epsilon \begin{bmatrix}1 \\ \chi_D \end{bmatrix} \in \cone(\LS_+^k(P))$;
\item
for every $i \in [n] \setminus D$, there exists $\epsilon > 0$ where $u_{x_i=0,c^{\top}x=c_0}(P) - \epsilon \begin{bmatrix}1 \\ \chi_D \end{bmatrix} \in \cone(\LS_+^k(P))$.
\end{itemize}
Then, $u_{c^{\top}x=c_0}(P) - \epsilon \begin{bmatrix}1 \\ \chi_D \end{bmatrix} \in \cone(\LS_+^{k+1}(P))$ for some $\epsilon > 0$, and the $\LS_+$-rank of $c^{\top}x \leq c_0$ is at least $k+2$.
\end{lemma}

\begin{proof}
Define $\T \ce \set{ x \in \set{0,1}^n \cap P : c^{\top}x = c_0}$ (i.e., $\T$ consists of the integral points in $P$ which lie on the facet of $P_I$ defined by the inequality $c^{\top}x \leq c_0$), and let
\[
Y_0 \ce \sum_{ T \in \T } \begin{bmatrix} 1 \\ \chi_T \end{bmatrix} \begin{bmatrix} 1 \\ \chi_T \end{bmatrix}^{\top}.
\]
Then observe that, for every $i \in [n]$, $Y_0e_i = u_{x_i=1, c^{\top}x = c_0}(P)$ and $Y_0(e_0-e_i) = u_{x_i=0, c^{\top}x = c_0}(P)$, and thus $Y_0e_i, Y_0(e_0-e_i) \in \cone(P_I)$ for every $i \in [n]$. Also, since $Y_0$ is defined to be a sum of symmetric positive semidefinite matrices, we have $Y_0 \succeq 0$.


Next, we claim that the null space of $Y_0$ has dimension $1$ and is spanned by the vector $c' \ce \begin{bmatrix} -c_0 \\ c \end{bmatrix}$.  Notice that if vector $x$ satisfies $Y_0x = 0$, then $x^{\top} Y_0 x = 0$, and so $x^{\top} \begin{bmatrix} 1 \\ \chi_T \end{bmatrix} = 0$ for every $T \in \T$. Since $c^{\top}x \leq c_0$ is a facet-inducing inequality for $P_I$, this implies that $x$ must be a multiple of $c'$.

Next, we define the matrix $Y_1 \in \mR^{(n+1) \times (n+1)}$ where
\[
Y_1e_i \ce \begin{cases}
\begin{bmatrix} \frac{-c^{\top}\chi_D}{c_0} \epsilon(1 - \epsilon) \\ -\epsilon \chi_D \end{bmatrix} & \tn{if $i = 0$;}\\
\begin{bmatrix} -\epsilon \\ -\epsilon \chi_D \end{bmatrix} & \tn{if $i \in D$;}\\
0 & \tn{otherwise.}
\end{cases}
\]
Observe that $Y_1 = Y_1^{\top}$, and $\diag(Y_1) = Y_1e_0$. 

We show that  there exists $\epsilon > 0$ where $Y \ce Y_0 +  Y_1 \in \widehat{\LS}_+^{k+1}(P)$. First, it is apparent that $Y$ is symmetric and satisfies $Ye_0 = \diag(Y)$ (as both $Y_0$ and $Y_1$ satisfy these properties). Now observe that 
\begin{align*}
c'^{\top}Yc' = c'^{\top}Y_1c' 
&= \begin{bmatrix} -c_0 & c^{\top} \end{bmatrix}  \begin{bmatrix}  \frac{-c^{\top}\chi_D}{c_0} \epsilon(1 - \epsilon) & -\epsilon \chi_D^{\top} \\ -\epsilon \chi_D & -\epsilon \chi_D\chi_D^{\top} \end{bmatrix} \begin{bmatrix} -c_0 \\ c \end{bmatrix}\\
 &= \epsilon \left( (1+\epsilon)  c^{\top} \chi_D c_0 - \left( c^{\top}\chi_D \right)^2 \right)\\
 & = \epsilon c^{\top}\chi_D \left( (1+\epsilon)c_0 - c^{\top}\chi_D \right),
\end{align*}
which is positive for all $\epsilon > 0$ (due to the assumption that $c_0 > c^{\top}\chi_D > 0$). Thus, $Y \succeq 0$ for all sufficiently small $\epsilon > 0$.

Next, we show that $Ye_i \in \cone(\LS_+^k(P))$ for every $i \in [n]$. If $i \in [n] \setminus D$, then $Y_1e_i = 0$, and so 
\[
Ye_i = Y_0e_i = u_{x_i=1,c^{\top}x=c_0}(P) \in \cone(P_I) \subseteq \cone(\LS_+^k(P)).
\]
For $i \in D$, observe that $Y_1e_i = -\epsilon \begin{bmatrix} 1 \\ \chi_D \end{bmatrix}$, and so $Ye_i \in \cone(\LS_+^k(P))$ follows from the second assumption.

Next, we show that $Y(e_0-e_i) \in \cone(\LS_+^k(P))$ for every $i \in [n]$. Observe that
\[
Y(e_0-e_i) = \begin{cases}
u_{x_i=0, c^{\top}x = c_0}(P) + \epsilon \begin{bmatrix} 1 - (1-\epsilon)\frac{c^{\top}\chi_D}{c_0}  \\ 0 \end{bmatrix} &
\tn{if $i \in D$;}\\
u_{x_i=0, c^{\top}x = c_0}(P) - \epsilon \begin{bmatrix}  1 \\ \chi_D \end{bmatrix}  + \epsilon \begin{bmatrix} 1 - (1-\epsilon)\frac{c^{\top}\chi_D}{c_0}  \\ 0 \end{bmatrix} &
\tn{if $i \in [n] \setminus D$.}\\
\end{cases}
\]
Notice that $u_{x_i=0, c^{\top}x = c_0}(P)  \in \cone(P_I) \subseteq   \cone(\LS_+^k(P))$ for every $i \in D$, and $u_{x_i=0, c^{\top}x = c_0}(P) - \epsilon \begin{bmatrix}  1 \\ \chi_D \end{bmatrix} \in \cone(\LS_+^k(P))$ for every $i \in [n] \setminus D$ by the third assumption. Also, we have $0 \leq (1-\epsilon)\frac{c^{\top}\chi_D}{c_0} < 1$ for all small $\epsilon > 0$. Since $P$ is assumed to be lower-comprehensive, $0 \in P_I$, and so $ \begin{bmatrix} 1 - (1-\epsilon)\frac{c^{\top}\chi_D}{c_0}  \\ 0 \end{bmatrix}  \in \cone(P_I) \subseteq \cone(\LS_+^k(P))$. Since $\cone(\LS_+^k(P))$ is closed under vector addition, we have $Y(e_0 - e_i) \in \cone(\LS_+^k(P))$ for every $i \in [n]$.

Thus, it follows that $Ye_0 \in \cone(\LS_+^{k+1}(P))$ for some $\epsilon > 0$. Now
\[
c'^{\top}(Ye_0) = c'^{\top} (Y_1e_0) = \begin{bmatrix} -c_0& c^{\top} \end{bmatrix}  \begin{bmatrix}  \frac{-c^{\top}\chi_D}{c_0} \epsilon(1 - \epsilon) \\ -\epsilon \chi_D
\end{bmatrix} = -\epsilon^2 c^{\top}\chi_D < 0
\]
for every $\epsilon > 0$. Thus, $Ye_0 \not\in \cone(P_I)$, which implies that the facet-inducing inequality $c^{\top}x \leq c_0$ is not valid for $\LS_+^{k+1}(P)$. 

Finally, given $\epsilon \geq 0$, define $\bar{x} \in \mR^n$ where $\bar{x}_i \ce \frac{Y_0[i,0]}{Y_0[0,0]}$ for every $i \in [n]$. (Note that it is necessary that $Y_0[0,0] > 0$, or otherwise no integral point in $P$ satisfies $c^{\top}x = c_0$. Then the second and/or third assumption would imply that there exists a point in $\cone(P)$ with negative entries, contradicting $P \subseteq [0,1]^n$.) By the construction of $Y_0$ and that $c^{\top}x \leq c_0$ is a facet-inducing inequality for $P_I$, $\bar{x}$ is in the relative interior of $\set{ x \in P_I : c^{\top}x = c_0}$. Thus, Lemma~\ref{lem401} implies that $\bar{x}$ is in the interior of $\LS_+^{k+1}(P)$, which in turn implies that $Y_0e_0$ is in the interior of $\cone(\LS_+^{k+1}(P))$. Thus, it follows that $Y_0e_0 - \epsilon \begin{bmatrix} 1 \\ \chi_D \end{bmatrix} \in \cone(\LS_+^{k+1}(P))$ for some $\epsilon > 0$, and our claim follows.
\end{proof}

We now specialize the preceding discussion to stretched cliques. The vector $v(G,\epsilon)$ defined below is the concrete witness that will be used throughout the remainder of the paper. Given $G \in \K_{n,d}$ and $\epsilon > 0$, let $D_0 \ce \set{i_0 : i \in D(G)}$ (i.e., $D_0$ is the set of hub vertices in $G$), and  define the vector
\[
v(G,\epsilon) \ce u_{\bar{e}^{\top}x = d+1}(\FRAC(G)) - \epsilon \begin{bmatrix} 1 \\ \chi_{D_0} \end{bmatrix}.
\]
\begin{example}\label{egvA5}
We illustrate the definition of $v(G,\epsilon)$ using the graph $\A_5$ from Figure~\ref{figAk}. Recall that $\A_5 \in \K_{5,2}$, so $d=2$, and the set of hub vertices is $D_0=\{4_0,5_0\}$.

Next, the vector $u_{\bar e^\top x=3}(\FRAC(\A_5))$ is obtained by summing $\begin{bmatrix}1\\ \chi_S\end{bmatrix}$ over all stable sets $S$ of cardinality $3$ in $\A_5$. These $16$ stable sets are:
\begin{align*}
&\{1,4_0,5_0\},\ \{1,4_0,5_1\},\ \{1,4_1,5_0\},\ \{1,4_1,5_1\},\\
&\{2,4_0,5_0\},\ \{2,4_0,5_2\},\ \{2,4_2,5_0\},\ \{2,4_2,5_2\},\\
&\{3,4_0,5_0\},\ \{3,4_0,5_2\},\ \{3,4_2,5_0\},\ \{3,4_2,5_2\},\\
&\{4_0,5_1,5_2\},\ \{4_1,4_2,5_0\},\ \{4_1,4_2,5_2\},\ \{4_1,5_1,5_2\}.
\end{align*}
Therefore, in the vector $u_{\bar e^\top x=3}(\FRAC(\A_5))$, the $0$-entry is the number of stable sets of size $3$, and for each $i \in V(\A_5)$ the $i$-entry is the number of such stable sets that contain $i$.

Finally, the vector $\epsilon \begin{bmatrix}1\\ \chi_{D_0}\end{bmatrix}$ perturbs the coordinates indexed by $0$, $4_0$, and $5_0$. Thus, the entries of $v(G,\epsilon)$ are as shown in the following table:
\[
\begin{array}{l|cccccccccc}
i 
& 0 & 1 & 2 & 3 & 4_0 & 4_1 & 4_2 & 5_0 & 5_1 & 5_2
\\ \hline
\left[v(\A_5,\epsilon)\right]_i
& 16-\epsilon & 4 & 4 & 4 & 7-\epsilon & 5 & 6 & 7-\epsilon & 4 & 7
\end{array}
\]
\end{example}

In general, notice that $|D_0| = d$, and thus $\begin{bmatrix} -(d+1) \\ \bar{e} \end{bmatrix}^{\top} v(G,\epsilon) = \epsilon$. Therefore, if we manage to show that $v(G, \epsilon) \in \cone(\LS_+^k(G))$ for some $\epsilon > 0$ and a given non-negative integer $k$, then it would follow that $\bar{e}^{\top}x \leq d+1$ is not valid for $\LS_+^k(G)$, which would imply that $G$ has $\LS_+$-rank at least $k+1$. Moreover, the entries of $v(G,\epsilon)$ are chosen so that this vector behaves naturally under the graph operations appearing in our inductive arguments, which is what makes it particularly useful for proving rank lower bounds.

\begin{lemma}\label{lem403}
Let $G \in \tilde{\K}_{n,d}$ where $n \geq 3$ and $d \geq 1$, and let $k \geq 0$ be an integer. Suppose that
\begin{itemize}
\item
for every $i \in D(G)$, $v(G \ominus i_0, \epsilon) \in \cone(\LS_+^k(G \ominus i_0))$ for some $\epsilon > 0$;
\item
for every $i \in [n] \setminus D(G)$, $v(G - i, \epsilon) \in \cone(\LS_+^k(G - i))$ for some $\epsilon > 0$.
\end{itemize}
Then $v(G, \epsilon) \in \cone(\LS_+^{k+1}(G))$ for some $\epsilon > 0$.
\end{lemma}

\begin{proof}
We prove the result using Lemma~\ref{lem402} with $P \ce \FRAC(G)$, which is indeed lower-comprehensive and convex; moreover, $P_I=\STAB(G)$ which is full-dimensional. First, given $G \in \tilde{\K}_{n,d}$, it follows from Lemma~\ref{lem304} that $\bar{e}^{\top}x \leq d+1$ is a facet-inducing inequality for $\STAB(G)$. Now let $D$ be the set of hub vertices in $G$. Then $|D| = d \geq 1$, and so $0 < \bar{e}^{\top}\chi_D < d+1$.

Next, let $i \in D(G)$. Notice that given $S \subseteq V(G)$ where $i_0 \in S$, $S$ is a stable set of $G$ of size $d+1$ if and only if $S \setminus \set{i_0}$ is a stable set of $G \ominus i_0$ of size $d$. Therefore, 
\begin{align*}
{}& u_{x_{i_0}=1, \bar{e}^{\top}x=d+1}(\FRAC(G)) - \epsilon\begin{bmatrix} 1 \\ \chi_D \end{bmatrix} \in \cone(\LS_+^k(G)) \\
\iff{}& u_{\bar{e}^{\top}x=d}(\FRAC(G\ominus i_0)) - \epsilon\begin{bmatrix} 1 \\ \chi_{D \setminus \set{i_0}} \end{bmatrix} \in \cone(\LS_+^k(G \ominus i_0)) \\
\iff{}& v(G \ominus i_0, \epsilon) \in \cone(\LS_+^k(G \ominus i_0)).
\end{align*}
Thus, the first assumption here exactly fulfills the second condition in Lemma~\ref{lem402}. Similarly, we see that
\[
u_{x_i=0, \bar{e}^{\top}x=d+1}(\FRAC(G)) - \epsilon\begin{bmatrix} 1 \\ \chi_D \end{bmatrix} \in \cone(\LS_+^k(G)) \iff
v(G - i, \epsilon) \in \cone(\LS_+^k(G - i))
\]
for every vertex $i \not\in D$. Thus, the second assumption here fulfills the third assumption in Lemma~\ref{lem402} for the cases when $i$ is an unstretched vertex, and it only remains to establish that $v(G-i, \epsilon) \in \cone(\LS_+^k(G-i))$ when $i$ is a wing vertex.

Without loss of generality, suppose $i=j_1$ where $j \in D(G)$. Let $S \subseteq V(G)$ be a stable set in $G$ where $|S| = d+1$ and $j_1 \not\in S$. Then $S$ either contains $j_0$ or it does not.
Thus, we have
\begin{align*}
{}& u_{x_{j_1}=0, \bar{e}^{\top}x=d+1}(\FRAC(G)) - \epsilon\begin{bmatrix} 1 \\ \chi_D \end{bmatrix} \\
={}& \underbrace{u_{x_{j_1} = 0, x_{j_0}=1, \bar{e}^{\top}x=d+1}(\FRAC(G)) - \epsilon\begin{bmatrix} 1 \\ \chi_D \end{bmatrix}}_{v^{(1)}} + \underbrace{u_{x_{j_1}=0,x_{j_0}=0, \bar{e}^{\top}x=d+1}(\FRAC(G))}_{v^{(2)}}.
\end{align*} 
Notice that since $\set{j_0,j_1} \in E(G)$, if $x$ is an incidence vector of a stable set in $G$ with $x_{j_0}=1$, then it must follow that $x_{j_1}=0$. Thus,  $v^{(1)} = u_{x_{j_0}=1, \bar{e}^{\top}x=d+1}(\FRAC(G)) - \epsilon\begin{bmatrix} 1 \\ \chi_D \end{bmatrix}$, which we already showed above is in $\cone(\LS_+^k(G))$ due to our first assumption for hub vertices in $G$. Also, $v^{(2)} \in \cone(\STAB(G)) \subseteq \cone(\LS_+^k(G))$. Thus, it follows that $v^{(1)} + v^{(2)} \in \cone(\LS_+^k(G))$ (which is closed under vector addition). This in turn implies that $v(G-j_1, \epsilon) \in \cone(\LS_+^k(G - j_1))$. Thus, Lemma~\ref{lem402} applies, and we conclude that $v(G,\epsilon) \in \cone(\LS_+^{k+1}(G))$ in this case.
\end{proof}

Lemma~\ref{lem403} sets up the desired inductive step on $\tilde{\K}_{n,d}$; to close the argument, we now introduce auxiliary lemmas that pass through core stretched cliques and a different subclass of stretched cliques.

\begin{lemma}\label{lem404}
Let $G \in \K_{n,d}$ where $n \geq 3$ and $d$ is non-negative. Let $\epsilon > 0$ be a real number, and let $k$ be a non-negative integer. Also let $G' \in \K_{n,d+1}$ be a graph obtained from $G$ by stretching an unstretched vertex in $[n] \setminus D(G)$. If $v(G,\epsilon) \in \cone(\LS_+^k(G))$, then $v(G', \epsilon) \in \cone(\LS_+^k(G'))$.
\end{lemma}

\begin{proof}
Let $i \in [n] \setminus D(G)$ be the vertex in $G$ that is stretched to obtain $G'$. (Thus, $D(G') = D(G) \cup \set{i}$.) For convenience, we also let $D_0$ and $D_0'$ respectively denote the set of hub vertices in $G$ and $G'$. Next, we define $v' \in \mathbb{R}^{\set{0} \cup V(G')}$, where
\[
v'_j  = \begin{cases}
(v(G,\epsilon))_j & \tn{if $j = 0$ or $j \in V(G) \setminus \set{i}$;}\\
(v(G,\epsilon))_i & \tn{if $j= i_1$ or $j= i_2$;}\\
(v(G,\epsilon))_0 - (v(G,\epsilon))_i & \tn{if $j= i_0$.}
\end{cases}
\]
Since the vertex-stretching operation (regardless if it is proper or not) is a star-homomorphism (as defined in~\cite[Section 3]{AuT24b}),  it follows from~\cite[Proposition 11]{AuT24b} that $v(G, \epsilon) \in \cone(\LS_+^k(G)) \Rightarrow v' \in \cone(\LS_+^k(G'))$. Now recall that 
\[
v(G,\epsilon) = u_{\bar{e}^{\top}x = d+1}(\FRAC(G))  - \epsilon \begin{bmatrix} 1 \\ \chi_{D_0} \end{bmatrix}.
\]
Then, by the construction of $v'$, we have
\[
v' = u_{x_{i_0}=1,\bar{e}^{\top}x = d+2}(\FRAC(G')) + u_{x_{i_1}=1,x_{i_2}=1,\bar{e}^{\top}x = d+2}(\FRAC(G')) - \epsilon \begin{bmatrix} 1 \\ \chi_{D_0'} \end{bmatrix}.
\]
Now observe that every stable set of size $d+2$ in $G'$ must either contain none of $i_1, i_2$ (in which case it must contain $i_0$), both of $i_1$ and $i_2$, or exactly one of $i_1$ and $i_2$. Thus, we can write
\begin{align*}
v(G',\epsilon)
={}& u_{\bar{e}^{\top}x = d+2}(\FRAC(G')) - \epsilon \begin{bmatrix} 1 \\ \chi_{D_0'} \end{bmatrix}\\
={}& u_{x_{i_0} = 1, \bar{e}^{\top}x = d+2}(\FRAC(G')) +
 u_{x_{i_1} = 1, x_{i_2}=1, \bar{e}^{\top}x = d+2}(\FRAC(G')) + \\
{}& u_{x_{i_1} + x_{i_2}=1, \bar{e}^{\top}x = d+2}(\FRAC(G')) -\epsilon \begin{bmatrix} 1 \\ \chi_{D_0'} \end{bmatrix}\\
={}& v' + u_{x_{i_1} + x_{i_2}=1, \bar{e}^{\top}x = d+2}(\FRAC(G')).
 \end{align*}
Since $\cone(\LS_+^k(G'))$ is closed under vector addition, and that
\[
u_{x_{i_1} + x_{i_2}=1, \bar{e}^{\top}x = d+2}(\FRAC(G')) \in \cone(\STAB(G')) \subseteq \cone(\LS_+^k(G')),
\]
we conclude that $v(G',\epsilon) \in \cone(\LS_+^k(G'))$.
\end{proof}

Again, since $\tilde{\K}_{n,d}$ is not closed under the deletion of unstretched vertices or the destruction of hub vertices, we need to impose additional conditions for the inductive argument to close. The next lemmas are designed precisely to bridge this gap: they show how to pass from a core stretched clique back to a larger graph, which motivates the introduction of the class $\hat{\K}_{n,d}$. Given integers $n \geq 3$ and $d \geq 0$, we define $\hat{\K}_{n,d} \subseteq \K_{n,d}$ to be the set of stretched cliques $G$ where
\begin{equation}\label{eqhatKnd}
|\set{ \set{i_1, j_1}, \set{i_1,j_2}, \set{i_2,j_1}, \set{i_2, j_2} } \cap E(G)| = 1, \quad \forall i,j \in D(G), i \neq j.
\end{equation}
In other words, given $G \in \K_{n,d}$, if there is exactly one edge in $G$ that joins a vertex associated with $i$ and a vertex associated with $j$ for every pair of distinct indices $i,j \in D(G)$, then $G \in \hat{\K}_{n,d}$.

\def\y{0.7}
\def\sc{2.2}
\def\x{270 - 360/6}
\def\z{360/6}

\begin{figure}[htbp]
\begin{center}
\begin{tabular}{cc}

\scalebox{1}{
\begin{tikzpicture}
[scale=\sc, thick,main node/.style={circle, minimum size=3.8mm, inner sep=0.1mm,draw,font=\tiny\sffamily}]
\node[main node] at ({cos(\x+(0)*\z)},{sin(\x+(0)*\z)}) (1) {$1$};
\node[main node] at ({cos(\x+(1)*\z)},{sin(\x+(1)*\z)}) (2) {$2$};
\node[main node] at ({cos(\x+(2)*\z)},{sin(\x+(2)*\z)}) (3) {$3$};

\node[main node] at ({cos(\x+(3)*\z)},{sin(\x+(3)*\z)}) (40) {$4_0$};
\node[main node] at ({  cos(\x+(3)*\z) + (1-\y)*cos(270)},{ sin(\x+(3)*\z) + (1-\y)*sin(270)}) (41) {$4_1$};
\node[main node] at ({ cos(\x+(3)*\z) + (1-\y)*cos(150},{  sin(\x+(3)*\z) + (1-\y)*sin(150)}) (42) {$4_2$};

\node[main node] at ({cos(\x+(4)*\z)},{sin(\x+(4)*\z)}) (50) {$5_0$};
\node[main node] at ({  cos(\x+(4)*\z) + (1-\y)*cos(330)},{ sin(\x+(4)*\z) + (1-\y)*sin(330)}) (51) {$5_1$};
\node[main node] at ({ cos(\x+(4)*\z) + (1-\y)*cos(210},{ sin(\x+(4)*\z) + (1-\y)*sin(210)}) (52) {$5_2$};

\node[main node] at ({cos(\x+(5)*\z)},{sin(\x+(5)*\z)}) (60) {$6_0$};
\node[main node] at ({  cos(\x+(5)*\z) + (1-\y)*cos(30)},{ sin(\x+(5)*\z) + (1-\y)*sin(30)}) (61) {$6_1$};
\node[main node] at ({ cos(\x+(5)*\z) + (1-\y)*cos(270},{  sin(\x+(5)*\z) + (1-\y)*sin(270)}) (62) {$6_2$};

 \path[every node/.style={font=\sffamily}]
(1) edge (2)
(1) edge (3)
(2) edge (3)
(40) edge (41)
(40) edge (42)
(50) edge (51)
(50) edge (52)
(60) edge (61)
(60) edge (62)
(41) edge (1)
(41) edge (2)
(42) edge (3)
(41) edge (51)
(42) edge (62)
(41) edge (62)
(51) edge (2)
(52) edge (62)
(52) edge (1)
(52) edge (3)
(61) edge (1)
(61) edge (2)
(62) edge (3);
\end{tikzpicture}
}

&

\scalebox{1}{
\begin{tikzpicture}
[scale=\sc, thick,main node/.style={circle, minimum size=3.8mm, inner sep=0.1mm,draw,font=\tiny\sffamily}]
\node[main node] at ({cos(\x+(0)*\z)},{sin(\x+(0)*\z)}) (1) {$1$};
\node[main node] at ({cos(\x+(1)*\z)},{sin(\x+(1)*\z)}) (2) {$2$};
\node[main node] at ({cos(\x+(2)*\z)},{sin(\x+(2)*\z)}) (3) {$3$};

\node[main node] at ({cos(\x+(3)*\z)},{sin(\x+(3)*\z)}) (40) {$4_0$};
\node[main node] at ({  cos(\x+(3)*\z) + (1-\y)*cos(270)},{ sin(\x+(3)*\z) + (1-\y)*sin(270)}) (41) {$4_1$};
\node[main node] at ({ cos(\x+(3)*\z) + (1-\y)*cos(150},{  sin(\x+(3)*\z) + (1-\y)*sin(150)}) (42) {$4_2$};

\node[main node] at ({cos(\x+(4)*\z)},{sin(\x+(4)*\z)}) (50) {$5_0$};
\node[main node] at ({  cos(\x+(4)*\z) + (1-\y)*cos(330)},{ sin(\x+(4)*\z) + (1-\y)*sin(330)}) (51) {$5_1$};
\node[main node] at ({ cos(\x+(4)*\z) + (1-\y)*cos(210},{ sin(\x+(4)*\z) + (1-\y)*sin(210)}) (52) {$5_2$};

\node[main node] at ({cos(\x+(5)*\z)},{sin(\x+(5)*\z)}) (60) {$6_0$};
\node[main node] at ({  cos(\x+(5)*\z) + (1-\y)*cos(30)},{ sin(\x+(5)*\z) + (1-\y)*sin(30)}) (61) {$6_1$};
\node[main node] at ({ cos(\x+(5)*\z) + (1-\y)*cos(270},{  sin(\x+(5)*\z) + (1-\y)*sin(270)}) (62) {$6_2$};

 \path[every node/.style={font=\sffamily}]
(1) edge (2)
(1) edge (3)
(2) edge (3)
(40) edge (41)
(40) edge (42)
(50) edge (51)
(50) edge (52)
(60) edge (61)
(60) edge (62)
(41) edge (3)
(41) edge (2)
(42) edge (1)
(42) edge (2)
(42) edge (3)
(42) edge (51)
(42) edge (62)
(51) edge (1)
(51) edge (61)
(52) edge (2)
(52) edge (3)
(61) edge (2)
(61) edge (3)
(62) edge (1)
(62) edge (3);
\end{tikzpicture}
}
\\
$G_1$ & $G_2$ 
\end{tabular}
\caption{Illustrating the definition of the set of graphs $\hat{\K}_{n,d}$ (see Example~\ref{eg405})}\label{fighatKnd}
\end{center}
\end{figure}

\begin{example}\label{eg405}
Consider the graphs in Figure~\ref{fighatKnd}. First, notice that $G_1 \in \tilde{\K}_{6,3}$. However, since there are two edges ($\set{4_1,6_2}$ and $\set{4_2,6_2}$) joining vertices associated with $i=4$ and $j=6$, $G_1$ violates~\eqref{eqhatKnd}, and thus $G_1 \not\in \hat{\K}_{6,3}$.

On the other hand, notice that $G_2 \in \hat{\K}_{6,3}$ as there is exactly one edge joining a vertex associated with $i$ and a vertex associated with $j$ for every $(i,j) \in \set{ (4,5), (4,6), (5,6)}$. However, $\tilde{\Gamma}_{G_2}(4_2) = \set{1,2,3,5,6}$, and so $G_2 \not\in \tilde{\K}_{6,3}$. Using the procedure outlined in Lemma~\ref{lem305}, we can obtain that $G' \ce G_2 - \set{4_0, 4_1}$ is a core stretched clique of $G_2$.

Now notice that $\Gamma_{G_2}(4_1) \subseteq \Gamma_{G_2}(4_2)$. Thus, we can start with $G'$ and apply an improper vertex-stretching operation to $4_2$ to obtain $G_2$. As we show in Lemma~\ref{lem406}, this is not a coincidence for graphs in $\hat{\K}_{n,d} \setminus \tilde{\K}_{n,d}$.
\end{example}

Prior to this work, every known $\ell$-minimal graph (i.e., the $3$-cycle, the graphs $G_{2,1}$, $G_{2,2}$, $G_{3,1}$, and $G_{4,1}$ from Figure~\ref{figKnownEG}, as well as all other $3$-minimal and $4$-minimal graphs found in~\cite{AuT24b}) belongs to $\hat{\K}_{\ell+2, \ell-1}$. We shall show by the end of this section that many more graphs in $\hat{\K}_{\ell+2,\ell-1}$ are $\ell$-minimal. (We will also show in Section~\ref{sec04} that $\ell$-minimal graphs also exist outside of $\hat{\K}_{\ell+2,\ell-1}$.)

First, the following lemma makes concrete a property of the stretched cliques in $\hat{\K}_{n,d} \setminus \tilde{\K}_{n,d}$ we mentioned in Example~\ref{eg405}.

\begin{lemma}\label{lem406}
Let $G \in \hat{\K}_{n,d}$ where $n \geq 3$ and $d$ is non-negative. Suppose $G$ has deficiency $k \geq 1$. Then $G$ contains a core stretched clique $G' \in \hat{\K}_{n,d-k}$. Moreover, $G$ can be obtained from $G'$ by $2$-stretching $k$ vertices.
\end{lemma}

\begin{proof}
Given $G \in \hat{\K}_{n,d}$ with deficiency $k \geq 1$, there exists a stretched-clique decomposition $\set{C_0, \ldots, C_k}$ of $G$. Let $G'$ be the core stretched clique of $G$ induced by $C_0$. Since $G'$ is an induced subgraph of $G$, it cannot violate~\eqref{eqhatKnd}. Hence, we have $G' \in \hat{\K}_{n,d-k}$.

It remains to show that $G$ can be obtained from $G'$ by $2$-stretching $k$ vertices in $G'$. Consider again the stretched-clique decomposition of $G$, and focus on the edge $C_1 \in E(G)$. Without loss of generality, suppose $C_1 = \set{i_0, i_1}$ for some $i \in D(G)$. Since $i_2 \in V(G')$, we know that $\tilde{\Gamma}_{G}(i_2) = [n] \setminus \set{i}$. Thus, for every $j \in D(G) \setminus \set{i}$, there is an edge between $i_2$ and one of $j_1, j_2$. Since $G \in \hat{\K}_{n,d}$ and thus satisfies~\eqref{eqhatKnd},  there is no edge between $i_1$ and a vertex associated with $j$. This implies that $\Gamma_{G}(i_1) \subseteq \Gamma_{G}(i_2)$. Thus, we can $2$-stretch $i_2 \in V(G')$ to obtain the subgraph of $G$ induced by $V(G') \cup C_1$. Applying this procedure iteratively to $C_2, \ldots, C_k$ finishes the proof.
\end{proof}

Lemmas~\ref{lem404} and~\ref{lem406} combine to imply the following.

\begin{lemma}\label{lem407}
Let $G \in \hat{\K}_{n,d}$ where $n \geq 3$ and $d$ is non-negative. Let $\epsilon > 0$ be a real number, let $k$ be a non-negative integer, and let $G'$ be a core stretched clique of $G$. If $v(G', \epsilon) \in \cone(\LS_+^k(G'))$, then $v(G,\epsilon) \in \cone(\LS_+^k(G))$.
\end{lemma}

We are now finally ready to use Lemma~\ref{lem403} to show that, for many stretched cliques, the point $v(G,\epsilon)$ (which does not belong to $\STAB(G)$ for all sufficiently small $\epsilon > 0$) survives many iterations of $\LS_+$.

\begin{proposition}\label{prop408}
Let $G \in \hat{\K}_{n,d} \cap \tilde{\K}_{n,d}$, where $n \geq 3$ and $d$ is non-negative. Also, let $k \ce \max\set{3, \omega(G)}$. Then $v(G,\epsilon) \in \cone\left(\LS_+^{n-k}(G)\right)$ for some $\epsilon > 0$.
\end{proposition}

\begin{proof}
We prove our claim by induction on $n$. If $k=3$ (which implies $\omega(G) \in \set{2,3}$), then the base case is $n=3$. Here, we have $d \in \set{0,1,2,3}$, and the only graph in $\hat{\K}_{3,d} \cap \tilde{\K}_{3,d}$ is the $(2d+3)$-cycle. One can check that $v(G, \epsilon) \in \cone\left(\FRAC(G)\right) = \cone\left(\LS_+^0(G)\right)$ for all $\epsilon \in [0,1]$ in all four cases. If $k \geq 4$, then the base case is when $n=k$, which implies that $G = K_n$. In this case, we have $v(K_n, \epsilon) \in \cone\left(\FRAC(K_n)\right)$ for every $\epsilon \in [0, n-2]$.

Next, we prove the inductive step by applying the framework outlined in Lemma~\ref{lem403}. Let $G \in \hat{\K}_{n,d} \cap \tilde{\K}_{n,d}$ where $n \geq k+1$ (which implies $d \geq 1$). Given $i \in D(G)$, observe that $G \ominus i_0 \in \hat{\K}_{n-1, d-1}$. Then $G \ominus i_0$ contains a core stretched clique $G' \in \hat{\K}_{n-1,d'} \cap \tilde{\K}_{n-1,d'}$ for some $d' \leq d-1$. Also, since $G'$ is a subgraph of $G$, we have $\omega(G') \leq \omega(G) \leq k$. Thus, by the inductive hypothesis, we have $v(G', \epsilon) \in \cone(\LS_+^{n-k-1}(G'))$ for some $\epsilon > 0$. Then Lemma~\ref{lem407} implies that $v(G \ominus i_0, \epsilon) \in \cone(\LS_+^{n-k-1}(G \ominus i_0))$ for some $\epsilon > 0$.

Next, given $i \in [n] \setminus D(G)$, we have $G - i \in \hat{\K}_{n-1, d}$. Using a similar argument as in the preceding paragraph, observe that $G-i$ contains a core stretched clique $G'' \in \hat{\K}_{n-1,d''} \cap  \tilde{\K}_{n-1,d''}$ for some $d'' \leq d$, with $\omega(G'') \leq \omega(G) \leq k$. Thus, the inductive hypothesis implies that $v(G'', \epsilon) \in \cone(\LS_+^{n-k-1}(G''))$ for some $\epsilon > 0$, which (due to Lemma~\ref{lem407} again) implies that $v(G -i, \epsilon) \in \cone(\LS_+^{n-k-1}(G - i))$ for some $\epsilon > 0$. Thus, it follows from Lemma~\ref{lem403} that $v(G, \epsilon) \in \cone(\LS_+^{n-k}(G))$ for some $\epsilon > 0$.
\end{proof}

Proposition~\ref{prop408} readily implies the following, which provides a $\LS_+$-rank lower bound to many stretched cliques that depends on the clique number of the graph.

\begin{theorem}\label{thm409}
Let $G \in \hat{\K}_{n,d}$ where $n \geq 3$ and $d$ is non-negative. Let $k \ce \max\set{3, \omega(G)}$. Then $v(G,\epsilon) \in \cone\left(\LS_+^{n-k}(G)\right)$, and $r_+(G) \geq n-k+1$.
\end{theorem}

\begin{proof}
Given $G \in \hat{\K}_{n,d}$, it must contain a core stretched clique $G' \in \hat{\K}_{n,d'} \cap \tilde{\K}_{n,d'}$ for some $d' \leq d$. Since $\omega(G') \leq \omega(G)$, Proposition~\ref{prop408} implies that $v(G', \epsilon) \in \cone(\LS_+^{n-k}(G'))$ for some $\epsilon > 0$. Then Lemma~\ref{lem407} implies that $v(G, \epsilon) \in \cone(\LS_+^{n-k}(G))$ for some $\epsilon > 0$. Since $v(G, \epsilon)$ violates $-(d+1)x_0 + \bar{e}^{\top}x \leq 0$ (which is valid for $\cone\left(\STAB(G)\right)$), it follows that $r_+(G) \geq n-k+1$.
\end{proof}

\section{ $\ell$-minimal graphs and implications}\label{sec04}

Recall that a graph $G$ is $\ell$-minimal if $r_+(G) = \ell$ and $|V(G)| = 3\ell$. In this section, we describe a number of $\ell$-minimal graphs, as well as other implications of the structural results we developed in Section~\ref{sec03}.

First, notice that every graph $G \in \K_{\ell+2, \ell-1}$ contains exactly $3\ell$ vertices. Thus, Theorem~\ref{thm409} implies the following.

\begin{corollary}\label{cor501}
Let $\ell$ be a positive integer, and let $G \in \hat{\K}_{\ell+2, \ell-1}$ where $\omega(G) \leq 3$. Then $G$ is $\ell$-minimal.
\end{corollary}

Notice that the graphs covered by Corollary~\ref{cor501} are also extremal for the classical upper bound $r_+(G)\le \alpha(G)$. Indeed, if $G\in \hat{\K}_{\ell+2,\ell-1}$, then
Lemma~\ref{lem302}(iii) gives $\alpha(G)=\ell$, while Corollary~\ref{cor501} gives $r_+(G)=\ell$. Thus $r_+(G)=\alpha(G)$ for all graphs in this family.

Now recall the graphs $\A_k$ defined in Section~\ref{sec01} and illustrated in Figure~\ref{figAk}. We now apply Corollary~\ref{cor501} to show that these graphs are indeed $\ell$-minimal, giving a proof to Proposition~\ref{prop502}.

\begin{proof}[Proof of Proposition~\ref{prop502}]
We prove the result using Corollary~\ref{cor501}. First, as suggested by the vertex labels, we see that $\A_k \in \K_{k,k-3}$ for every $k \geq 3$, with $D(\A_k) = \set{4, \ldots, k}$. Now for every $i,j \in D(\A_k)$ where $i<j$, the only edge between vertices associated with $i$ and $j$ is $\set{i_2,j_1}$, and thus $\A_k \in \hat{\K}_{k,k-3}$. 

Next, we show that $\A_k$ does not contain $K_4$ as an induced subgraph. Observe that each hub vertex has degree two and thus cannot be contained in a $K_4$. Also, $\Gamma_{\A_k}(1) = \set{2,3, 4_2, \ldots, k_2}$. Since $\set{4_2, \ldots, k_2}$ is a stable set, $1$ cannot be contained in a $K_4$ either. However, among the remaining vertices (the unstretched vertices $2$, $3$, and the wing vertices), the only $3$-cycles are induced by $\set{2,3,i_1}$ for some $i \in \set{4, \ldots, k}$. Thus, it follows from Corollary~\ref{cor501} that $\A_k$ is $(k-2)$-minimal.
\end{proof}

Thus, we now know that $\ell$-minimal graphs do exist for every positive integer $\ell$. In fact, we show that the number of $\ell$-minimal graphs grows (at least) exponentially as a function of $\ell$. Given an integer $k \geq 3$ and a subset $S \subseteq \set{4, \ldots, k}$, we define the graph $\A_{k,S}$ where
\begin{align*}
V(\A_{k,S}) &\ce V(\A_k), \\
E(\A_{k,S}) &\ce E(\A_k) \cup \set{ \set{i_2,2} : i \in S}.
\end{align*}

Using the same ideas from the proof of Proposition~\ref{prop502}, one can show that $\A_{k,S} \in \hat{\K}_{k,k-3}$ and $\omega(\A_{k,S}) = 3$ for all possible choices of $S$, and so $r_+(\A_{k,S}) = k-2$ for every possible choice of $S$. The next result shows that distinct choices of $S$ indeed produce non-isomorphic graphs.

\begin{lemma}\label{lem503}
Let $k \geq 3$ be an integer, and let $S,S' \subseteq \set{4, \ldots, k}$. If $S \neq S'$, then $\A_{k,S}$ and $\A_{k,S'}$ are not isomorphic to each other. 
\end{lemma}

\begin{proof}
First, notice that the degrees of vertices in $\set{1,3} \cup \set{ i_0, i_1 : 4 \leq i \leq k}$ are invariant under the choice of $S$. For the other vertices, we have $\deg(2) = k-1 + |S|$, and
\[
\deg(i_2) = \begin{cases}
k-i+2 &\tn{if $i \not\in S$;}\\
k-i+3 &\tn{if $i \in S$.}
\end{cases}
\]
Thus, given the list of vertex degrees of $\A_{k,S}$, we can remove the entries that we know correspond to the vertex degrees of $\set{1,3} \cup \set{i_0, i_1 : 4 \leq i \leq k}$, and then uniquely recover the set $S$ from the remaining vertex degrees. Therefore, we see that $\A_{k,S}$ and $\A_{k,S'}$ have distinct lists of vertex degrees whenever $S \neq S'$, and so the two graphs cannot possibly be isomorphic to each other.
\end{proof}

Therefore, for every $k \geq 3$, $\set{ \A_{k,S} : S \subseteq \set{4,\ldots, k}}$ gives a set of $2^{k-3}$ non-isomorphic $(k-2)$-minimal graphs, and we have the following.

\begin{theorem}\label{thm504}
There are at least $2^{\ell-1}$ non-isomorphic $\ell$-minimal graphs for every positive integer $\ell$.
\end{theorem}

Theorem~\ref{thm504} is tight for $\ell =1$ and $\ell =2$, as all $1$-minimal and $2$-minimal graphs are known. For $\ell = 3$, an exhaustive computational search found $13$  non-isomorphic graphs that satisfy the conditions in Corollary~\ref{cor501}. We also computed the optimal value of $\max\set{ \bar{e}^{\top}x : x \in \LS_+^2(G)}$ for each of these graphs using CVX, a package for specifying and solving convex programs~\cite{CVX, GrantB08} with the SeDuMi solver~\cite{Sturm99}. All $13$ graphs have an optimal value of at least $3.004$, which aligns with our analytical findings that they are all $3$-minimal. These graphs and their corresponding optimal values are listed in Figure~\ref{fighatK52}. We also performed a similar search for $\ell = 4$, and found $588$ non-isomorphic graphs in $\hat{\K}_{6,3}$ with clique number at most $3$. This suggests that the number of $\ell$-minimal graphs grows rather rapidly as a function of $\ell$.

\def\x{270 - 135}
\def\z{360/4}
\def\y{0.7}
\def\sc{1.2}
\def\placev{\node[main node] at ({cos(\x+(1)*\z)},{sin(\x+(1)*\z)}) (1) {}; \node[main node] at ({cos(\x+(1.5)*\z)},{sin(\x+(1.5)*\z)}) (2) {}; \node[main node] at ({cos(\x+(2)*\z)},{sin(\x+(2)*\z)}) (3) {}; \node[main node] at ({ \y* cos(\x+(3)*\z) + (1-\y)*cos(\x+(2)*\z)},{ \y* sin(\x+(3)*\z) + (1-\y)*sin(\x+(2)*\z)}) (41) {}; \node[main node] at ({cos(\x+(3)*\z)},{sin(\x+(3)*\z)}) (40) {}; \node[main node] at ({ \y* cos(\x+(3)*\z) + (1-\y)*cos(\x+(4)*\z)},{ \y* sin(\x+(3)*\z) + (1-\y)*sin(\x+(4)*\z)}) (42) {}; \node[main node] at ({ \y* cos(\x+(4)*\z) + (1-\y)*cos(\x+(3)*\z)},{ \y* sin(\x+(4)*\z) + (1-\y)*sin(\x+(3)*\z)}) (51) {}; \node[main node] at ({cos(\x+(4)*\z)},{sin(\x+(4)*\z)}) (50) {}; \node[main node] at ({ \y* cos(\x+(4)*\z) + (1-\y)*cos(\x+(5)*\z)},{ \y* sin(\x+(4)*\z) + (1-\y)*sin(\x+(5)*\z)}) (52) {};}

\def\placecommone{  \path (1) edge (2) (2) edge (3) (3) edge (1) (40) edge (41) (40) edge (42) (50) edge (51) (50) edge (52);}

\begin{figure}[htbp]
\begin{center}
\begin{tabular}{ccccccc}
\begin{tikzpicture}[scale=\sc, thick,main node/.style={circle,  minimum size=1.5mm, fill=black, inner sep=0.1mm,draw,font=\tiny\sffamily}] \placev \placecommone
 \path (1) edge (41) (2) edge (41) (3) edge (42)  (1) edge (51) (2) edge (51) (3) edge (52) (41) edge (52);
\end{tikzpicture}
& 
\begin{tikzpicture}[scale=\sc, thick,main node/.style={circle,  minimum size=1.5mm, fill=black, inner sep=0.1mm,draw,font=\tiny\sffamily}] \placev \placecommone
 \path (1) edge (41) (2) edge (41) (3) edge (42)  (1) edge (52) (2) edge (51) (3) edge (51) (41) edge (51);
\end{tikzpicture}
& 
\begin{tikzpicture}[scale=\sc, thick,main node/.style={circle,  minimum size=1.5mm, fill=black, inner sep=0.1mm,draw,font=\tiny\sffamily}] \placev \placecommone
 \path (1) edge (41) (2) edge (41) (3) edge (42)  (1) edge (52) (2) edge (51) (3) edge (51) (41) edge (52);
\end{tikzpicture}
& 
\begin{tikzpicture}[scale=\sc, thick,main node/.style={circle,  minimum size=1.5mm, fill=black, inner sep=0.1mm,draw,font=\tiny\sffamily}] \placev \placecommone
 \path (1) edge (41) (2) edge (41) (3) edge (42)  (1) edge (51) (2) edge (51) (3) edge (52) (42) edge (52);
\end{tikzpicture}
& 
\begin{tikzpicture}[scale=\sc, thick,main node/.style={circle,  minimum size=1.5mm, fill=black, inner sep=0.1mm,draw,font=\tiny\sffamily}] \placev \placecommone
 \path (1) edge (41) (2) edge (41) (3) edge (42)  (1) edge (52) (2) edge (51) (3) edge (51) (42) edge (52);
\end{tikzpicture}
& 
\begin{tikzpicture}[scale=\sc, thick,main node/.style={circle,  minimum size=1.5mm, fill=black, inner sep=0.1mm,draw,font=\tiny\sffamily}] \placev \placecommone
 \path (1) edge (41) (1) edge (42) (2) edge (41) (3) edge (42)  (1) edge (52) (2) edge (51) (3) edge (52) (41) edge (51);
\end{tikzpicture}
& 
\begin{tikzpicture}[scale=\sc, thick,main node/.style={circle,  minimum size=1.5mm, fill=black, inner sep=0.1mm,draw,font=\tiny\sffamily}] \placev \placecommone
 \path (1) edge (41) (1) edge (42) (2) edge (41) (3) edge (42)  (1) edge (51) (2) edge (52) (3) edge (51) (41) edge (51);
\end{tikzpicture}\\
3.01280&
3.01224&
3.01183&
3.01059&
3.01020&
3.00911&
3.00808\\
\\
\begin{tikzpicture}[scale=\sc, thick,main node/.style={circle,  minimum size=1.5mm, fill=black, inner sep=0.1mm,draw,font=\tiny\sffamily}] \placev \placecommone
 \path (1) edge (41) (1) edge (42) (2) edge (41) (3) edge (42)  (1) edge (51) (2) edge (52) (3) edge (51) (42) edge (52);
\end{tikzpicture}
& 
\begin{tikzpicture}[scale=\sc, thick,main node/.style={circle,  minimum size=1.5mm, fill=black, inner sep=0.1mm,draw,font=\tiny\sffamily}] \placev \placecommone
 \path (1) edge (41) (1) edge (42) (2) edge (41) (3) edge (42)  (1) edge (51) (2) edge (52) (3) edge (51) (41) edge (52);
\end{tikzpicture}
& 
\begin{tikzpicture}[scale=\sc, thick,main node/.style={circle,  minimum size=1.5mm, fill=black, inner sep=0.1mm,draw,font=\tiny\sffamily}] \placev \placecommone
 \path (1) edge (41) (1) edge (42) (2) edge (41) (3) edge (42)  (1) edge (52) (2) edge (51) (3) edge (51) (41) edge (52);
\end{tikzpicture}
& 
\begin{tikzpicture}[scale=\sc, thick,main node/.style={circle,  minimum size=1.5mm, fill=black, inner sep=0.1mm,draw,font=\tiny\sffamily}] \placev \placecommone
 \path (1) edge (41) (1) edge (42) (2) edge (41) (3) edge (42)  (1) edge (51) (1) edge (52)  (2) edge (52) (3) edge (51) (41) edge (51);
\end{tikzpicture}
& 
\begin{tikzpicture}[scale=\sc, thick,main node/.style={circle,  minimum size=1.5mm, fill=black, inner sep=0.1mm,draw,font=\tiny\sffamily}] \placev \placecommone
 \path (1) edge (41) (1) edge (42) (2) edge (41) (3) edge (42)  (1) edge (52)   (2) edge (51) (2) edge (52) (3) edge (51) (42) edge (51);
\end{tikzpicture}
& 
 \begin{tikzpicture}[scale=\sc, thick,main node/.style={circle,  minimum size=1.5mm, fill=black, inner sep=0.1mm,draw,font=\tiny\sffamily}] \placev \placecommone
 \path (1) edge (41) (1) edge (42) (2) edge (41) (3) edge (42)  (1) edge (52)   (2) edge (51) (2) edge (52) (3) edge (51) (41) edge (51);
\end{tikzpicture}\\
3.00709&
3.00688&
3.00682&
3.00512&
3.00493&
3.00483
\\
\end{tabular}
\caption{The $13$ graphs $G \in \hat{\K}_{5,2}$ with $\omega(G) \leq 3$, and their corresponding optimal values of $\max\set{\bar{e}^{\top}x : x \in \LS_+^2(G)}$ according to CVX}\label{fighatK52}
\end{center}
\end{figure}
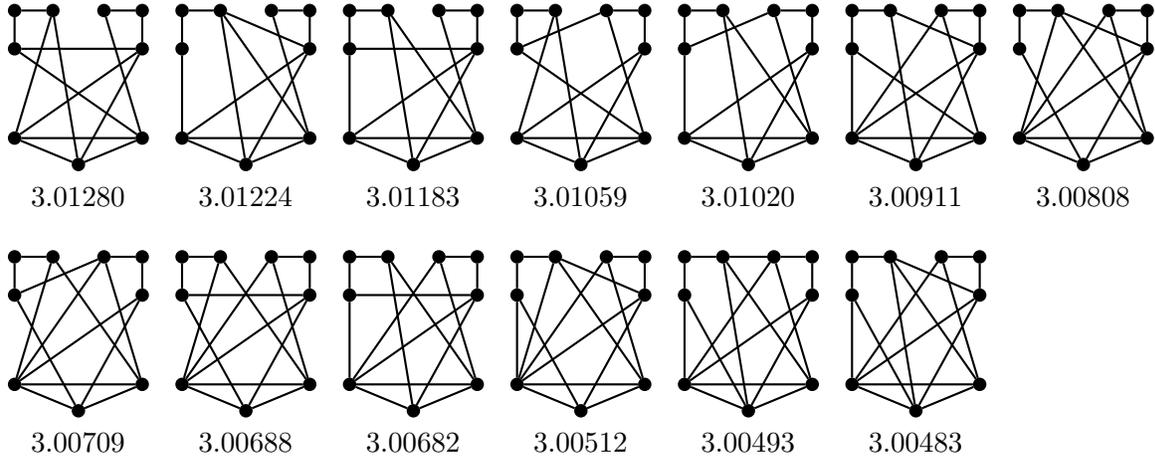

Next, given a graph $G$, we define the \emph{edge density} of $G$ to be $d(G) \ce \frac{|E(G)|}{\binom{ |V(G)| }{2}}$. It is known that the graphs on the two extremes of the edge density spectrum (i.e., the empty graph and the complete graph) both have low $\LS_+$-rank. This raises the natural question of finding the possible range of edge densities among $\ell$-minimal graphs (see~\cite[Problem 34]{AuT24b}). Given a positive integer $\ell$, let $d^+(\ell)$ (resp. $d^-(\ell)$) be the maximum (resp. minimum) edge density among $\ell$-minimal graphs. Our analysis of the graphs $\A_{k,S}$ above implies the following.

\begin{proposition}\label{prop505}
For every positive integer $\ell$, 
\[
d^-(\ell) \leq \frac{\ell^2+7\ell-2}{9\ell^2-3\ell} \quad 
\mbox{ and } \quad
d^+(\ell) \geq \frac{\ell^2+9\ell-4}{9\ell^2-3\ell}.
\]
\end{proposition}

\begin{proof}
Observe that $\A_{k,\es}$ has $\frac{1}{2} (k^2+3k-12)$ edges, and so
\[
d^-(\ell) \leq \frac{ \frac{1}{2} ((\ell+2)^2+3(\ell+2)-12)}{ \binom{3 \ell}{2} } = \frac{\ell^2 + 7\ell -2}{9\ell^2-3\ell}.
\]
Likewise, the bound for $d^+(\ell)$ follows from the fact that $\A_{k, \set{4,\ldots,k}}$ contains $\frac{1}{2}(k^2+5k-18)$ edges.
\end{proof}

The bounds of $d^-(\ell), d^+(\ell)$ from Proposition~\ref{prop505} are tight for $\ell=1$ and $\ell=2$ (again, due to all $1$- and $2$-minimal graphs being known). The bound for $d^-(3)$ is also tight~\cite[Proposition 28]{AuT24b}. However, the bound for $d^+(\ell)$ above (which is based on an $\ell$-minimal graph in $\hat{\K}_{\ell+2, \ell-1}$) is likely not tight in general, as we show below that there does exist $\ell$-minimal graphs outside of $\hat{\K}_{\ell+2,\ell-1}$.

\begin{figure}[htbp]
\begin{center}

\def\sc{2}
\def\x{270 - 360/5}
\def\z{360/5}

\begin{tikzpicture}
[scale=\sc, thick,main node/.style={circle, minimum size=3.8mm, inner sep=0.1mm,draw,font=\tiny\sffamily}]
\node[main node] at ({cos(\x+(0)*\z)},{sin(\x+(0)*\z)}) (1) {$1$};
\node[main node] at ({cos(\x+(1)*\z)},{sin(\x+(1)*\z)}) (2) {$2$};
\node[main node] at ({cos(\x+(2)*\z)},{sin(\x+(2)*\z)}) (3) {$3$};

\node[main node] at ({ \y* cos(\x+(3)*\z) + (1-\y)*cos(\x+(2)*\z)},{ \y* sin(\x+(3)*\z) + (1-\y)*sin(\x+(2)*\z)}) (4) {$4_1$};
\node[main node] at ({cos(\x+(3)*\z)},{sin(\x+(3)*\z)}) (5) {$4_0$};
\node[main node] at ({ \y* cos(\x+(3)*\z) + (1-\y)*cos(\x+(4)*\z)},{ \y* sin(\x+(3)*\z) + (1-\y)*sin(\x+(4)*\z)}) (6) {$4_2$};

\node[main node] at ({ \y* cos(\x+(4)*\z) + (1-\y)*cos(\x+(3)*\z)},{ \y* sin(\x+(4)*\z) + (1-\y)*sin(\x+(3)*\z)}) (7) {$5_1$};
\node[main node] at ({cos(\x+(4)*\z)},{sin(\x+(4)*\z)}) (8) {$5_0$};
\node[main node] at ({ \y* cos(\x+(4)*\z) + (1-\y)*cos(\x+(5)*\z)},{ \y* sin(\x+(4)*\z) + (1-\y)*sin(\x+(5)*\z)}) (9) {$5_2$};

 \path[every node/.style={font=\sffamily}]
(1) edge (2)
(2) edge (3)
(1) edge (3)
(4) edge (5)
(5) edge (6)
(7) edge (8)
(8) edge (9)
(4) edge (3)
(6) edge (1)
(6) edge (2)
(7) edge (2)
(7) edge (3)
(9) edge (1)
(9) edge (2)
(7) edge (6)
(7) edge (4);
\end{tikzpicture}

\caption{A $3$-minimal graph which does not belong to $\hat{\K}_{5,2}$}\label{fig3mindense}
\end{center}
\end{figure}
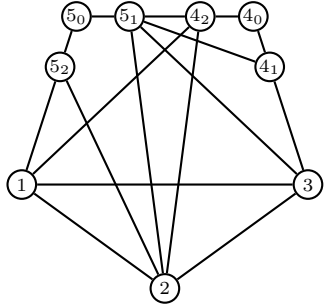

\begin{proposition}\label{prop506}
The graph in Figure~\ref{fig3mindense} is $3$-minimal.
\end{proposition}

\begin{proof}
Let $G$ be the graph in Figure~\ref{fig3mindense}, and we prove our claim using Lemma~\ref{lem403}. Notice that $G \in \tilde{\K}_{5,2}$ with $D(G) = \set{4,5}$. Next, $G \ominus 4_0$ is isomorphic to the $2$-minimal graph $G_{2,2}$ from Figure~\ref{figKnownEG}. Since $G_{2,2} \in \hat{\K}_{4,1}$ and $\omega(G_{2,2}) = 3$, it follows from Theorem~\ref{thm409} that $v(G_{2,2}, \epsilon) \in \cone(\LS_+(G_{2,2}))$ for some $\epsilon > 0$.  Thus, we know that $v(G \ominus 4_0, \epsilon) \in \cone(\LS_+(G \ominus 4_0))$ for some $\epsilon > 0$. Likewise, $G \ominus 5_0$ is isomorphic to $G_{2,1}$, and the same argument shows that $v(G \ominus 5_0, \epsilon) \in \cone(\LS_+(G \ominus 5_0))$ for some $\epsilon > 0$.

\def\sc{1.5}
\def\x{270 - 360/5}
\def\z{360/5}

\begin{figure}[htbp]
\begin{center}
\begin{tabular}{cc}

\begin{tikzpicture}
[scale=\sc, thick,main node/.style={circle,  minimum size=1.5mm, fill=black, inner sep=0.1mm,draw,font=\tiny\sffamily}]
\node[main node] at ({cos(\x+(0)*\z)},{sin(\x+(0)*\z)}) (1) {};
\node[main node] at ({cos(\x+(2)*\z)},{sin(\x+(2)*\z)}) (3) {};

\node[main node] at ({cos(\x+(3)*\z)},{sin(\x+(3)*\z)}) (5) {};

\node[main node] at ({ \y* cos(\x+(4)*\z) + (1-\y)*cos(\x+(3)*\z)},{ \y* sin(\x+(4)*\z) + (1-\y)*sin(\x+(3)*\z)}) (7) {};
\node[main node] at ({cos(\x+(4)*\z)},{sin(\x+(4)*\z)}) (8) {};
\node[main node] at ({ \y* cos(\x+(4)*\z) + (1-\y)*cos(\x+(5)*\z)},{ \y* sin(\x+(4)*\z) + (1-\y)*sin(\x+(5)*\z)}) (9) {};

 \path[every node/.style={font=\sffamily}]
(1) edge (3)
(7) edge (8)
(8) edge (9)
(5) edge (3)
(5) edge (1)
(7) edge (3)
(9) edge (1)
(7) edge (5);
\end{tikzpicture}
&

\begin{tikzpicture}
[scale=\sc, thick,main node/.style={circle,  minimum size=1.5mm, fill=black, inner sep=0.1mm,draw,font=\tiny\sffamily}]
\node[main node] at ({cos(\x+(0)*\z)},{sin(\x+(0)*\z)}) (1) {};
\node[main node] at ({cos(\x+(2)*\z)},{sin(\x+(2)*\z)}) (3) {};

\node[main node] at ({ \y* cos(\x+(3)*\z) + (1-\y)*cos(\x+(2)*\z)},{ \y* sin(\x+(3)*\z) + (1-\y)*sin(\x+(2)*\z)}) (4) {};
\node[main node] at ({cos(\x+(3)*\z)},{sin(\x+(3)*\z)}) (5) {};
\node[main node] at ({ \y* cos(\x+(3)*\z) + (1-\y)*cos(\x+(4)*\z)},{ \y* sin(\x+(3)*\z) + (1-\y)*sin(\x+(4)*\z)}) (6) {};

\node[main node] at ({ \y* cos(\x+(4)*\z) + (1-\y)*cos(\x+(3)*\z)},{ \y* sin(\x+(4)*\z) + (1-\y)*sin(\x+(3)*\z)}) (7) {};
\node[main node] at ({cos(\x+(4)*\z)},{sin(\x+(4)*\z)}) (8) {};
\node[main node] at ({ \y* cos(\x+(4)*\z) + (1-\y)*cos(\x+(5)*\z)},{ \y* sin(\x+(4)*\z) + (1-\y)*sin(\x+(5)*\z)}) (9) {};

 \path[every node/.style={font=\sffamily}]
(1) edge (3)
(4) edge (5)
(5) edge (6)
(7) edge (8)
(8) edge (9)
(4) edge (3)
(6) edge (1)
(7) edge (3)
(9) edge (1)
(7) edge (6)
(7) edge (4);
\end{tikzpicture}
\\
$G_{2,1}$ & $G - 2$
\end{tabular}
\caption{Illustrating the proof of Proposition~\ref{prop506}}\label{fig3mindenseproof}
\end{center}
\end{figure}
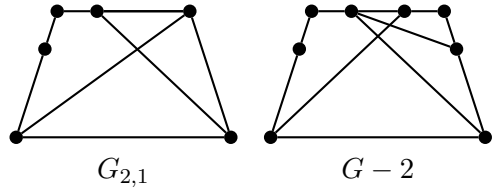

Next, observe that we can $2$-stretch a vertex in $G_{2,1}$ to obtain a graph isomorphic to $G - 2$ (see Figure~\ref{fig3mindenseproof}). Thus, Lemma~\ref{lem404} implies that $v(G-2, \epsilon) \in \cone(\LS_+(G-2))$ for some $\epsilon > 0$. Similarly, observe that $G-1$ and $G-3$ can be obtained by stretching a vertex in $G_{2,1}$ and $G_{2,2}$, respectively. Thus, using the same rationale as above, we conclude that there exists $\epsilon > 0$ where $v(G - i, \epsilon) \in \cone(\LS_+(G-i))$ for every $i \in [n] \setminus D(G)$.

Hence, Lemma~\ref{lem403} applies, and we conclude that $v(G,\epsilon) \in \cone(\LS_+^2(G))$ for some $\epsilon > 0$, which implies that $r_+(G) \geq 3$. Finally, $r_+(G) \leq 3$ follows from $|V(G)| =9$ and Theorem~\ref{thm101}, and this finishes the proof.
\end{proof}

The graph in Figure~\ref{fig3mindense} provides what we believe is the first example of an $\ell$-minimal graph that does not belong to $\hat{\K}_{\ell+2, \ell-1}$. Moreover, it is  not the only such graph. Figure~\ref{figK52} lists the $25$ non-isomorphic graphs in $\K_{5,2} \setminus \hat{\K}_{5,2}$ with clique number at most $3$, as well as their corresponding optimal value of $\max\set{ \bar{e}^{\top}x : x \in \LS_+^2(G)}$ computed in CVX. The computational results suggest that $18$ of these graphs are indeed $3$-minimal, which was subsequently verified in~\cite{AuT26}. On the other hand, the remaining $7$ graphs have an optimal value that is very close to $3$, which seems to indicate that the facet-inducing inequality $\bar{e}^{\top}x \leq 3$ has $\LS_+$-rank $2$ in those cases.

\def\x{270 - 135}
\def\z{360/4}
\def\y{0.7}
\def\sc{1.2}

\def\placev{\node[main node] at ({cos(\x+(1)*\z)},{sin(\x+(1)*\z)}) (1) {}; \node[main node] at ({cos(\x+(1.5)*\z)},{sin(\x+(1.5)*\z)}) (2) {}; \node[main node] at ({cos(\x+(2)*\z)},{sin(\x+(2)*\z)}) (3) {}; \node[main node] at ({ \y* cos(\x+(3)*\z) + (1-\y)*cos(\x+(2)*\z)},{ \y* sin(\x+(3)*\z) + (1-\y)*sin(\x+(2)*\z)}) (41) {}; \node[main node] at ({cos(\x+(3)*\z)},{sin(\x+(3)*\z)}) (40) {}; \node[main node] at ({ \y* cos(\x+(3)*\z) + (1-\y)*cos(\x+(4)*\z)},{ \y* sin(\x+(3)*\z) + (1-\y)*sin(\x+(4)*\z)}) (42) {}; \node[main node] at ({ \y* cos(\x+(4)*\z) + (1-\y)*cos(\x+(3)*\z)},{ \y* sin(\x+(4)*\z) + (1-\y)*sin(\x+(3)*\z)}) (51) {}; \node[main node] at ({cos(\x+(4)*\z)},{sin(\x+(4)*\z)}) (50) {}; \node[main node] at ({ \y* cos(\x+(4)*\z) + (1-\y)*cos(\x+(5)*\z)},{ \y* sin(\x+(4)*\z) + (1-\y)*sin(\x+(5)*\z)}) (52) {};}

\def\placecommone{  \path (1) edge (2) (2) edge (3) (3) edge (1) (40) edge (41) (40) edge (42) (50) edge (51) (50) edge (52);}

\begin{figure}[htbp]
\begin{center}
\begin{tabular}{ccccccc}
\begin{tikzpicture}[scale=\sc, thick,main node/.style={circle,  minimum size=1.5mm, fill=black, inner sep=0.1mm,draw,font=\tiny\sffamily}] \placev \placecommone
 \path (1) edge (41) (2) edge (41) (3) edge (42)  (1) edge (51) (2) edge (51) (3) edge (52) (41) edge (52) (42) edge (52);
\end{tikzpicture}
& 
\begin{tikzpicture}[scale=\sc, thick,main node/.style={circle,  minimum size=1.5mm, fill=black, inner sep=0.1mm,draw,font=\tiny\sffamily}] \placev \placecommone
 \path (1) edge (41) (2) edge (41) (3) edge (42)  (1) edge (52) (2) edge (51) (3) edge (51) (41) edge (52) (42) edge (52);
\end{tikzpicture}
& 
\begin{tikzpicture}[scale=\sc, thick,main node/.style={circle,  minimum size=1.5mm, fill=black, inner sep=0.1mm,draw,font=\tiny\sffamily}] \placev \placecommone
 \path (1) edge (41)  (2) edge (41) (3) edge (42)  (1) edge (52) (2) edge (51) (3) edge (51) (41) edge (52) (42) edge (51);
\end{tikzpicture}
& 
\begin{tikzpicture}[scale=\sc, thick,main node/.style={circle,  minimum size=1.5mm, fill=black, inner sep=0.1mm,draw,font=\tiny\sffamily}] \placev \placecommone
 \path (1) edge (41)  (2) edge (41) (3) edge (42)  (1) edge (52) (2) edge (51) (3) edge (51) (41) edge (51) (41) edge (52) (42) edge (51);
\end{tikzpicture}
& 
\begin{tikzpicture}[scale=\sc, thick,main node/.style={circle,  minimum size=1.5mm, fill=black, inner sep=0.1mm,draw,font=\tiny\sffamily}] \placev \placecommone
 \path (1) edge (41) (2) edge (41) (3) edge (42)  (1) edge (52) (2) edge (51) (3) edge (51) (41) edge (51) (42) edge (52);
\end{tikzpicture}
& 
\begin{tikzpicture}[scale=\sc, thick,main node/.style={circle,  minimum size=1.5mm, fill=black, inner sep=0.1mm,draw,font=\tiny\sffamily}] \placev \placecommone
 \path (1) edge (41) (2) edge (41) (3) edge (42)  (1) edge (52) (2) edge (51) (3) edge (51) (41) edge (51) (41) edge (52) (42) edge (52);
\end{tikzpicture}
&
\begin{tikzpicture}[scale=\sc, thick,main node/.style={circle,  minimum size=1.5mm, fill=black, inner sep=0.1mm,draw,font=\tiny\sffamily}] \placev \placecommone
 \path (1) edge (41) (2) edge (41) (3) edge (42)  (1) edge (52) (2) edge (51) (3) edge (51) (41) edge (52) (42) edge (51) (42) edge (52);
\end{tikzpicture}
\\
3.01029&
3.00971& 
3.00897 &
3.00896 &
3.00871 &
3.00868 &
3.00863\\
\\
\begin{tikzpicture}[scale=\sc, thick,main node/.style={circle,  minimum size=1.5mm, fill=black, inner sep=0.1mm,draw,font=\tiny\sffamily}] \placev \placecommone
 \path (1) edge (41) (2) edge (41) (3) edge (42)  (1) edge (52) (2) edge (51) (3) edge (51) (41) edge (51) (41) edge (52) (42) edge (51) (42) edge (52);
\end{tikzpicture}
&  
\begin{tikzpicture}[scale=\sc, thick,main node/.style={circle,  minimum size=1.5mm, fill=black, inner sep=0.1mm,draw,font=\tiny\sffamily}] \placev \placecommone
 \path (1) edge (41) (1) edge (42) (2) edge (41) (3) edge (42)  (1) edge (52) (2) edge (51) (3) edge (51) (41) edge (51) (41) edge (52);
\end{tikzpicture}
& 
\begin{tikzpicture}[scale=\sc, thick,main node/.style={circle,  minimum size=1.5mm, fill=black, inner sep=0.1mm,draw,font=\tiny\sffamily}] \placev \placecommone
 \path (1) edge (41) (1) edge (42) (2) edge (41) (3) edge (42)  (1) edge (51) (2) edge (52) (3) edge (51) (41) edge (52) (42) edge (52);
\end{tikzpicture}
& 
\begin{tikzpicture}[scale=\sc, thick,main node/.style={circle,  minimum size=1.5mm, fill=black, inner sep=0.1mm,draw,font=\tiny\sffamily}] \placev \placecommone
 \path (1) edge (41) (1) edge (42) (2) edge (41) (3) edge (42)  (1) edge (52) (2) edge (51) (3) edge (51) (41) edge (52) (42) edge (52);
\end{tikzpicture}
& 
\begin{tikzpicture}[scale=\sc, thick,main node/.style={circle,  minimum size=1.5mm, fill=black, inner sep=0.1mm,draw,font=\tiny\sffamily}] \placev \placecommone
 \path (1) edge (41) (1) edge (42) (2) edge (41) (3) edge (42)  (1) edge (52) (2) edge (51) (3) edge (51) (41) edge (51) (41) edge (52);
\end{tikzpicture}
& 
\begin{tikzpicture}[scale=\sc, thick,main node/.style={circle,  minimum size=1.5mm, fill=black, inner sep=0.1mm,draw,font=\tiny\sffamily}] \placev \placecommone
 \path (1) edge (41) (1) edge (42) (2) edge (41) (3) edge (42)  (1) edge (51) (2) edge (52) (3) edge (51) (41) edge (51) (41) edge (52);
\end{tikzpicture}
& 
\begin{tikzpicture}[scale=\sc, thick,main node/.style={circle,  minimum size=1.5mm, fill=black, inner sep=0.1mm,draw,font=\tiny\sffamily}] \placev \placecommone
 \path (1) edge (41) (1) edge (42) (2) edge (41) (3) edge (42)  (1) edge (52) (2) edge (51) (3) edge (51) (41) edge (51) (42) edge (52);
\end{tikzpicture}
\\
3.00863&
3.00727&
3.00657&
3.00635&
3.00627&
3.00615&
3.00605\\
\\
 \begin{tikzpicture}[scale=\sc, thick,main node/.style={circle,  minimum size=1.5mm, fill=black, inner sep=0.1mm,draw,font=\tiny\sffamily}] \placev \placecommone
 \path (1) edge (41) (1) edge (42) (2) edge (41) (3) edge (42)  (1) edge (52) (2) edge (51) (3) edge (51) (41) edge (51) (41) edge (52) (42) edge (51);
\end{tikzpicture}
&  
\begin{tikzpicture}[scale=\sc, thick,main node/.style={circle,  minimum size=1.5mm, fill=black, inner sep=0.1mm,draw,font=\tiny\sffamily}] \placev \placecommone
 \path (1) edge (41) (1) edge (42) (2) edge (41) (3) edge (42)  (1) edge (52) (2) edge (51) (3) edge (51) (41) edge (51) (41) edge (52) (42) edge (52);
\end{tikzpicture}
&  
\begin{tikzpicture}[scale=\sc, thick,main node/.style={circle,  minimum size=1.5mm, fill=black, inner sep=0.1mm,draw,font=\tiny\sffamily}] \placev \placecommone
 \path (1) edge (41) (1) edge (42) (2) edge (41) (3) edge (42)  (1) edge (52) (2) edge (51) (3) edge (51) (41) edge (51) (41) edge (52) (42) edge (51) (42) edge (52);
\end{tikzpicture}
&  
\begin{tikzpicture}[scale=\sc, thick,main node/.style={circle,  minimum size=1.5mm, fill=black, inner sep=0.1mm,draw,font=\tiny\sffamily}] \placev \placecommone
 \path (1) edge (41) (1) edge (42) (2) edge (41) (3) edge (42)  (1) edge (52)   (2) edge (51) (2) edge (52) (3) edge (51) (41) edge (51) (42) edge (51);
\end{tikzpicture}
& 
\begin{tikzpicture}[scale=\sc, thick,main node/.style={circle,  minimum size=1.5mm, fill=black, inner sep=0.1mm,draw,font=\tiny\sffamily}] \placev \placecommone
 \path (1) edge (41) (1) edge (42) (2) edge (41) (3) edge (42)  (1) edge (52)   (2) edge (51) (2) edge (52) (3) edge (51) (41) edge (51) (42) edge (52);
\end{tikzpicture}
& 
\begin{tikzpicture}[scale=\sc, thick,main node/.style={circle,  minimum size=1.5mm, fill=black, inner sep=0.1mm,draw,font=\tiny\sffamily}] \placev \placecommone
 \path (1) edge (41) (1) edge (42) (2) edge (41) (3) edge (42)  (1) edge (51) (1) edge (52)  (2) edge (52) (3) edge (51) (41) edge (51) (42) edge (52);
\end{tikzpicture}
& 
\begin{tikzpicture}[scale=\sc, thick,main node/.style={circle,  minimum size=1.5mm, fill=black, inner sep=0.1mm,draw,font=\tiny\sffamily}] \placev \placecommone
 \path (1) edge (41) (1) edge (42) (2) edge (41) (3) edge (42)  (1) edge (52)   (2) edge (51) (2) edge (52) (3) edge (51) (41) edge (51) (42) edge (51) (42) edge (52);
\end{tikzpicture}

\\
3.00577&
3.00571&
3.00560&
3.00483&
3.00000&
3.00000&
3.00000
\\
\\
\begin{tikzpicture}[scale=\sc, thick,main node/.style={circle,  minimum size=1.5mm, fill=black, inner sep=0.1mm,draw,font=\tiny\sffamily}] \placev \placecommone
 \path (1) edge (41) (1) edge (42) (2) edge (41) (3) edge (42)  (1) edge (51) (2) edge (52) (3) edge (51) (41) edge (51) (42) edge (52);
\end{tikzpicture}
& 
\begin{tikzpicture}[scale=\sc, thick,main node/.style={circle,  minimum size=1.5mm, fill=black, inner sep=0.1mm,draw,font=\tiny\sffamily}] \placev \placecommone
 \path (1) edge (41) (1) edge (42) (2) edge (41) (3) edge (42)  (1) edge (51) (2) edge (52) (3) edge (51) (41) edge (51) (41) edge (52) (42) edge (52);
\end{tikzpicture}
& 
 \begin{tikzpicture}[scale=\sc, thick,main node/.style={circle,  minimum size=1.5mm, fill=black, inner sep=0.1mm,draw,font=\tiny\sffamily}] \placev \placecommone
 \path (1) edge (41) (2) edge (41) (3) edge (42)  (1) edge (52)   (2) edge (51)  (3) edge (51) (41) edge (52) (42) edge (51) (42) edge (52);
\end{tikzpicture}
& 
 \begin{tikzpicture}[scale=\sc, thick,main node/.style={circle,  minimum size=1.5mm, fill=black, inner sep=0.1mm,draw,font=\tiny\sffamily}] \placev \placecommone
 \path (1) edge (41) (2) edge (41) (3) edge (42)  (1) edge (52)   (2) edge (51)  (3) edge (51) (41) edge (52) (42) edge (51);
\end{tikzpicture}
\\
3.00000 &
3.00000 &
3.00000 &
3.00000
\\
\end{tabular}
\caption{The $25$ graphs $G \in \K_{5,2} \setminus \hat{\K}_{5,2}$ with $\omega(G) \leq 3$, and their corresponding optimal values of $\max\set{\bar{e}^{\top}x : x \in \LS_+^2(G)}$ according to CVX}\label{figK52}
\end{center}
\end{figure}
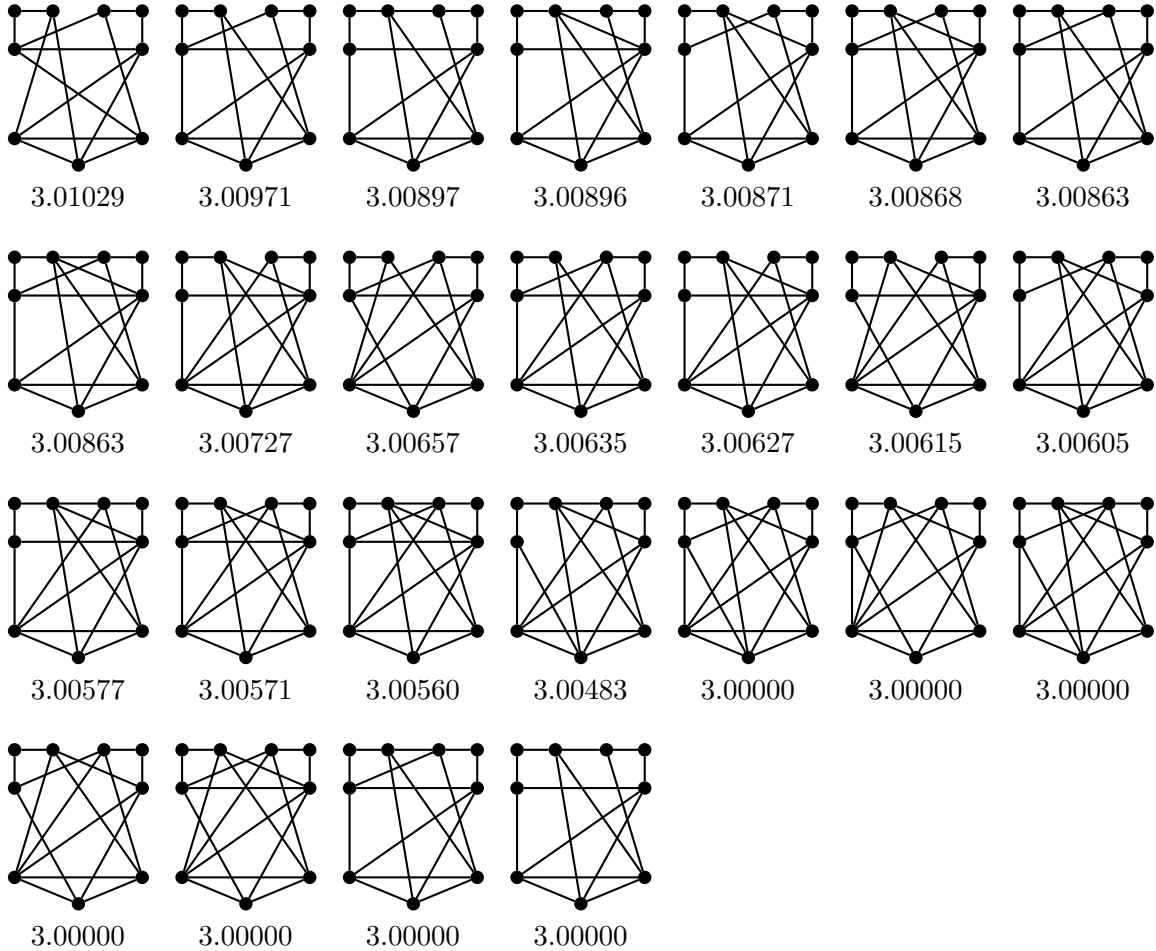

Finally, we finish this section by describing a family of relatively small vertex-transitive graphs with high $\LS_+$-rank. Given integers $a,b$ and $n \geq 1$, we define $a +_n b$ to be the unique integer $c \in [n]$ where $a+b-c$ is a multiple of $n$. (I.e., $a+_n b$ works similarly to addition modulo-$n$, except the operation outputs $n$ instead of $0$ when $a+b$ is divisible by $n$.) We also define $a -_n b$ analogously. Then, given an odd integer $k \geq 3$, we define the graph $\B_k$ where
\begin{align*}
V(\B_k) \ce{}& \set{ i_0, i_1, i_2, i_3 : i \in [k]}\\
E(\B_k) \ce{}& \set{ \set{i_0, i_1}, \set{i_1, i_2}, \set{i_2, i_3}, \set{i_3, i_0} : i \in [k]} \cup \\
& \set{ \set{i_0,j_2}, \set{i_1,j_3} : (j -_k i) \in \set{1,2,\ldots, \frac{k-1}{2}} }.
\end{align*}

Figure~\ref{figBk} illustrates the first few members of the family of graphs $\B_k$. Then we have the following.

\def\sc{2}
\def\w{1.1}
\def\y{15}

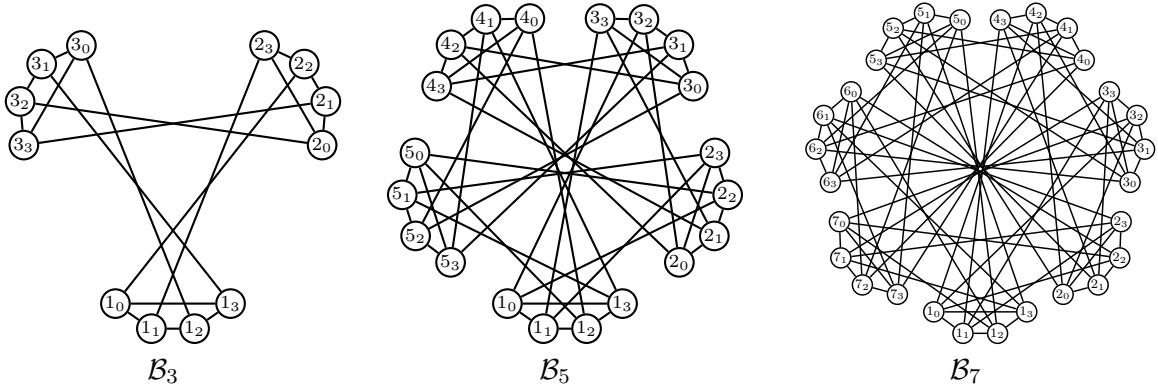
\begin{figure}[htbp]
\begin{center}
\begin{tabular}{ccc}

\def\x{120}
\def\z{270-1.5*\y}

\begin{tikzpicture}
[scale=\sc, thick,main node/.style={circle, minimum size=3.8mm, inner sep=0.1mm,draw,font=\tiny\sffamily}]

\node[main node] at ({cos(\z+0*\x+0*\y)}, {sin(\z+0*\x+0*\y)})  (10) {$1_0$};
\node[main node] at ({cos(\z+1*\x+0*\y)}, {sin(\z+1*\x+0*\y)})  (20) {$2_0$};
\node[main node] at ({cos(\z+2*\x+0*\y)}, {sin(\z+2*\x+0*\y)})  (30) {$3_0$};
\node[main node] at ({\w*cos(\z+0*\x+1*\y)}, {\w*sin(\z+0*\x+1*\y)})  (11) {$1_1$};
\node[main node] at ({\w*cos(\z+1*\x+1*\y)}, {\w*sin(\z+1*\x+1*\y)})  (21) {$2_1$};
\node[main node] at ({\w*cos(\z+2*\x+1*\y)}, {\w*sin(\z+2*\x+1*\y)})  (31) {$3_1$};
\node[main node] at ({\w*cos(\z+0*\x+2*\y)}, {\w*sin(\z+0*\x+2*\y)})  (12) {$1_2$};
\node[main node] at ({\w*cos(\z+1*\x+2*\y)}, {\w*sin(\z+1*\x+2*\y)})  (22) {$2_2$};
\node[main node] at ({\w*cos(\z+2*\x+2*\y)}, {\w*sin(\z+2*\x+2*\y)})  (32) {$3_2$};
\node[main node] at ({cos(\z+0*\x+3*\y)}, {sin(\z+0*\x+3*\y)})  (13) {$1_3$};
\node[main node] at ({cos(\z+1*\x+3*\y)}, {sin(\z+1*\x+3*\y)})  (23) {$2_3$};
\node[main node] at ({cos(\z+2*\x+3*\y)}, {sin(\z+2*\x+3*\y)})  (33) {$3_3$};

 \path[every node/.style={font=\sffamily}]
(10) edge (11)
(20) edge (21)
(30) edge (31)
(11) edge (12)
(21) edge (22)
(31) edge (32)
(12) edge (13)
(22) edge (23)
(32) edge (33)
(13) edge (10)
(23) edge (20)
(33) edge (30)
(10) edge (22)
(11) edge (23)
(20) edge (32)
(21) edge (33)
(30) edge (12)
(31) edge (13);
\end{tikzpicture}

&

\def\x{72}
\def\z{270-1.5*\y}

\begin{tikzpicture}
[scale=\sc, thick,main node/.style={circle, minimum size=3.8mm, inner sep=0.1mm,draw,font=\tiny\sffamily}]

\node[main node] at ({cos(\z+0*\x+0*\y)}, {sin(\z+0*\x+0*\y)})  (10) {$1_0$};
\node[main node] at ({cos(\z+1*\x+0*\y)}, {sin(\z+1*\x+0*\y)})  (20) {$2_0$};
\node[main node] at ({cos(\z+2*\x+0*\y)}, {sin(\z+2*\x+0*\y)})  (30) {$3_0$};
\node[main node] at ({cos(\z+3*\x+0*\y)}, {sin(\z+3*\x+0*\y)})  (40) {$4_0$};
\node[main node] at ({cos(\z+4*\x+0*\y)}, {sin(\z+4*\x+0*\y)})  (50) {$5_0$};
\node[main node] at ({\w*cos(\z+0*\x+1*\y)}, {\w*sin(\z+0*\x+1*\y)})  (11) {$1_1$};
\node[main node] at ({\w*cos(\z+1*\x+1*\y)}, {\w*sin(\z+1*\x+1*\y)})  (21) {$2_1$};
\node[main node] at ({\w*cos(\z+2*\x+1*\y)}, {\w*sin(\z+2*\x+1*\y)})  (31) {$3_1$};
\node[main node] at ({\w*cos(\z+3*\x+1*\y)}, {\w*sin(\z+3*\x+1*\y)})  (41) {$4_1$};
\node[main node] at ({\w*cos(\z+4*\x+1*\y)}, {\w*sin(\z+4*\x+1*\y)})  (51) {$5_1$};
\node[main node] at ({\w*cos(\z+0*\x+2*\y)}, {\w*sin(\z+0*\x+2*\y)})  (12) {$1_2$};
\node[main node] at ({\w*cos(\z+1*\x+2*\y)}, {\w*sin(\z+1*\x+2*\y)})  (22) {$2_2$};
\node[main node] at ({\w*cos(\z+2*\x+2*\y)}, {\w*sin(\z+2*\x+2*\y)})  (32) {$3_2$};
\node[main node] at ({\w*cos(\z+3*\x+2*\y)}, {\w*sin(\z+3*\x+2*\y)})  (42) {$4_2$};
\node[main node] at ({\w*cos(\z+4*\x+2*\y)}, {\w*sin(\z+4*\x+2*\y)})  (52) {$5_2$};
\node[main node] at ({cos(\z+0*\x+3*\y)}, {sin(\z+0*\x+3*\y)})  (13) {$1_3$};
\node[main node] at ({cos(\z+1*\x+3*\y)}, {sin(\z+1*\x+3*\y)})  (23) {$2_3$};
\node[main node] at ({cos(\z+2*\x+3*\y)}, {sin(\z+2*\x+3*\y)})  (33) {$3_3$};
\node[main node] at ({cos(\z+3*\x+3*\y)}, {sin(\z+3*\x+3*\y)})  (43) {$4_3$};
\node[main node] at ({cos(\z+4*\x+3*\y)}, {sin(\z+4*\x+3*\y)})  (53) {$5_3$};

 \path[every node/.style={font=\sffamily}]
(10) edge (11)
(20) edge (21)
(30) edge (31)
(40) edge (41)
(50) edge (51)
(11) edge (12)
(21) edge (22)
(31) edge (32)
(41) edge (42)
(51) edge (52)
(12) edge (13)
(22) edge (23)
(32) edge (33)
(42) edge (43)
(52) edge (53)
(13) edge (10)
(23) edge (20)
(33) edge (30)
(43) edge (40)
(53) edge (50)
(10) edge (22)
(10) edge (32)
(11) edge (23)
(11) edge (33)
(20) edge (32)
(20) edge (42)
(21) edge (33)
(21) edge (43)
(30) edge (42)
(30) edge (52)
(31) edge (43)
(31) edge (53)
(40) edge (52)
(40) edge (12)
(41) edge (53)
(41) edge (13)
(50) edge (12)
(50) edge (22)
(51) edge (13)
(51) edge (23);
\end{tikzpicture}

&

\def\x{360/7}
\def\z{270-1.5*\y}
\def\y{12}
\scalebox{0.7}{
\begin{tikzpicture}
[scale={\sc / 0.7}, thick,main node/.style={circle, minimum size=3.8mm, inner sep=0.1mm,draw,font=\tiny\sffamily}]

\node[main node] at ({cos(\z+0*\x+0*\y)}, {sin(\z+0*\x+0*\y)})  (10) {$1_0$};
\node[main node] at ({cos(\z+1*\x+0*\y)}, {sin(\z+1*\x+0*\y)})  (20) {$2_0$};
\node[main node] at ({cos(\z+2*\x+0*\y)}, {sin(\z+2*\x+0*\y)})  (30) {$3_0$};
\node[main node] at ({cos(\z+3*\x+0*\y)}, {sin(\z+3*\x+0*\y)})  (40) {$4_0$};
\node[main node] at ({cos(\z+4*\x+0*\y)}, {sin(\z+4*\x+0*\y)})  (50) {$5_0$};
\node[main node] at ({cos(\z+5*\x+0*\y)}, {sin(\z+5*\x+0*\y)})  (60) {$6_0$};
\node[main node] at ({cos(\z+6*\x+0*\y)}, {sin(\z+6*\x+0*\y)})  (70) {$7_0$};

\node[main node] at ({\w*cos(\z+0*\x+1*\y)}, {\w*sin(\z+0*\x+1*\y)})  (11) {$1_1$};
\node[main node] at ({\w*cos(\z+1*\x+1*\y)}, {\w*sin(\z+1*\x+1*\y)})  (21) {$2_1$};
\node[main node] at ({\w*cos(\z+2*\x+1*\y)}, {\w*sin(\z+2*\x+1*\y)})  (31) {$3_1$};
\node[main node] at ({\w*cos(\z+3*\x+1*\y)}, {\w*sin(\z+3*\x+1*\y)})  (41) {$4_1$};
\node[main node] at ({\w*cos(\z+4*\x+1*\y)}, {\w*sin(\z+4*\x+1*\y)})  (51) {$5_1$};
\node[main node] at ({\w*cos(\z+5*\x+1*\y)}, {\w*sin(\z+5*\x+1*\y)})  (61) {$6_1$};
\node[main node] at ({\w*cos(\z+6*\x+1*\y)}, {\w*sin(\z+6*\x+1*\y)})  (71) {$7_1$};

\node[main node] at ({\w*cos(\z+0*\x+2*\y)}, {\w*sin(\z+0*\x+2*\y)})  (12) {$1_2$};
\node[main node] at ({\w*cos(\z+1*\x+2*\y)}, {\w*sin(\z+1*\x+2*\y)})  (22) {$2_2$};
\node[main node] at ({\w*cos(\z+2*\x+2*\y)}, {\w*sin(\z+2*\x+2*\y)})  (32) {$3_2$};
\node[main node] at ({\w*cos(\z+3*\x+2*\y)}, {\w*sin(\z+3*\x+2*\y)})  (42) {$4_2$};
\node[main node] at ({\w*cos(\z+4*\x+2*\y)}, {\w*sin(\z+4*\x+2*\y)})  (52) {$5_2$};
\node[main node] at ({\w*cos(\z+5*\x+2*\y)}, {\w*sin(\z+5*\x+2*\y)})  (62) {$6_2$};
\node[main node] at ({\w*cos(\z+6*\x+2*\y)}, {\w*sin(\z+6*\x+2*\y)})  (72) {$7_2$};

\node[main node] at ({cos(\z+0*\x+3*\y)}, {sin(\z+0*\x+3*\y)})  (13) {$1_3$};
\node[main node] at ({cos(\z+1*\x+3*\y)}, {sin(\z+1*\x+3*\y)})  (23) {$2_3$};
\node[main node] at ({cos(\z+2*\x+3*\y)}, {sin(\z+2*\x+3*\y)})  (33) {$3_3$};
\node[main node] at ({cos(\z+3*\x+3*\y)}, {sin(\z+3*\x+3*\y)})  (43) {$4_3$};
\node[main node] at ({cos(\z+4*\x+3*\y)}, {sin(\z+4*\x+3*\y)})  (53) {$5_3$};
\node[main node] at ({cos(\z+5*\x+3*\y)}, {sin(\z+5*\x+3*\y)})  (63) {$6_3$};
\node[main node] at ({cos(\z+6*\x+3*\y)}, {sin(\z+6*\x+3*\y)})  (73) {$7_3$};

 \path[every node/.style={font=\sffamily}]
(10) edge (11)
(20) edge (21)
(30) edge (31)
(40) edge (41)
(50) edge (51)
(11) edge (12)
(21) edge (22)
(31) edge (32)
(41) edge (42)
(51) edge (52)
(12) edge (13)
(22) edge (23)
(32) edge (33)
(42) edge (43)
(52) edge (53)
(13) edge (10)
(23) edge (20)
(33) edge (30)
(43) edge (40)
(53) edge (50)
(60) edge (61)
(61) edge (62)
(62) edge (63)
(63) edge (60)
(70) edge (71)
(71) edge (72)
(72) edge (73)
(73) edge (70)
(10) edge (22)
(10) edge (32)
(10) edge (42)
(11) edge (23)
(11) edge (33)
(11) edge (43)
(20) edge (32)
(20) edge (42)
(20) edge (52)
(21) edge (33)
(21) edge (43)
(21) edge (53)
(30) edge (42)
(30) edge (52)
(30) edge (62)
(31) edge (43)
(31) edge (53)
(31) edge (63)
(40) edge (52)
(40) edge (62)
(40) edge (72)
(41) edge (53)
(41) edge (63)
(41) edge (73)
(50) edge (62)
(50) edge (72)
(50) edge (12)
(51) edge (63)
(51) edge (73)
(51) edge (13)
(60) edge (72)
(60) edge (12)
(60) edge (22)
(61) edge (73)
(61) edge (13)
(61) edge (23)
(70) edge (12)
(70) edge (22)
(70) edge (32)
(71) edge (13)
(71) edge (23)
(71) edge (33);
\end{tikzpicture}
}
\\
$\B_3$ & $\B_5$ & $\B_7$
\end{tabular}
\caption{Illustrating the graphs $\B_k$}\label{figBk}
\end{center}
\end{figure}

\begin{proposition}\label{prop508}
For every odd integer $k \geq 3$, $\B_k$ is vertex-transitive and $r_+(\B_k) \geq k-2$.
\end{proposition}

\begin{proof}
First, let  $\B_k'$ denote the graph obtained from $\B_k$ by removing the vertices $\set{i_3 : i \in [k]}$. We see that $\B_k' \in \hat{\K}_{k,k}$, with $\omega(\B_k') = 2$. (In fact, for every $j \in \set{0,1,2}$, removing all vertices in $\set{ i_j : i \in [k]}$ from $\B_k$ also results in a graph isomorphic to $\B_k'$.) Hence, it follows from Theorem~\ref{thm409} that $r_+(\B_k') \geq k-2$. Since $\B_k'$ is an induced subgraph of $\B_k$, it follows that $r_+(\B_k) \geq k-2$.

\def\sc{2.1}
\def\w{1}
\def\y{15}

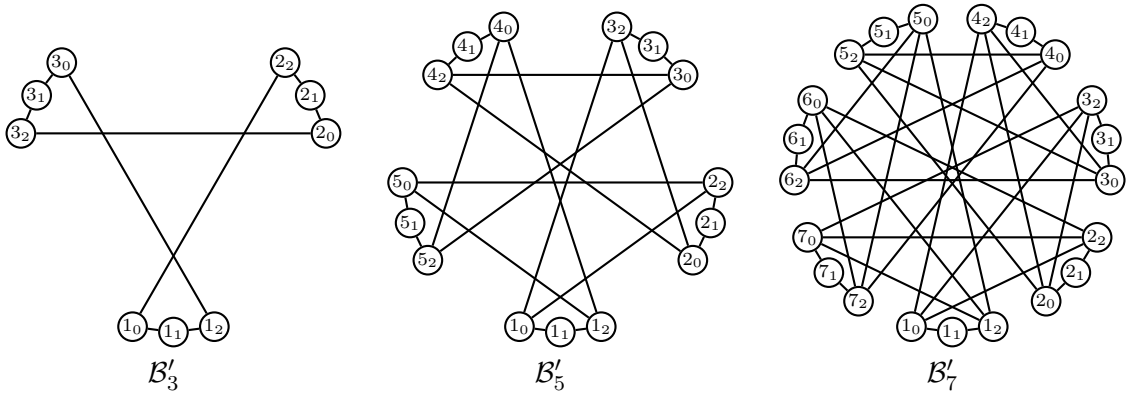
\begin{figure}[htbp]
\begin{center}
\begin{tabular}{ccc}

\def\x{120}
\def\z{270-1*\y}

\begin{tikzpicture}
[scale=\sc, thick,main node/.style={circle, minimum size=3.8mm, inner sep=0.1mm,draw,font=\tiny\sffamily}]

\node[main node] at ({cos(\z+0*\x+0*\y)}, {sin(\z+0*\x+0*\y)})  (10) {$1_0$};
\node[main node] at ({cos(\z+1*\x+0*\y)}, {sin(\z+1*\x+0*\y)})  (20) {$2_0$};
\node[main node] at ({cos(\z+2*\x+0*\y)}, {sin(\z+2*\x+0*\y)})  (30) {$3_0$};
\node[main node] at ({\w*cos(\z+0*\x+1*\y)}, {\w*sin(\z+0*\x+1*\y)})  (11) {$1_1$};
\node[main node] at ({\w*cos(\z+1*\x+1*\y)}, {\w*sin(\z+1*\x+1*\y)})  (21) {$2_1$};
\node[main node] at ({\w*cos(\z+2*\x+1*\y)}, {\w*sin(\z+2*\x+1*\y)})  (31) {$3_1$};
\node[main node] at ({\w*cos(\z+0*\x+2*\y)}, {\w*sin(\z+0*\x+2*\y)})  (12) {$1_2$};
\node[main node] at ({\w*cos(\z+1*\x+2*\y)}, {\w*sin(\z+1*\x+2*\y)})  (22) {$2_2$};
\node[main node] at ({\w*cos(\z+2*\x+2*\y)}, {\w*sin(\z+2*\x+2*\y)})  (32) {$3_2$};

 \path[every node/.style={font=\sffamily}]
(10) edge (11)
(20) edge (21)
(30) edge (31)
(11) edge (12)
(21) edge (22)
(31) edge (32)
(10) edge (22)
(20) edge (32)
(30) edge (12);
\end{tikzpicture}

&

\def\x{72}
\def\z{270-1*\y}

\begin{tikzpicture}
[scale=\sc, thick,main node/.style={circle, minimum size=3.8mm, inner sep=0.1mm,draw,font=\tiny\sffamily}]

\node[main node] at ({cos(\z+0*\x+0*\y)}, {sin(\z+0*\x+0*\y)})  (10) {$1_0$};
\node[main node] at ({cos(\z+1*\x+0*\y)}, {sin(\z+1*\x+0*\y)})  (20) {$2_0$};
\node[main node] at ({cos(\z+2*\x+0*\y)}, {sin(\z+2*\x+0*\y)})  (30) {$3_0$};
\node[main node] at ({cos(\z+3*\x+0*\y)}, {sin(\z+3*\x+0*\y)})  (40) {$4_0$};
\node[main node] at ({cos(\z+4*\x+0*\y)}, {sin(\z+4*\x+0*\y)})  (50) {$5_0$};
\node[main node] at ({\w*cos(\z+0*\x+1*\y)}, {\w*sin(\z+0*\x+1*\y)})  (11) {$1_1$};
\node[main node] at ({\w*cos(\z+1*\x+1*\y)}, {\w*sin(\z+1*\x+1*\y)})  (21) {$2_1$};
\node[main node] at ({\w*cos(\z+2*\x+1*\y)}, {\w*sin(\z+2*\x+1*\y)})  (31) {$3_1$};
\node[main node] at ({\w*cos(\z+3*\x+1*\y)}, {\w*sin(\z+3*\x+1*\y)})  (41) {$4_1$};
\node[main node] at ({\w*cos(\z+4*\x+1*\y)}, {\w*sin(\z+4*\x+1*\y)})  (51) {$5_1$};
\node[main node] at ({\w*cos(\z+0*\x+2*\y)}, {\w*sin(\z+0*\x+2*\y)})  (12) {$1_2$};
\node[main node] at ({\w*cos(\z+1*\x+2*\y)}, {\w*sin(\z+1*\x+2*\y)})  (22) {$2_2$};
\node[main node] at ({\w*cos(\z+2*\x+2*\y)}, {\w*sin(\z+2*\x+2*\y)})  (32) {$3_2$};
\node[main node] at ({\w*cos(\z+3*\x+2*\y)}, {\w*sin(\z+3*\x+2*\y)})  (42) {$4_2$};
\node[main node] at ({\w*cos(\z+4*\x+2*\y)}, {\w*sin(\z+4*\x+2*\y)})  (52) {$5_2$};

 \path[every node/.style={font=\sffamily}]
(10) edge (11)
(20) edge (21)
(30) edge (31)
(40) edge (41)
(50) edge (51)
(11) edge (12)
(21) edge (22)
(31) edge (32)
(41) edge (42)
(51) edge (52)
(10) edge (22)
(10) edge (32)
(20) edge (32)
(20) edge (42)
(30) edge (42)
(30) edge (52)
(40) edge (52)
(40) edge (12)
(50) edge (12)
(50) edge (22);
\end{tikzpicture}
&

\def\x{360/7}
\def\z{270-1*\y}

\begin{tikzpicture}
[scale=\sc, thick,main node/.style={circle, minimum size=3.8mm, inner sep=0.1mm,draw,font=\tiny\sffamily}]

\node[main node] at ({cos(\z+0*\x+0*\y)}, {sin(\z+0*\x+0*\y)})  (10) {$1_0$};
\node[main node] at ({cos(\z+1*\x+0*\y)}, {sin(\z+1*\x+0*\y)})  (20) {$2_0$};
\node[main node] at ({cos(\z+2*\x+0*\y)}, {sin(\z+2*\x+0*\y)})  (30) {$3_0$};
\node[main node] at ({cos(\z+3*\x+0*\y)}, {sin(\z+3*\x+0*\y)})  (40) {$4_0$};
\node[main node] at ({cos(\z+4*\x+0*\y)}, {sin(\z+4*\x+0*\y)})  (50) {$5_0$};
\node[main node] at ({cos(\z+5*\x+0*\y)}, {sin(\z+5*\x+0*\y)})  (60) {$6_0$};
\node[main node] at ({cos(\z+6*\x+0*\y)}, {sin(\z+6*\x+0*\y)})  (70) {$7_0$};

\node[main node] at ({\w*cos(\z+0*\x+1*\y)}, {\w*sin(\z+0*\x+1*\y)})  (11) {$1_1$};
\node[main node] at ({\w*cos(\z+1*\x+1*\y)}, {\w*sin(\z+1*\x+1*\y)})  (21) {$2_1$};
\node[main node] at ({\w*cos(\z+2*\x+1*\y)}, {\w*sin(\z+2*\x+1*\y)})  (31) {$3_1$};
\node[main node] at ({\w*cos(\z+3*\x+1*\y)}, {\w*sin(\z+3*\x+1*\y)})  (41) {$4_1$};
\node[main node] at ({\w*cos(\z+4*\x+1*\y)}, {\w*sin(\z+4*\x+1*\y)})  (51) {$5_1$};
\node[main node] at ({\w*cos(\z+5*\x+1*\y)}, {\w*sin(\z+5*\x+1*\y)})  (61) {$6_1$};
\node[main node] at ({\w*cos(\z+6*\x+1*\y)}, {\w*sin(\z+6*\x+1*\y)})  (71) {$7_1$};

\node[main node] at ({\w*cos(\z+0*\x+2*\y)}, {\w*sin(\z+0*\x+2*\y)})  (12) {$1_2$};
\node[main node] at ({\w*cos(\z+1*\x+2*\y)}, {\w*sin(\z+1*\x+2*\y)})  (22) {$2_2$};
\node[main node] at ({\w*cos(\z+2*\x+2*\y)}, {\w*sin(\z+2*\x+2*\y)})  (32) {$3_2$};
\node[main node] at ({\w*cos(\z+3*\x+2*\y)}, {\w*sin(\z+3*\x+2*\y)})  (42) {$4_2$};
\node[main node] at ({\w*cos(\z+4*\x+2*\y)}, {\w*sin(\z+4*\x+2*\y)})  (52) {$5_2$};
\node[main node] at ({\w*cos(\z+5*\x+2*\y)}, {\w*sin(\z+5*\x+2*\y)})  (62) {$6_2$};
\node[main node] at ({\w*cos(\z+6*\x+2*\y)}, {\w*sin(\z+6*\x+2*\y)})  (72) {$7_2$};

 \path[every node/.style={font=\sffamily}]
(10) edge (11)
(20) edge (21)
(30) edge (31)
(40) edge (41)
(50) edge (51)
(11) edge (12)
(21) edge (22)
(31) edge (32)
(41) edge (42)
(51) edge (52)
(60) edge (61)
(61) edge (62)
(70) edge (71)
(71) edge (72)
(10) edge (22)
(10) edge (32)
(10) edge (42)
(20) edge (32)
(20) edge (42)
(20) edge (52)
(30) edge (42)
(30) edge (52)
(30) edge (62)
(40) edge (52)
(40) edge (62)
(40) edge (72)
(50) edge (62)
(50) edge (72)
(50) edge (12)
(60) edge (72)
(60) edge (12)
(60) edge (22)
(70) edge (12)
(70) edge (22)
(70) edge (32);
\end{tikzpicture}
\\
$\B_3'$ & $\B_5'$ & $\B_7'$
\end{tabular}
\caption{Illustrating the graphs $\B_k'$ for the proof of Proposition~\ref{prop508} }\label{figBk'}
\end{center}
\end{figure}

It remains to show that $\B_k$ is vertex-transitive, and we do so via describing three automorphisms of $\B_k$. For every $i \in [k]$, define the functions $f_1, f_2, f_3 : V(\B_k) \to V(\B_k)$ as follows:
\[
\renewcommand*{\arraystretch}{1.1}
\begin{array}{c|cccc}
j & 0 & 1 & 2 & 3 \\
\hline
f_1(i_j) & i_1 & i_0 & i_3 & i_2 \\
f_2(i_j) & (i +_k 1)_0 & (i +_k 1)_1 & (i +_k 1)_2 & (i +_k 1)_3 \\
f_3(i_j) & (2 -_k i)_3 & (2 -_k i)_2 & (2 -_k i)_1 & (2 -_k i)_0 
\end{array}
\]
In all cases, one can show that if $\set{u,v}$ is an edge in $\B_k$, then so are $\set{f_1(u),f_1(v)}$, $\set{f_2(u),f_2(v)}$, and $\set{f_3(u),f_3(v)}$. Furthermore, for every distinct $u,v \in V(\B_k)$, there exists a composition of $f_1$, $f_2$, and $f_3$ which maps $u$ to $v$. Hence, we conclude that $\B_k$ is vertex-transitive.
\end{proof}

Proposition~\ref{prop508} readily implies the following:

\begin{theorem}\label{thm509}
For every positive integer $\ell$, there exists a vertex-transitive graph $G$ where $|V(G)| \leq 4\ell+12$ and $r_+(G) \geq \ell$.
\end{theorem}

\begin{proof}
If $\ell$ is odd, then $\B_{\ell+2}$, which has $4\ell+8$ vertices, satisfies the desired conditions. If $\ell$ is even, then $\B_{\ell+3}$, which has $4\ell+12$ vertices, satisfies $r_+(G) \geq \ell$.
\end{proof}

\subsection{$\CG$-rank of stretched cliques}

Next, we comment on the hardness of $\STAB(G)$ for some stretched cliques $G$ with respect to another well-studied cutting-plane procedure, which is due to Chv{\'a}tal~\cite{Chvatal73} and Gomory~\cite{Gomory58}. Given a set $P \subseteq [0,1]^n$ and a valid inequality $a^{\top}x \leq \b$ for $P$, where $a \in \mZ^n$, we say that $a^{\top}x \leq \lfloor \b \rfloor$ is a \emph{Chv{\'a}tal--Gomory cut} of $P$. Observe that every Chv{\'a}tal--Gomory cut of $P$ is valid for $P_I$. Thus, if we define $\CG(P)$ to be the set of points which satisfy all Chv{\'a}tal--Gomory cuts for $P$, then we have $P_I \subseteq \CG(P) \subseteq P$. The set $\CG(P)$ is known as the \emph{Chv{\'a}tal--Gomory closure} of $P$, and is a closed convex set for every $P$.

As with $\LS_+$, we can also apply this cutting-plane procedure iteratively. Given an integer $\ell \geq 2$, we recursively define $\CG^{\ell}(P) \ce \CG\left( \CG^{\ell-1}(P) \right)$. We can then define the \emph{$\CG$-rank} of a valid inequality of $P_I$ (relative to $P$) to be the smallest integer $\ell$ for which it is valid for $\CG^{\ell}(P)$, and let the $\CG$-rank of a set $P$ to be the smallest integer $\ell$ where $\CG^{\ell}(P) = P_I$. For convenience and for consistency with our discussion on $\LS_+$, given a graph $G$, we will also write $\CG^{\ell}(G)$ instead of $\CG^{\ell}(\FRAC(G))$, and refer to the $\CG$-rank of $\FRAC(G)$ simply as the $\CG$-rank of $G$. A notable distinction between the procedures $\LS_+$ and $\CG$ is that optimizing a linear function over $\CG^{\ell}(P)$ is $\mathcal{N}\mathcal{P}$-hard in general, even for $\ell = O(1)$.

After establishing that stretched cliques can be the worst-case instances for $\LS_+$, we now show that there are families of stretched cliques with unbounded $\CG$-rank. To do so, let us consider a special family of graphs. Given an integer $k \geq 3$, we define the graph $H_k'$ where
\begin{align*}
V(H_k') \ce{}& \set{1,2} \cup  \set{ i_0, i_1, i_2 : i \in \set{3,\ldots,k}}\\
E(H_k') \ce{}& \set{\set{i_0, i_1}, \set{i_0, i_2},  \set{i_1, 2}, \set{i_2,1} : i \in \set{3,\ldots,k}} \cup \\
& \set{ \set{i_1,j_2}  : i,j \in \set{3,\ldots,k}, i \neq j}.
\end{align*}

\def\y{0.70}
\def\sc{2}
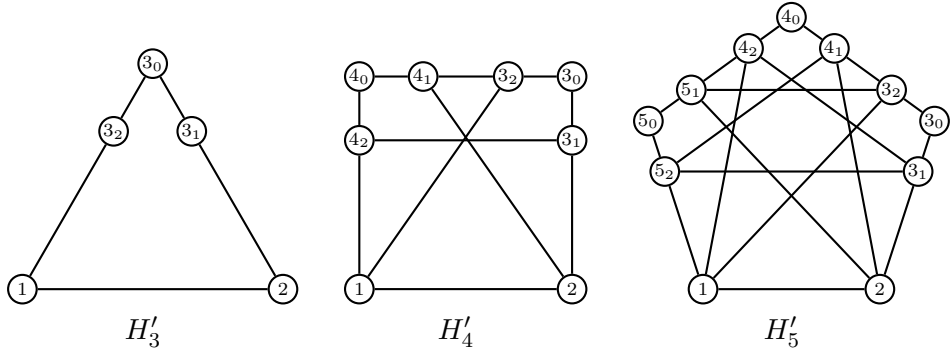
\begin{figure}[ht!]
\begin{center}
\begin{tabular}{ccc}

\def\x{270 - 180/3}
\def\z{360/3}
\begin{tikzpicture}[scale=\sc, thick,main node/.style={circle, minimum size=3.8mm, inner sep=0.1mm,draw,font=\tiny\sffamily}]
\node[main node] at ({cos(\x+(0)*\z)},{sin(\x+(0)*\z)}) (1) {$1$};
\node[main node] at ({cos(\x+(1)*\z)},{sin(\x+(1)*\z)}) (6) {$2$};

\node[main node] at ({ \y* cos(\x+(2)*\z) + (1-\y)*cos(\x+(1)*\z)},{ \y* sin(\x+(2)*\z) + (1-\y)*sin(\x+(1)*\z)}) (7) {$3_1$};
\node[main node] at ({cos(\x+(2)*\z)},{sin(\x+(2)*\z)}) (8) {$3_0$};
\node[main node] at ({ \y* cos(\x+(2)*\z) + (1-\y)*cos(\x+(3)*\z)},{ \y* sin(\x+(2)*\z) + (1-\y)*sin(\x+(3)*\z)}) (9) {$3_2$};

 \path[every node/.style={font=\sffamily}]
(8) edge (7)
(8) edge (9)
(1) edge (6)
(1) edge (9)
(7) edge (6);
\end{tikzpicture}

&

\def\x{270 - 180/4}
\def\z{360/4}

\begin{tikzpicture}[scale=\sc, thick,main node/.style={circle, minimum size=3.8mm, inner sep=0.1mm,draw,font=\tiny\sffamily}]

\node[main node] at ({cos(\x+(0)*\z)},{sin(\x+(0)*\z)}) (1) {$1$};
\node[main node] at ({cos(\x+(1)*\z)},{sin(\x+(1)*\z)}) (6) {$2$};

\node[main node] at ({ \y* cos(\x+(2)*\z) + (1-\y)*cos(\x+(1)*\z)},{ \y* sin(\x+(2)*\z) + (1-\y)*sin(\x+(1)*\z)}) (7) {$3_1$};
\node[main node] at ({cos(\x+(2)*\z)},{sin(\x+(2)*\z)}) (8) {$3_0$};
\node[main node] at ({ \y* cos(\x+(2)*\z) + (1-\y)*cos(\x+(3)*\z)},{ \y* sin(\x+(2)*\z) + (1-\y)*sin(\x+(3)*\z)}) (9) {$3_2$};

\node[main node] at ({ \y* cos(\x+(3)*\z) + (1-\y)*cos(\x+(2)*\z)},{ \y* sin(\x+(3)*\z) + (1-\y)*sin(\x+(2)*\z)}) (10) {$4_1$};
\node[main node] at ({cos(\x+(3)*\z)},{sin(\x+(3)*\z)}) (11) {$4_0$};
\node[main node] at ({ \y* cos(\x+(3)*\z) + (1-\y)*cos(\x+(4)*\z)},{ \y* sin(\x+(3)*\z) + (1-\y)*sin(\x+(4)*\z)}) (12) {$4_2$};

 \path[every node/.style={font=\sffamily}]
(8) edge (7)
(8) edge (9)
(11) edge (10)
(11) edge (12)
(1) edge (6)
(1) edge (9)
(1) edge (12)
(7) edge (6)
(7) edge (12)
(10) edge (6)
(10) edge (9);
\end{tikzpicture}

&

\def\x{270 - 180/5}
\def\z{360/5}

\begin{tikzpicture}[scale=\sc, thick,main node/.style={circle, minimum size=3.8mm, inner sep=0.1mm,draw,font=\tiny\sffamily}]
\node[main node] at ({cos(\x+(0)*\z)},{sin(\x+(0)*\z)}) (1) {$1$};
\node[main node] at ({cos(\x+(1)*\z)},{sin(\x+(1)*\z)}) (6) {$2$};

\node[main node] at ({ \y* cos(\x+(2)*\z) + (1-\y)*cos(\x+(1)*\z)},{ \y* sin(\x+(2)*\z) + (1-\y)*sin(\x+(1)*\z)}) (7) {$3_1$};
\node[main node] at ({cos(\x+(2)*\z)},{sin(\x+(2)*\z)}) (8) {$3_0$};
\node[main node] at ({ \y* cos(\x+(2)*\z) + (1-\y)*cos(\x+(3)*\z)},{ \y* sin(\x+(2)*\z) + (1-\y)*sin(\x+(3)*\z)}) (9) {$3_2$};

\node[main node] at ({ \y* cos(\x+(3)*\z) + (1-\y)*cos(\x+(2)*\z)},{ \y* sin(\x+(3)*\z) + (1-\y)*sin(\x+(2)*\z)}) (10) {$4_1$};
\node[main node] at ({cos(\x+(3)*\z)},{sin(\x+(3)*\z)}) (11) {$4_0$};
\node[main node] at ({ \y* cos(\x+(3)*\z) + (1-\y)*cos(\x+(4)*\z)},{ \y* sin(\x+(3)*\z) + (1-\y)*sin(\x+(4)*\z)}) (12) {$4_2$};

\node[main node] at ({ \y* cos(\x+(4)*\z) + (1-\y)*cos(\x+(3)*\z)},{ \y* sin(\x+(4)*\z) + (1-\y)*sin(\x+(3)*\z)}) (13) {$5_1$};
\node[main node] at ({cos(\x+(4)*\z)},{sin(\x+(4)*\z)}) (14) {$5_0$};
\node[main node] at ({ \y* cos(\x+(4)*\z) + (1-\y)*cos(\x+(5)*\z)},{ \y* sin(\x+(4)*\z) + (1-\y)*sin(\x+(5)*\z)}) (15) {$5_2$};

 \path[every node/.style={font=\sffamily}]
(8) edge (7)
(8) edge (9)
(11) edge (10)
(11) edge (12)
(14) edge (13)
(14) edge (15)
(1) edge (6)
(1) edge (9)
(1) edge (12)
(1) edge (15)
(7) edge (6)
(7) edge (12)
(7) edge (15)
(10) edge (6)
(10) edge (9)
(10) edge (15)
(13) edge (6)
(13) edge (9)
(13) edge (12);
\end{tikzpicture}
\\
$H_3'$ & $H_4'$ & $H_5'$ 
\end{tabular}
\caption{Several graphs in the family $H_k'$}\label{figH_k'}
\end{center}
\end{figure}

Figure~\ref{figH_k'} gives the drawings of $H_k'$ for $k \in \set{3,4,5}$. The authors recently studied the $\LS_+$-relaxations of $H_k'$ in~\cite{AuT24b} (where the graphs had slightly different vertex labels). These graphs are also closely related to the graphs $H_k$, which form the first known family of graphs $G$ for which $r_+(G)$ is asymptotically a linear function of $|V(G)|$~\cite{AuT24}.

Next, observe that $H_k' \in \tilde{\K}_{k, k-2}$ for every $k \geq 3$, and so it follows from Lemma~\ref{lem304} that $\bar{e}^{\top}x \leq k-1$ is a facet-inducing inequality for $\STAB(H_k')$. The following result is a consequence of~\cite[Proposition 29]{AuT24b} and~\cite[Theorem 29]{AuT24}.

\begin{proposition}\label{prop510}
For every $k \geq 3$, the facet-inducing inequality $\bar{e}^{\top}x \leq k-1$ for $\STAB(H_k')$ has $\LS_+$-rank at least $\frac{3k}{16}$, and has  $\CG$-rank at least $\log_4 \left( \frac{3k-7}{2} \right)$.
\end{proposition}

Next, we use the $\CG$-rank bound on $H_k'$ above to prove a $\CG$-rank lower bound on some stretched cliques.

\begin{proposition}\label{prop511}
Let $G \in \K_{n,d}$ where $n \geq d+2$ and $d \geq 1$. Furthermore, suppose that $\set{i_1, j_1}, \set{i_2, j_2} \not\in E(G)$ for all distinct $i,j \in D(G)$. Then the valid inequality $\bar{e}^{\top}x \leq d+1$ of $\STAB(G)$ has $\CG$-rank at least $\log_4 \left( \frac{3d-1}{2} \right)$.
\end{proposition}

\begin{proof}
We first prove the claim for the case when $n = d+2$. Since $G$ does not contain the edges $\set{i_1, j_1}$ and $\set{i_2, j_2}$ for every distinct $i,j \in D(G)$, it follows that $E(G) \subseteq E(H_{d+2}')$. Hence, $\FRAC(G)$ is defined by a subset of the inequalities defining $\FRAC(H_{d+2}')$, and as a result $\CG^{\ell}(G)$ is defined by a subset of the inequalities defining $\CG^{\ell}(H_{d+2}')$ for every positive integer $\ell$. This implies that $\CG^{\ell}(H_{d+2}') \subseteq \CG^{\ell}(G)$ for every $\ell \in \mN$, and thus the $\CG$-rank bound of the inequality $\bar{e}^{\top}x \leq d+1$ for $H_{d+2}'$ in Proposition~\ref{prop510} applies for $G$.

Next, suppose $n \geq d+3$. Then, we can delete all but two unstretched vertices from $G$ to obtain a subgraph $G' \in \K_{d+2, d}$ where $E(G') \subseteq E(H_{d+2}')$. Now the argument from the preceding paragraph applies, and we obtain that the $\CG$-rank of the inequality $\sum_{i \in V(G')} x_i \leq d+1$ is at least that of $H_{d+2}'$. Furthermore, since $\FRAC(G')$ is a projection of $\FRAC(G)$, if $\sum_{i \in V(G')} x_i \leq d+1$ is not valid for $\CG^{\ell}(G')$ for a given $\ell \in \mN$, then $\sum_{i \in V(G)} x_i \leq d+1$ cannot be valid for $\CG^{\ell}(G)$. Hence, the $\CG$-rank of $\sum_{i \in V(G)} x_i \leq d+1$ for $\STAB(G)$ is at least that of $\sum_{i \in V(G')} x_i \leq d+1$ for $\STAB(G')$. Thus, our claim follows in this case as well.
\end{proof}

Proposition~\ref{prop511} shows that one can construct families of stretched cliques with arbitrarily high $\CG$-rank. One such family is the graphs $\A_{k,S}$ (for any choice of $S$). Therefore, we see that the stable set polytopes of these graphs are not only challenging instances for $\LS_+$, but are also computationally costly for $\CG$.

\section{Future research directions}\label{sec05}

We conclude the manuscript by mentioning a few natural questions raised by our work herein.

\begin{problem}
Obtain a combinatorial characterization of all $\ell$-minimal graphs.
\end{problem}

This problem is related to Conjecture 40 of \cite{LiptakT03}. The conjecture has two parts. The first part is the existence of $\ell$-minimal graphs for every positive integer $\ell$. The second part of the conjecture stated ``Moreover, the equality is attained by a subdivision of the clique $K_{\ell +2}$.'' (We used $\ell$ in place of $k$ used in \cite{LiptakT03}.) Escalante, Montelar and Nasini \cite{EscalanteMN06} proved that if the word ``subdivision'' is interpreted as ``only replacing edges with paths of length at least one'' that second part of the conjecture is true for $\ell =3$, but false for all $\ell \geq 4$. In the paper \cite{LiptakT03}, other, more general subdivision operations were discussed (including certain versions of \emph{stretching}). Note that here, we proved that with this interpretation of subdivision of a clique (which includes stretching), the second part of Conjecture 40 of \cite{LiptakT03} also holds.
 
Given a positive integer $\ell$, we showed (Corollary~\ref{cor501}) that $G \in \hat{\K}_{\ell+2, \ell-1}$ with $\omega(G) \leq 3$ is sufficient for a given graph $G$ to be $\ell$-minimal. However, as shown in Proposition~\ref{prop506} and suggested by the numerical evidence presented in Figure~\ref{figK52}, being in $\hat{\K}_{\ell+2, \ell-1}$ is not a necessary condition for $\ell$-minimal graphs. In fact, since the completion of this manuscript, subsequent computational work~\cite{AuT26} has produced many additional examples: in particular, it shows that there are at least 18 $3$-minimal graphs which do not belong to $\K_{5,2}$, as well as many more $4$-minimal graphs which do not belong to $\K_{6,3}$, including at least one which does not contain $K_6$ as a graph minor.

Thus, there are still much about $\ell$-minimal graphs that we have yet to understand. What are some other interesting properties of these graphs?  More ambitiously, can we obtain a combinatorial characterization of exactly when a given graph is $\ell$-minimal?

\begin{problem}
Let $\bar{n}_+^{\geq}(\ell)$ be the smallest possible number of vertices needed for a vertex-transitive graph $G$ to have $r_+(G) \geq \ell$. What is $\displaystyle\lim_{\ell \to \infty} \frac{\bar{n}_+^{\geq}(\ell)}{\ell}$?
\end{problem}

It follows immediately from Theorem~\ref{thm509} that $\bar{n}_+^{\geq}(\ell) \leq 4\ell+12$, and so $\displaystyle\lim_{\ell \to \infty} \frac{\bar{n}_+^{\geq}(\ell)}{\ell} \leq 4$. On the other hand, it is obvious that $\bar{n}_+^{\geq}(\ell) \geq n_+(\ell) = 3\ell$ for all $\ell \geq 1$, and so $\displaystyle\lim_{\ell \to \infty} \frac{\bar{n}_+^{\geq}(\ell)}{\ell} \geq 3$. Can we find out what the true value of the limit (or even a closed-form formula for $\bar{n}_+^{\geq}(\ell)$), or at least prove tighter bounds?

(As an aside, we remark that the problem could be rather different if we instead consider $\bar{n}^{=}_+(\ell)$, which is defined to be the smallest possible number of vertices needed for a vertex-transitive graph $G$ to have $r_+(G)$ \emph{equal to} $\ell$. In this case, since the line graph of every odd clique is vertex-transitive, we know that $\bar{n}^{=}_+(\ell) \leq 2\ell^2+\ell$~\cite{StephenT99}. Moreover, subsequent computational work~\cite{AuT26} determined the smallest vertex-transitive graphs with $\LS_+$-rank $\ell$ for every $\ell \leq 4$; in particular,
\[
\bar{n}^{=}_+(1)=3,\qquad
\bar{n}^{=}_+(2)=8,\qquad
\bar{n}^{=}_+(3)=13,\qquad
\bar{n}^{=}_+(4)=16.
\]
Despite the pattern of the initial values, it is not immediately clear to us that $\bar{n}^{=}_+(\ell)$ must be an increasing function of $\ell$ in general. Thus, there is a chance that the limit $\displaystyle\lim_{\ell \to \infty} \frac{\bar{n}^{=}_+(\ell)}{\ell}$ may not exist.)

\begin{problem}
For each pair of positive integers $(n, \ell)$ with $n \geq \ell$, characterize the family of graphs $G$ on $n$ vertices which maximize the integrality ratio: 
\[
\frac{\alpha_{\LS_+^{\ell}}(G)}{\alpha(G)},
\]
where $\alpha_{\LS_+^{\ell}}(G) \ce \max\left\{\bar{e}^{\top}x \, : \, x \in \LS_+^{\ell}(G)\right\}.$  
\end{problem}

In this manuscript, we showed that $\ell$-minimal graphs exist for every $\ell \geq 1$, establishing graphs which are worst-case scenarios for $\LS_+$ in the sense of needing the maximum possible number of iterations of $\LS_+$ to ``compute'' the stable set polytope. In addition to the $\LS_+$-rank, another measure of the hardness of a graph is the \emph{integrality gap} for the relaxation $\LS_+(G)$. Progress in this direction would provide new understanding about the $\LS_+$-relaxations of the stable set polytope of graphs from a different angle. Note that for random graphs $G_{n,1/2}$ we understand such integrality ratios to some extent. $\alpha(G_{n,1/2})$ is almost surely around $2\log_2(n)$. Feige and Krauthgamer~\cite{FeigeK03} showed (in providing an answer to the other question of Knuth about $\LS_+$ in \cite{Knuth1994}) that $\alpha_{\LS_+^{\ell}}(G)$ is almost surely around $\sqrt{n/2^{\ell}}$ for $\ell = o(\log(n))$.

\begin{problem}
Are there other applications for Lemma~\ref{lem402} and the ideas used in its proof?
\end{problem}

The definition of $\LS_+^k$ naturally lends itself to inductive arguments when it comes to establishing rank lower bounds for a family of instances. Previous examples of this type of argument include the aforementioned result by Stephen and the second author on the line graphs of odd cliques~\cite{StephenT99}, as well as for the family of graphs $H_k$ in~\cite{AuT24}. For our main result in this manuscript, a key insight was to build our proof around certifying the membership of the vector $v(G,\epsilon)$, which behaves well under deletion and destruction of vertices, even when the underlying graphs in $\K_{n,d}$ do not exhibit nearly as much symmetry as the two previous families of examples.

In particular, the foundation of our argument is Lemma~\ref{lem402}, a noteworthy feature of which is that it allows us to establish $\LS_+$-rank lower bounds without having to construct and verify specific numerical certificates. We  intentionally stated this lemma as a result for $\LS_+$-relaxations in general, and it would be interesting to see if this result and its insights can lead to breakthroughs in the analysis of other convex relaxations.

\bibliographystyle{alpha}
\bibliography{ref} 

\end{document}